\documentclass[11pt]{article} 

\RequirePackage[OT1]{fontenc}
\RequirePackage{amsthm,amsmath,amsfonts,amssymb,mathtools}
\RequirePackage[colorlinks, citecolor=blue, linkcolor=red]{hyperref}
\RequirePackage{epsfig, graphicx, color}
\RequirePackage{times}
\RequirePackage{latexsym}
\RequirePackage{psfrag}
\RequirePackage{color}
\usepackage{natbib}
\usepackage{xcolor}
\usepackage{bm, bbm}

\usepackage{fullpage}
\usepackage{epsf,epsfig,subfigure}
\usepackage{epsf,epsfig}
\usepackage{fancyhdr}
\usepackage{graphics,graphicx}
\usepackage{enumerate}
\usepackage{enumitem}
\usepackage{algorithmicx,algorithm}
\usepackage{bbm}
\usepackage{multirow}
\usepackage{booktabs}
\usepackage{longtable}
\usepackage{subfigure}

\usepackage{etoolbox}

\RequirePackage[textsize=tiny, textwidth = 2cm, shadow]{todonotes}
\RequirePackage[normalem]{ulem}

\newtheorem{theorem}{Theorem}

\newtheorem{lemma}[theorem]{Lemma}
\newtheorem{condition}{Condition}
\newtheorem{corollary}[theorem]{Corollary}

\theoremstyle{definition}
\newtheorem{definition}[theorem]{Definition}
\newtheorem{example}[theorem]{Example}

\theoremstyle{remark}
\newtheorem{remark}{Remark}
\AtBeginEnvironment{remark}{%
  \pushQED{\qed}%
}
\AtEndEnvironment{remark}{\popQED\endexample}

\newtheorem{innercustomthm}{}
\newenvironment{customthm}[1]
  {\renewcommand\theinnercustomthm{#1}\innercustomthm}
  {\endinnercustomthm}
\newtheorem{innercustomgeneric}{\customgenericname}
\providecommand{\customgenericname}{}

\newcommand{\R}{\mathbb{R}}

\newcommand{\E}{\mathbb{E}}
\newcommand{\argmin}{{\arg\hspace*{-0.5mm}\min}}
\newcommand{\I}{\mathrm{I}}

\newcommand{\sign}{\textrm{sign}}
\newcommand{\one}{\mathbbm{1}}

\newcommand{\Lin}[1]{{ \color{red}{Lin: #1}}}

\def\st{\mathrm{s.t.}}

\def\pr{\mathbb{P}}
\def\var{\mathrm{var}}
\def\est{\mathrm{DCal}}
\def\ora{\texttt{ora}}
\def\hpi{\hat{\pi}}
\def\tpi{\tilde{\pi}}
\def\r{\hat{r}}
\def\hmu{\hat{\mu}}
\def\htau{\hat{\tau}}

\def\bX{\mathbf{X}}
\def\bmu{\boldsymbol{\mu}}

\def\calN{\mathcal{N}}

\def\cB{\mathcal{B}}

\def\cD{\mathcal{D}}

\def\cA{\mathcal{A}}
\def\cB{\mathcal{B}}
\def\argmin{{\arg\hspace*{-0.5mm}\min}}

\def\bbP{\mathbb{P}}
\def\bbE{\mathbb{E}}
\def\dc{\textrm{DCal}}
\def\sc{\textrm{SCal}}
\def\DML{\mathrm{DML}}
\def\IPW{\mathrm{IPW}}
\def\ARB{\mathrm{ARB}}
\def\RCAL{\mathrm{RCAL}}
\def\TMLE{\mathrm{TMLE}}
\def\calL{\mathcal{L}}

\def\dipw{{\sc, r}}

\def\bcdot{\boldsymbol{\cdot}}
\def\cbps{\mathrm{hdCBPS}}

\def\IF{\mathsf{IF}}
\def\tr{\mathrm{tr}}

\newcommand{\mb}{\mathbf}
\newcommand{\mbb}{\boldsymbol}

\newcommand\independent{\protect\mathpalette{\protect\independenT}{\perp}}
\def\independenT#1#2{\mathrel{\rlap{$#1#2$}\mkern2mu{#1#2}}}

\newcommand{\vertiii}[1]{{\left\vert\kern-0.25ex\left\vert\kern-0.25ex\left\vert #1 
    \right\vert\kern-0.25ex\right\vert\kern-0.25ex\right\vert}}
    
\def\Holder{H\"{o}lder}
\def\aux{\mathrm{aux}}

\makeatletter
\def\thickhline{%
  \noalign{\ifnum0=`}\fi\hrule \@height \thickarrayrulewidth \futurelet
   \reserved@a\@xthickhline}
\def\@xthickhline{\ifx\reserved@a\thickhline
               \vskip\doublerulesep
               \vskip-\thickarrayrulewidth
             \fi
      \ifnum0=`{\fi}}
\makeatother

\newlength{\thickarrayrulewidth}
\title{Root-n consistent semiparametric learning with high-dimensional nuisance parameters under minimal sparsity}

\author{Lin Liu$^1$\thanks{email: \href{linliu@sjtu.edu.cn}{linliu@sjtu.edu.cn}}, \ \ Xinbo Wang$^2$\thanks{email: \href{cinbo_w@sjtu.edu.cn}{cinbo\_w@sjtu.edu.cn}}, \ \ and Yuhao Wang$^{3, 4}$
	\thanks{email: \href{yuhaow@tsinghua.edu.cn}{yuhaow@tsinghua.edu.cn}. 
		 The authors contributed equally to this work, names are in alphabetical order. Correspondence should be addressed to YW: \href{yuhaow@tsinghua.edu.cn}{yuhaow@tsinghua.edu.cn}.
	}
}

\date{
	$^1$Institute of Natural Sciences, MOE-LSC, School of Mathematical Sciences, CMA-Shanghai, SJTU-Yale Joint Center for Biostatistics and Data Science, Shanghai Jiao Tong University \\
	$^2$Department of Bioinformatics and Biostatistics, School of Life Sciences, SJTU-Yale Joint Center for Biostatistics and Data Science, Shanghai Jiao Tong University \\
	$^3$Institute for Interdisciplinary Information Sciences, Tsinghua University \\
        $^4$Shanghai Artificial Intelligence Laboratory and Shanghai Qi Zhi Institute \\
	\vspace{1em}
	\today
}

\begin{document}

\maketitle

\begin{abstract}
Treatment effect estimation under unconfoundedness is a fundamental task in causal inference. In response to the challenge of analyzing high-dimensional datasets collected in substantive fields such as epidemiology, genetics, economics, and social sciences, various methods for treatment effect estimation with high-dimensional nuisance parameters (the outcome regression and the propensity score) have been developed in recent years. However, it is still unclear what is the necessary and sufficient sparsity condition on the nuisance parameters such that we can estimate the treatment effect at $1 / \sqrt{n}$-rate. In this paper, we propose a new Double-Calibration strategy that corrects the estimation bias of the nuisance parameter estimates computed by regularized high-dimensional techniques and demonstrate that the corresponding Doubly-Calibrated estimator achieves $1 / \sqrt{n}$-rate as long as one of the nuisance parameters is sparse with sparsity below $\sqrt{n} / \log p$, where $p$ denotes the ambient dimension of the covariates, whereas the other nuisance parameter can be arbitrarily complex and completely misspecified. The Double-Calibration strategy can also be applied to settings other than treatment effect estimation, e.g. regression coefficient estimation in the presence of a diverging number of controls in a semiparametric partially linear model, and local average treatment effect estimation with instrumental variables.
\end{abstract}


\section{Introduction}
\label{sec:intro}

This article concerns the problem of efficient estimation of a parameter of scientific interest, denoted as $\tau \equiv \tau (\theta)$ under the semiparametric framework, 
where $\theta$ denotes the (potentially) high-dimensional nuisance parameters. A typical example of such parameters of interest is $\tau \equiv \E [Y (t)]$, which denotes the mean of a potential outcome $Y (t)$ with the binary treatment $T$ set to $t \in \{0, 1\}$ and is identifiable from the observed data under no unmeasured confounding. Since the average treatment effect (ATE) is just the contrast $\E [Y (1)] - \E [Y (0)]$, to simplify our exposition, we refer to $\tau = \E [Y (t)]$ as the ATE (with a slight abuse of terminology) and take $t = 1$ throughout. To estimate ATE under the no unmeasured confounding assumption, we need to account for the nuisance parameters $\theta$, which can be decomposed into two components $\theta \equiv (\pi, r)$: the Propensity Score (PS), denoted as $\pi (x) \coloneqq \E [T | X = x]$, and the Outcome Regression (OR) model, denoted as $r (x) \coloneqq \E [Y | X = x, T = 1]$. Here $X$ denotes the $p$-dimensional covariates. Although parameters other than ATE will also be covered in this paper, we mainly focus on ATE in the Introduction.

It is now well-established in modern semiparametric theory that the (nonparametric first-order) influence function \citep{robins1994estimation, hahn1998role} of $\tau$ is of the form
\begin{equation}
\label{if}
\IF (\theta) \equiv H (\theta) - \tau, \text{ where } H (\theta) = \frac{T}{\pi (X)} (Y - r (X)) + r (X),
\end{equation}
which gives rise to the celebrated Augmented Inverse Probability Weighting (AIPW) estimator. The AIPW estimator is Doubly-Robust (DR) \citep{scharfstein1999rejoinder, robins2001comments} and is $\sqrt{n}$-consistent and asymptotically normal when $\pi$ and $r$ are estimated at sufficiently fast rates. To establish these convergence rates, it is essential to impose certain complexity-reducing assumptions \citep{robins1997toward, ritov2014bayesian, liu2020nearly, liu2024assumption} on the nuisance parameters $\pi$ and $r$. 

During the past decade, in response to the challenge of high-dimensional datasets with (ambient) dimension $p$ potentially much greater than the sample size $N$ ($p \gg N$), a large body of literature in statistics and econometrics (\citet{farrell2015robust, shortreed2017outcome, van2017generally, chernozhukov2018double, ju2020robust, chernozhukov2022debiased, avagyan2021high, ning2020robust, tan2020regularized, tan2020model, tang2023ultra, athey2018approximate, hirshberg2021augmented, bradic2019sparsity, bradic2019minimax, wang2024debiased, dukes2021inference, sun2022high} and references therein) has been devoted to solving this problem by positing $\pi$ and/or $r$ to be sparse Generalized Linear Models (GLMs): $\pi (x) \equiv \phi (\gamma^{\top} x)$, $r (x) \equiv \psi (\beta^{\top} x)$ where $\gamma$ and/or $\beta$ has only a few non-zero coordinates, and $\phi$ and $\psi$ are two (possibly nonlinear) link functions. We further denote $s_{\pi} \coloneqq \Vert \gamma \Vert_{0}$ and $s_{r} \coloneqq \Vert \beta \Vert_{0}$ as the sparsities of $\pi$ and $r$.

A natural solution in this context is to first estimate $\gamma$ and $\beta$ by high-dimensional regularized regression approaches, and then estimate $\tau$ by taking the sample average of $H (\theta)$, with $\theta$ replaced by the above estimates of $\gamma$ and $\beta$. This route was indeed taken by some of the earlier works in this direction, such as \citet{farrell2015robust} and \citet{chernozhukov2018double}\footnote{\citet{chernozhukov2018double} alternatively interpreted this construction as Neyman-Orthogonalization, a terminology also adopted in later works such as \citet{mackey2018orthogonal} and \citet{foster2023orthogonal}.}. In particular, the seminal Double Machine Learning (DML) estimator proposed in \citet{chernozhukov2018double} splits the entire sample into two non-overlapping subsets, one used to estimate $\gamma, \beta$ by $\ell_{1}$-regularized regression, denoted as $\hat{\gamma}, \hat{\beta}$, and the other used to estimate $\tau$ by taking the sample average of $H (\hat{\theta})$, denoted as $\hat{\tau}_{\DML}$:
\begin{equation}
\label{dml}
\hat{\tau}_{\DML} \coloneqq \frac{1}{n} \sum_{i = 1}^{n} H_{i} (\hat{\theta}), \text{ where } \hat{\theta} = (\hat{\pi}, \hat{r}) \text{ and } \hat{\pi} (x) = \phi (\hat{\gamma}^{\top} x), \hat{r} (x) = \psi (\hat{\beta}^{\top} x).
\end{equation}
Sample-splitting de-correlates the randomness in the nuisance parameter estimates $\hat{\theta}$ and in $\hat{\tau}_{\DML}$, eliminating the reliance on Donsker-type conditions on $\theta$, which are violated by high-dimensional sparse GLMs \citep{chernozhukov2018double}. The information loss due to sample-splitting can then be restored by cross-fitting \citep{schick1986asymptotically, ayyagari2010applications, zheng2011cross, robins2013new}. In this paper, we will also adopt the sample-splitting strategy, and estimate $\gamma$ and $\beta$ from a separate training sample (denoted as $\mathcal{D}_{\tr}$, see later Section \ref{sec:dipw}) using $\ell_{1}$-regularized GLM techniques. To simplify exposition, we omit the cross-fitting step. We refer readers to \citet{chernozhukov2018double, jiang2022new} for a more thorough discussion on the potential benefits of cross-fitting.

Since one typically estimates $\pi$ and $r$ at rates $\sqrt{s_{\pi} \log p / n}$ and $\sqrt{s_{r} \log p / n}$ using $\ell_{1}$-regularized GLMs \citep{buhlmann2011statistics}, it is straightforward to show that the DML estimator $\hat{\tau}_{\DML}$ is $\sqrt{n}$-consistent and asymptotically normal as long as\footnote{Throughout this paper, we adopt the common asymptotic notation such as $\ll, \gg, \lesssim, \gtrsim, \asymp$, $o (\cdot), O (\cdot), o_{\mathbb{p}} (\cdot), O_{\mathbb{p}} (\cdot)$.} 
\begin{equation}
\label{dml sparsity}
s_{\pi} \cdot s_{r} \ll n / \log^{2} p. \tag{DML sparsity}   
\end{equation}
Such a requirement on the product of the two sparsities ($s_\pi \cdot s_r$) could be quite stringent: e.g. if $s_{\pi} \asymp \sqrt{n} / \log p$, the \eqref{dml sparsity} condition implies $s_{r} \ll \sqrt{n} / \log p$. As a result, statisticians and econometricians have been puzzled by the following question:
\begin{customthm}{Problem*}
\label{question}
\textit{What is the sufficient and necessary, and hence minimal, sparsity condition on $\pi$ and $r$ for the parameter of interest, $\tau$, to be estimable at $1 / \sqrt{n}$-rate; and can we construct a $\sqrt{n}$-consistent estimator of $\tau$ under this minimal sparsity condition?}
\end{customthm}

\subsection{Literature overview and our contribution}

Since \citet{chernozhukov2018double}, various attempts have been made to weaken the \eqref{dml sparsity} condition. \citet{chernozhukov2022debiased} showed that $s_{\pi} \wedge s_{r} \ll \sqrt{n} / \log p$ is sufficient for the DML estimator $\hat{\tau}_{\DML}$ to be $\sqrt{n}$-consistent, yet with an additional (strong) assumption that $\Vert \gamma \Vert_{1}$ and $\Vert \beta \Vert_{1}$ are bounded, 
imposing (potentially unnecessary) constraints on the magnitudes of non-zero coordinates of $\gamma$ and $\beta$. Without the bounded $\ell_{1}$-norm condition, by leveraging the idea of covariate balancing \citep{imai2014covariate, zubizarreta2015stable, chan2016globally, athey2018approximate, ben2021balancing, bruns2023augmented} and certain structures of the linear OR and logistic-linear PS models, \citet{bradic2019sparsity} relaxed the \eqref{dml sparsity} condition to $s_{\pi} \ll n / \log p, s_{r} \ll \sqrt{n} / \log p$ or $s_{\pi} \ll \sqrt{n} / \log p, s_{r} \ll n^{3 / 4} / \log p$. Similar result to that of \citet{chernozhukov2018double} was also obtained in \citet{smucler2019unifying} (the rate-double-robustness part)\footnote{It is worth noting that \citet{chernozhukov2018double}, \citet{smucler2019unifying} and \citet{bradic2019minimax} also considered the slightly more general approximately sparse GLMs. In order not to confuse readers, in the Introduction, we will not explicitly distinguish the approximately sparse GLMs from the exactly sparse GLMs and paraphrase the established results in the literature under exactly sparse GLMs.}. But the result of \citet{smucler2019unifying} is not only derived for ATE, but also for a large class of parameters -- conventionally termed as ``DR functionals'' \citep{liu2024assumption} (also see Section 5 of \citet{chernozhukov2022locally}) -- with influence functions resembling the form in \eqref{if} \citep{rotnitzky2021characterization}.

Also exploiting the idea of covariate balancing, \citet{athey2018approximate} constructed a $\sqrt{n}$-consistent and asymptotically normal estimator when the OR model $r$ is linear with sparsity $s_{r} \ll \sqrt{n} / \log p$, but allowing $\pi$ to be arbitrarily complex, in which case, to avoid notation clutter, we set $s_{\pi} = \infty$. \citet{hirshberg2021augmented} later generalized the results of \citet{athey2018approximate} from ATE to a strict subclass of doubly-robust functionals. Complementing the results of \citet{athey2018approximate}, \citet{wang2024debiased} obtained a $\sqrt{n}$-consistent and asymptotically normal estimator when the PS model $\pi$ is logistic-linear\footnote{
	In fact, the results in \citet{wang2024debiased} can also be generalised to nonlinear link functions other than expit. They have discussed this informally in their Section~5.1.} with sparsity $s_{\pi} \ll \sqrt{n} / \log p$, but allowing $r$ to be arbitrarily complex, in which case we set $s_{r} = \infty$. This is desirable as, in applications such as clinical medicine, one often has a better grasp on the treatment assignment mechanism than on the outcome mechanism. The above results, taken together, suggest that the minimal sparsity condition is likely
\begin{equation}
\label{minimal sparsity}
s_{\pi} \wedge s_{r} \ll \sqrt{n} / \log p, s_{\pi} \vee s_{r} \text{ is arbitrary.} \tag{minimal sparsity}
\end{equation}
As an extension of \citet{athey2018approximate}, \citet{bradic2019minimax} constructed an $\sqrt{n}$-consistent ATE estimator that additionally accommodates the setting where both the OR model and the best linear projection of the PS model is sparse.
It shall be noted that they also exhibited two separate $\sqrt{n}$-consistent and asymptotically normal estimators for a different parameter of interest: the ``regression coefficient'', which we will discuss further in Remark \ref{rem:newey} of Section~\ref{sec:slope}. 
\citet{bradic2019minimax} also established necessary conditions for the existence of $\sqrt{n}$-consistent estimators of several parameters including ATE, which is believed to be tight. The guessed \ref{minimal sparsity} condition is also inspired by this result. The difference is that the necessary condition of \citet{bradic2019minimax} is stated in terms of how close the best linear approximation with roughly $\sqrt{n}$ many covariates approximates at least one of the two nuisance parameters. 

\citet{tan2020model, ning2020robust, avagyan2021high, dukes2021inference} and \citet{smucler2019unifying} (the model-double-robustness part) also tackled this problem, striving for $\sqrt{n}$-consistency and asymptotic normal estimators under the \eqref{minimal sparsity} condition. However, these results are also not entirely satisfactory (theoretically) -- although they allow one of two true nuisance parameters to be arbitrarily dense, the limits of their estimates by $\ell_{1}$-regularized regression, $\hat{\pi}$ and $\hat{r}$, still need to be sparse. 

As one can see, despite the tremendous progress made so far, a fully satisfactory solution to \ref{question} remains elusive. In this paper, we address \ref{question} by constructing novel $\sqrt{n}$-consistent and asymptotically normal estimators, under the \eqref{minimal sparsity} condition. This new methodology is based on a Double-Calibration (DCal) strategy to correct the bias of the DML estimator $\hat{\tau}_{\DML}$. The basic idea of DCal is also rooted in covariate balancing, which will be elucidated in detail in Section \ref{sec:ate}. It is noteworthy that the idea of using calibration to correct for the estimation bias of PS and OR has also been used in the regularized calibrated estimator of \citet{tan2020model}. The difference is that \citet{tan2020model} estimated the nuisance parameters and calibrated the estimation bias simultaneously, but our approach first proposes initial estimates of PS and OR, and then calibrates these estimates \textit{sequentially} via solving carefully-crafted constrained optimization problems.
Our new DCal estimators do not need to know which one of the two nuisance parameters, $\pi$ or $r$, is a GLM with sparsity below $\sqrt{n} / \log p$. Furthermore, the DCal estimator is not only applicable to the ATE, but also works for the regression coefficient, answering an open question raised in \citet{bradic2019minimax}. To the best of our knowledge, for both ATE and regression coefficient, no alternative estimators have been constructed before, that are $\sqrt{n}$-consistent and asymptotically normal under the \eqref{minimal sparsity} condition. Table \ref{tab:state-of-the-arts} directly compares some of the aforementioned results with ours. However, we want to point out an important caveat of our theoretical results -- we only establish $\sqrt{n}$-consistency and asymptotic normality of the DCal estimator under the \eqref{minimal sparsity} condition, but not its semiparametric efficiency \citep{newey1990semiparametric, jankova2018semiparametric}. Since the latter begs for establishing asymptotic lower bounds with tight constants, which is a much more involved problem, we leave this important question to future works.

\subsection{Notation and organization of the paper}
Before moving forwards, we outline some notation that is used throughout. We denote $\theta$ as the nuisance parameters. All expectations and probabilities (e.g. $\E$ and $\mathbb{P}$) are understood to be under the true data generating distribution unless stated otherwise (in that case, we will attach $\theta$ in the subscripts (e.g. $\E_{\theta}$ and $\mathbb{P}_{\theta}$) to emphasize the dependence of the distribution on $\theta$). As will be clear later in Section \ref{sec:ate}, we may split the entire dataset into three parts, a main dataset $\cD$ of size $n$, an auxiliary dataset $\cD_{\aux}$ of size $n_{\aux}$, and a training dataset $\cD_{\tr}$ of size $n_\tr$. 
For instance, $\mathbf{T} \equiv (T_1, \cdots, T_n)^\top$ and $\mathbf{X} = (X_1, \cdots, X_n)^\top$, respectively, denote the treatment variables and the covariates over the main dataset. The quantities computed from a given dataset will be attached, in the subscripts, by the index of that dataset. For example, the $n_{\tr} \times p$ design matrix from the training dataset will be denoted as $\bX_{\tr}$, and similarly for $\bX_{\aux}$. For each random variable from the training or the auxiliary dataset, we attach a superscript to distinguish them from the main dataset, e.g. $X_{i}^{(\aux)}$ is the covariate for the $i$-th subject from the auxiliary dataset. Occasionally, we use the empirical mean operators over the above three datasets, respectively denoted as $\mathbb{P}_{n}, \mathbb{P}_{n_{\tr}}$ and $\mathbb{P}_{n_{\aux}}$. We then let $\Vert X \Vert_{n, 2} \coloneqq (\{\mathbb{P}_{n} X_{1}^{2}\}^{1 / 2}, \cdots, \{\mathbb{P}_{n} X_{p}^{2}\}^{1 / 2})^{\top}$ denote the column-wise $\ell_{2}$-norms of a given matrix $\bX$. With a slight abuse of notation, for an $n \times p$-matrix $\mathbf{V}$ with its rows being $n$ i.i.d. copies $V_{1}, \cdots, V_{n}$ of a $p$-dimensional random vector $V$, we let $\Vert V \Vert_{n, 2}$ denote the $p$-dimensional vector of the column-wise $\ell_{2}$-norms of $\mathbf{V}$. Finally, given two vectors $v_1, v_2$ of the same length, $v_1 \odot v_2$ denotes the component-wise multiplication between $v_1$ and $v_2$.

In the sequel, we will frequently utilize an $\ell_{\infty}$-calibration approach realized by optimization programs with many inequality constraints. Since different optimization programs will be designed for different parameters, to streamline the notation, we introduce the following $\ell_{\infty}$-calibration operator. To this end, we first define the following two versions of short-hand notation for residuals: given generic random variables $S$ and $g$, $\Delta_{1, S, g}^{(a)} \coloneqq S^{(a)} - g$ and $\Delta_{2, S, g} \coloneqq S / g - 1$ where $a \in \{\aux, \mathrm{main}\}$ denotes the source of the dataset. As before, if $a = \mathrm{main}$, we silence the superscript. We deliberately choose not to define the second version with superscript $(a)$ because it will not depend on the auxiliary dataset in all appearances in our paper. With these, we define the $\ell_{\infty}$-calibration operator as follows:
\begin{equation}
\label{cal operator}
\hat{\bm{\varrho}} \coloneqq \calL_{\bm{\varrho}, \mathcal{C}_{\bm{\varrho}}} \left( \mathsf{f}; \Delta; \mathrm{h}; \mathrm{b} \right) \coloneqq \argmin_{\bm{\varrho} \in \mathcal{C}_{\bm{\varrho}}: |\mathbb{P}_{n} [\Delta \cdot \mathrm{h}]| \leq \mathrm{b}} \mathsf{f} (\bm{\varrho}).
\end{equation}
We now unpack the above definition. Here $\mathcal{C}_{\bm{\varrho}} \subset \R^{k}$ is the constraint set of $\bm{\varrho}$ and $\mathsf{f}: \R^{k} \rightarrow \R_{\geq 0}$ is the objective function. $\Delta$, depending on $\bm{\varrho}$, is the residual that we would like to calibrate such that it satisfies certain constraints that will be clear later. When $\Delta = \Delta_{1, S, g}^{(\aux)}$, to avoid clutter, we slightly abuse notation and define $\bbP_{n} [\Delta_{1, S, g}^{(\aux)} \cdot \mathrm{h}] \coloneqq \bbP_{n_{\aux}} [S^{(\aux)} \mathrm{h}^{(\aux)}] - \bbP_{n} [g \mathrm{h}]$. Here $\mathrm{h} \in \R^{l}$ denotes the (random) test functions that the residuals $\Delta^{(a)}$ are calibrated against and $\mathrm{b} \in \R^{l}$ of the same dimension as $\mathrm{h}$ denotes the upper bounds that the calibration aims to fulfill via inequality constraints. As will be clear later, $\mathrm{b}$ can be allowed to be data-dependent, hence possibly random.

The rest of this paper is organized as follows. In Section \ref{sec:ate}, we illustrate our new methodology for ATE, with the theoretical analysis detailed in Section \ref{sec:atethm}. Section \ref{sec:extensions} extends our new approach to (1) the regression coefficient estimation in a high-dimensional partially linear model, (2) the estimation of local average treatment effect with instrumental variables with high-dimensional observed confounding variables, and (3) the approximately sparse GLMs considered in \citet{bradic2019minimax}. In Section \ref{sec:sim}, for the problem of estimating ATE, we examine the finite-sample performance of a variant of the DCal estimator that is computationally more feasible. Compared against several popular methods, our new method exhibits comparable, and sometimes improved statistical properties and computational runtime. The code for replicating our simulation studies can be accessed via the link \href{https://github.com/Cinbo-Wang/DCal}{https://github.com/Cinbo-Wang/DCal}. Lastly, Section \ref{sec:conclusions} concludes this article and discusses some future directions. Some of the technical details are deferred to the Appendix.

\begin{table}[ht]
    \centering
    \begin{tabular}{c|c}
    	\thickhline
    	Paper & condition(s) on $s_{\pi}$ and $s_{r}$ \\
    	\thickhline
    	\citet{chernozhukov2018double} & $s_{\pi} s_{r} = o (n / \log^{2} p)$ \\
    	\hline
    	\multirow{2}{*}{\citet{chernozhukov2022debiased}} & $s_{\pi} \wedge s_{r} = o (\sqrt{n} / \log p)$ \\
    	& but $\Vert \beta \Vert_{1}$ and $\Vert \gamma \Vert_{1}$ are bounded \\
    	\hline
    	\citet{smucler2019unifying} & \multirow{2}{*}{$s_{\pi} s_{r} = o (n / \log^{2} p)$} \\
    	rate-double-robustness & \\
    	\hline
    	\multirow{2}{*}{\citet{bradic2019sparsity}} & $s_{\pi} = o (n / \log p), s_{r} = o (\sqrt{n} / \log p)$ \\
    	& or $s_{\pi} = o (\sqrt{n} / \log p), s_{r} = o (n^{3 / 4} / \log p)$ \\
    	\hline
    	\citet{athey2018approximate} & $s_{r} = o (\sqrt{n} / \log p)$, $s_\pi$ arbitrary \\
    	\hline
    	\citet{wang2024debiased} & $s_{\pi} = o (\sqrt{n} / \log p)$, $s_r$ arbitrary \\
    	\hline
    	\citet{smucler2019unifying} & $s_{\pi} \wedge s_{r} = o (\sqrt{n} / \log p)$, $s_{\pi} \vee s_{r}$ arbitrary \\
    	model-double-robustness & but the estimator of the denser model has sparse limit \\
    	\hline
    	Our result & $s_{\pi} \wedge s_{r} = o (\sqrt{n} / \log p)$, $s_{\pi} \vee s_{r}$ arbitrary \\
    	\thickhline
    \end{tabular}
    \caption{A comparison between our theoretical results and other related works on ATE estimation. Note that, for example, ``$s_\pi$ arbitrary'' means that $s_\pi$ can be as large as $\infty$, namely that $\pi$ does not necessarily need to follow a GLM. In fact, it can be arbitrarily complex. The results of \citet{bradic2019minimax} are difficult to be accommodated into this table due to the limited space and we refer interested readers directly to our literature overview for details.}
    \label{tab:state-of-the-arts}
\end{table}

\section{Average treatment effects: The Double-Calibration estimator}
\label{sec:ate}

In this section, we illustrate our new ``Double-Calibration (DCal)'' methodology for ATE estimation. We observe data $\{O_{i} \coloneqq (X_i, T_i, Y_i)\}_{i=1}^N$ that are i.i.d. draws from some underlying data-generating distribution $\bbP_{\theta^{\ast}}$ where from this section onward $\theta^{\ast} \coloneqq (\pi^{\ast}, r^{\ast})$ denotes the true OR and PS models. We assume throughout that, as standard in the literature on treatment effect estimation, $\bbP_{\theta^{\ast}}$ satisfies the unconfoundedness assumption:
\begin{condition}\label{cd:ignorability}
\[
Y_i(t) \independent T_i \mid X_i \text{ almost surely}, \text{ for $t \in \{0, 1\}$},
\]
\end{condition}
\noindent where $Y_i(1)$ and $Y_i(0)$ are two potential outcomes; and the overlap/positivity assumption:
\begin{condition}\label{cd:overlap}
There is an absolute constant $c_\pi \in (0, 0.5)$, such that $c_\pi < \pi^*(X) < 1 - c_\pi$ almost surely.
\end{condition}
\noindent It is well-known that, under Consistency, Conditions \ref{cd:ignorability} and \ref{cd:overlap}, the parameter of interest, ATE, can be identified as $\tau^\ast \equiv \E [r^{\ast} (X)] \equiv \E \left[ \frac{T Y}{\pi^{\ast} (X)} \right] \equiv \E [H (\theta^\ast)]$. We further impose the following modeling assumption on the nuisance parameters $\theta^{\ast} = (\pi^{\ast}, r^{\ast})$: 
\begin{condition}\label{cond:nuis}\leavevmode
There exist two (nonlinear) monotonically increasing, twice-differentiable link functions $\phi, \psi$ with uniformly bounded first and second derivatives, such that {\bf either} of the following holds:
\begin{enumerate}
\item[(i)] The OR model follows a GLM with (nonlinear) link $\psi$: $r_i^* \equiv r^* (X_i) \coloneqq \E(Y_i \mid X_i, T_i = 1) \equiv \psi (X_i^\top \beta^*)$;
\item[(ii)] The PS model follows a GLM with (nonlinear) link $\phi$: $\pi_i^* \equiv \pi^* (X_{i}) \coloneqq \pr(T_i = 1 \mid X_i) \equiv \phi (X_i^\top \gamma^*)$.
\end{enumerate}
Since $\pi^{\ast} \in (0, 1)$, we also restrict $\phi$ such that $\phi(t) \to 1$ as $t \to \infty$, and $\phi(t) \to 0$ as $t \to -\infty$.
\end{condition}
\noindent In words, Condition \ref{cond:nuis} states that at least one of the two nuisance parameters is truly a GLM, whereas the other one can be arbitrarily complex. We restrict the range of $\phi$ to respect the fact that $\pi$ is a probability, but this is not essential for the theoretical results of this paper.

In the rest of this section and Section \ref{sec:atethm}, we consider to have access to three separate datasets: the main dataset $\cD \coloneqq \{(X_i, T_i, Y_i)\}_{i=1}^n$, the auxiliary dataset $\cD_{\aux} \coloneqq \{(X_{\aux, i}, T_{\aux, i}, Y_{\aux, i})\}_{i=1}^{n_{\aux}}$, and the training dataset $\cD_\tr \coloneqq \{(X_{\tr, i}, T_{\tr, i}, Y_{\tr, i})\}_{i=1}^{n_\tr}$ with $n \asymp n_{\aux} \asymp n_{\tr}$. Such a setup is equivalent to splitting the entire sample with size $N = n + n_{\aux} + n_\tr$ into three non-overlapping subsets. As mentioned in the Introduction, cross-fitting can restore the information loss due to sample splitting, but we choose to omit this step to simplify the exposition. Under this simplified observation scheme, we instead consider the target parameter of interest as the ATE averaged over the main dataset $\cD$\footnote{It is also legitimate to choose two other alternatives: $\bar{\tau}^{\ast, 1} \coloneqq n^{-1} \sum_{i = 1}^n T_i Y_i / \pi^\ast (X_i)$ or $\bar{\tau}^{\ast, 2} \coloneqq n^{-1} \sum_{i = 1}^n H_i (\theta^\ast)$. We decided to go with $\bar{\tau}^{\ast}$ only for convenience.}:
\begin{equation} \label{ate target}
\bar{\tau}^{\ast} \coloneqq \frac{1}{n} \sum_{i = 1}^{n} r^{\ast} (X_i).
\end{equation}

\subsection{Preliminary: Singly-calibrating the OR model}
\label{sec:dipw}

To lay the ground, we start by considering the following estimator that ``singly'' calibrates the OR model, denoted as $\hat{\tau}_{\dipw}$:
\begin{align}
\label{sc}
\hat{\tau}_\dipw \coloneqq \frac{1}{n} \sum_{i=1}^n H_i (\hat{\pi}, \tilde{r}) \equiv \frac{1}{n} \sum_{i=1}^n \left(\frac{T_i (Y_i - \tilde{r}_i)}{\hat{\pi}_i} + \tilde{r}_i\right), \text{ where } \tilde{r}_i \coloneqq \hat{r}_i + \hat{\mu}_i.
\end{align}
Recall that $\r (x) \coloneqq \psi(x^\top \hat{\beta})$ is simply the plug-in GLM OR estimate with $\hat{\beta}$ estimated from the training dataset $\cD_\tr$. For the estimated PS model $\hat{\pi}$, we change the definition slightly and apply a trimming operation on the index $x^\top \hat{\gamma}$:
\[
\hat{\pi} (x) \coloneqq \phi \left(x^\top \hat{\gamma} \one\{|x^\top \hat{\gamma}| \le M_\gamma\} + \sign(x^\top \hat{\gamma})  M_\gamma \one\{|x^\top \hat{\gamma}| > M_\gamma\}\right),
\]
where $\sign (\cdot)$ is the function that returns the sign of the input variable. As in the OR case, $\hat{\gamma}$ is also estimated from the dataset $\cD_\tr$; $M_\gamma$ is either chosen a priori or is chosen adaptively from $\cD_\tr$. With such trimming operation, the overlap condition always holds for the estimated PS model. When the PS model is correctly specified and follows some sparsity constraint, one can prove that with probability converging to $1$, $\hat{\pi}_i \equiv \phi(X_i^\top \hat{\gamma})$, i.e., such trimming is moot (see e.g. Section~\ref{sec:prop}); indeed, the trimming is only helpful when the PS model is misspecified.
We now turn to $(\hat{\mu}_1, \cdots, \hat{\mu}_n)^\top \eqqcolon \hat{\mbb\mu}$, which is the solution to the following constrained program:
\begin{align} 
	\hat{\mbb\mu} = & \;\argmin_{\mbb\mu \in \R^n} \|\mbb\mu\|_2^2 \tag{6a} \label{dipw_obj} \\
	\st \;\; & \left\|\frac{1}{n_{\aux}} {\bX}_{\aux}^\top {\mbb \Pi}_{\aux} \tilde{\mathbf R}_{\aux} - \frac{1}{n} {\bX}^\top {\mbb \Pi} {\mbb \mu} \right\|_\infty \leq \eta_r, \tag{6b} \label{eq:dipwprogram} \\
 & \left\|{\mbb \mu}\right\|_\infty \le M_r \tag{6c} \label{eq:dipwinfinity}
\end{align}
where $M_r > 0$ is some known constant that depends on the bounds on $\hat{r}$ and $r^{\ast}$, $\bX \in \R^{n \times p}$ is the covariate matrix, ${\mbb \Pi} \in \R^{n \times n}$ is a diagonal matrix whose $(i, i)$-th entry is $\phi'(X_i^\top \hat{\gamma}) / \hat{\pi}_i$ and $\tilde{\mathbf R}$ is the $n$-dimensional residual vector whose $i$-th element is $\tilde{R}_i \coloneqq T_i / \hat{\pi}_i \cdot (Y_i - \r_i)$. Recalling from the Notation section, the quantities $\bX_\aux$ etc. are defined analogously on the auxiliary dataset $\cD_\aux$. In addition, it is understood that if $\mbb\mu \equiv \bm{0}$ has already satisfied the constraints \eqref{eq:dipwprogram} and \eqref{eq:dipwinfinity}, the above program will simply output $\hat{\mbb\mu} = \bm{0}$, or equivalently, $\tilde{r}_{i} \equiv \hat{r}_{i}$ for all $i = 1, \cdots, n$.

Since we have introduced the $\ell_{\infty}$-calibration operator \eqref{cal operator}, we can equivalently represent the above constrained program \eqref{dipw_obj} to \eqref{eq:dipwinfinity} more compactly as follows:
\addtocounter{equation}{+1}
\begin{equation}
\label{sc operator}
\hat{\bm{\mu}} \coloneqq \calL_{\bm{\mu}, B_{\infty}^{n} (M_{r})} \left( \Vert \cdot \Vert_{2}^{2}; \Delta^{(\aux)}_{1, T (Y - \hat{r}) / \hat{\pi}, \mu}; \phi' (X^{\top} \hat{\gamma}) X / \hat{\pi}; \mathbbm{1}_{p} \eta_{r} \right)
\end{equation}
where $B_{\infty}^{n} (M_{r})$ denotes the $\ell_{\infty}$-ball in $\R^{n}$ with radius $M_{r}$. The correspondence of the above display to the generic $\ell_{\infty}$-calibration operator $\calL_{\bm{\varrho}, \mathcal{C}_{\bm{\varrho}}} (\mathsf{f}; \Delta; \mathrm{h}; \mathrm{b})$ defined previously is as follows:
\begin{align*}
\bm\varrho & \leftarrow \bm\mu, \mathsf{f} (\cdot) \leftarrow \Vert \cdot \Vert_{2}^{2}, \mathcal{C}_{\bm{\varrho}} \leftarrow B_{\infty}^{n} (M_{r}), \\ 
\Delta & \leftarrow \Delta^{(\aux)}_{1, T (Y - \hat{r}) / \hat{\pi},  \mu }, \mathrm{h} \leftarrow \phi' (X^{\top} \hat{\gamma}) X/ \hat{\pi}, \mathrm{b} \leftarrow \mathbbm{1}_{p} \eta_{r}.
\end{align*}

\begin{remark}
\label{rem:dipw}
$\hat{\tau}_\dipw$ is the same as the DML estimator $\hat{\tau}_{\DML}$ except $\hat{r}$ is replaced by the post-calibration OR estimator $\tilde{r} = \hat{r} + \hat{\mu}$. It is in fact a variant of the recently proposed DIPW estimator for ATE \citep{wang2024debiased}. Constraint \eqref{eq:dipwprogram} balances $\tilde{r}$ such that it is on average close to $Y$, gauged by covariates $X$ under certain data-dependent weights. The use of the auxiliary dataset $\cD_{\aux}$ is to preserve the following moment equations satisfied by the calibrated OR estimator $\tilde{r}$:
\begin{equation}
\label{moment_ps}
\frac{1}{n} \sum_{i = 1}^n \left( \frac{T_i}{\pi^{\ast} (X_i)} - 1 \right) \hat{\mu}_i, \text{ and thus } \frac{1}{n} \sum_{i = 1}^n \left( \frac{T_i}{\pi^{\ast} (X_i)} - 1 \right) \tilde{r}_i \text{ have mean zero},
\end{equation}
because $\hat{\mu}_i$'s, as constructed, only depends on $X$, but not on $T$ or $Y$, from the main dataset. Note that \eqref{moment_ps} in fact holds with any quantity, that is independent of $T_i$ given $X_i$ from the main dataset $\cD$, in place of $\hat{\mu}_i$. The above moment equation plays a key role in the analysis of the statistical property of $\hat{\tau}_\dipw$ in \citet{wang2024debiased}.
\end{remark}

Using the analysis in the proof of \citet[Theorem~2]{wang2024debiased}, since $\hat{\beta}$ and $\hat{\gamma}$, computed using $\cD_\tr$, are independent from the main sample $\cD$, by choosing $\eta_r \asymp \sqrt{\frac{\log p}{n}}$, $\hat{\tau}_\dipw$ can be shown to be a $\sqrt{n}$-consistent estimator of $\bar{\tau}^\ast$ whenever the PS-sparsity $s_\pi \ll \sqrt{n} / \log p$, under Conditions \ref{cd:ignorability} to \ref{cond:nuis}.

Nevertheless, to achieve $\sqrt{n}$-rate under the \eqref{minimal sparsity} condition, $\hat{\tau}_{\sc, r}$ needs to be modified so that it is also $\sqrt{n}$-consistent if the OR-sparsity $s_r \ll \sqrt{n} / \log p$ but the PS model is arbitrary. In the next Section \ref{sec:dcal}, we will complete this remaining piece.

\subsection{Double calibration}
\label{sec:dcal}

$\hat{\mbb \mu}$ corrects the OR model misspecification in the initial OR estimate $\hat{r}$ and $\tilde{r} = \hat{r} + \hat{\mu}$ can be viewed as the post-calibrated OR estimator. To correct the misspecification of the PS model, we also need to calibrate the initial PS estimate $\hat{\pi}$. To that end, we propose the following Doubly-Calibrated (DCal) estimator for ATE:
\[
\hat{\tau}_{\est} \coloneqq \frac{1}{n} \sum_{i=1}^n H_i (\tilde{\pi}, \tilde{r}) \equiv \frac{1}{n} \sum_{i=1}^n \left(\frac{T_i(Y_i - \tilde{r}_i)}{\tpi_i} + \tilde{r}_i \right).
\]
Apparently, it is the same as the OR-calibrated estimator proposed in the previous section, except that $\hat{\pi}_i$ is replaced by $\tpi_i = \phi(\tilde{X}_i^\top \tilde{\gamma})$, where $\tilde{X}_i$ is the same as $X_i$ if $p \geq n$, or $\tilde{X}_i$ is equal to an $n$-dimensional vector with $X_i$ augmented by some synthetic covariates of dimension $n - p$ if $p < n$. For convenience, here the synthetic covariates are randomly generated from the uniform distribution over the interval $[-1, 1]$ but other strategies are also possible. This covariate augmentation ensures that almost surely there is a vector $\tilde\gamma^*$ such that $\pi_i^*$, the true PS evaluated at the $i$-th sample, satisfies $\pi_i^* = \phi(\tilde{X}_i^\top \tilde\gamma^*)$, no matter whether or not $\pi^\ast$ follows a GLM with link $\phi$. Writing $\bX_j, \tilde{\bX}_j$ as the $j$-th column of data matrices $\bX \in \R^{n \times p}, \tilde{\bX} \in  \R^{n \times (n \vee p)}$, similar to $\hat{\bm{\mu}}$, $\tilde{\gamma}$ solves the following constrained program:
\begin{align}
	\tilde{\gamma} \coloneqq & \ \argmin_\gamma \|\gamma - \hat{\gamma}\|_1 \tag{9a} \label{eq:program}\\
	\st \ & \left|\frac{1}{n} \sum_{i=1}^n \left(\frac{T_i}{\pi_i} - 1 \right) \psi'(X_i^\top \hat{\beta})X_{i,j}\right| \leq \eta_{\pi_1} \cdot \frac{\|\psi'(\bX \hat{\beta}) \odot \bX_j\|_2}{\sqrt{n}} \quad \forall\; 1 \le j \le p \tag{9b} \label{eq:const1}\\
	& \left|\frac{1}{n} \sum_{i=1}^n \left(\frac{T_i}{\pi_i} - 1 \right) \hmu_i\right| \leq \eta_{\pi_1} \cdot \frac{\|\hat{\bmu}\|_2}{\sqrt{n}} \tag{9c} \label{eq:const2}\\
	& \left|\frac{1}{n} \sum_{i=1}^n \left(\frac{T_i}{\pi_i} - 1 \right) \tilde{X}_{i,j}\right| \leq \eta_{\pi_2} \cdot \frac{\|\tilde{\bX}_j\|_2}{\sqrt{n}} \quad \forall\; 1 \le j \le p \vee n \tag{9d} \label{eq:constprop}\\
	& \max_{i} T_i / \pi_i \le M_\pi \tag{9e} \label{eq:const3}.
\end{align}
In this program, $\pi_i \equiv \pi (X_i) \coloneqq \phi(\tilde{X}_i^\top \gamma)$ and $M_{\pi} > 0$ is some known constant that depends on $c_\pi$\footnote{In~\eqref{eq:const1}, we use the convention that a function $g: \R^p \mapsto \R$ applied to a matrix ${\mbb M} \in \R^{n \times p}$ ($p$ could be equal to $1$) is a $n$-dimensional vector formed by applying $g$ to each row of the matrix.}. Notice also that in the above program, when $n > p$, then with a slight abuse of notation, we let $\hat{\gamma}$ denote the $n$ dimensional vector with first $p$ element equal to the original estimate $\hat{\gamma}$, and the rest $n - p$ coordinates equal to zero. Similar to the OR calibration step, we can compactly represent the above PS calibration step using the $\ell_{\infty}$-calibration operator:
\addtocounter{equation}{+1}
\begin{equation}
\label{ps operator}
\tilde{\gamma} \coloneqq \calL_{\gamma, \tilde{B}_{\infty}^{p} (M_{\pi})} \left( \Vert \cdot - \hat{\gamma} \Vert_{1}; \Delta_{2, T, \phi (X^{\top} \gamma)}; \mathrm{h}; \mathrm{b} \right)
\end{equation}
where
\begin{align*}
\tilde{B}_{\infty}^{p} (M_{\pi}) & \coloneqq \left\{ \gamma: \Vert T / \phi (\tilde{X}^{\top} \gamma) \Vert_{\infty} \leq M_{\pi} \right\}, \mathrm{h} \coloneqq \left( \psi' (X^{\top} \hat{\beta}) X^{\top}, \hat{\mu}, \tilde{X}^{\top} \right)^{\top}, \\
\text{and } \mathrm{b} & \coloneqq \left( \frac{\eta_{\pi_1}}{\sqrt{n}} \|\psi' (X^{\top} \hat{\beta}) X \|_{n, 2}, \frac{\eta_{\pi_{1}} \Vert \hat{\bm{\mu}} \Vert_{2}}{\sqrt{n}}, \frac{\eta_{\pi_{2}}}{\sqrt{n}} \Vert \tilde{X} \Vert_{n, 2} \right)^{\top},
\end{align*}
where we recall that, for an $n \times p$-matrix $\mathbf{V}$ with its rows being $n$ i.i.d. copies $V_{1}, \cdots, V_{n}$ of a $p$-dimensional random vector $V$, we let $\Vert V \Vert_{n, 2}$ denote the $p$-dimensional vector of the column-wise $\ell_{2}$-norms of $\mathbf{V}$. In particular, if the initial estimator $\hat{\gamma}$ has already satisfied all the inequality constraints, \eqref{ps operator} will simply output $\hat{\gamma}$ without further solving the optimization problem.


An appealing property of \eqref{eq:program} is that it is always guaranteed to have a feasible solution, even without any assumptions on the sparsity of the PS or the OR model:

\begin{lemma}\label{lem:program}
Suppose Condition~\ref{cd:overlap} holds, then with $\eta_{\pi_1} \asymp \sqrt{\frac{\log p}{n}}, \eta_{\pi_2} \asymp \sqrt{\frac{\log p \vee n}{n}}$ large enough, and with $M_\pi \ge c_\pi^{-1}$, we have that given any full rank $\bX$ and $(\hat{\mbb \mu}, \hat{\gamma}, \hat{\beta})$, with probability converging to $1$, any $\tilde\gamma^*$ that satisfies $\pi_i^* \equiv \phi(\tilde{X}_i^\top \tilde\gamma^*)$ for all $i$ is a feasible solution to~\eqref{eq:program}. 
\end{lemma}
Note that in Lemma~\ref{lem:program}, we say that ``$\bX$ is full rank'' whenever the rank of $\bX$ equals $\min\{n, p\}$, i.e., whenever it is of full row/column rank. When the program is not feasible, one can simply set $\gamma$ as $\hat{\gamma}$. Noteworthy, Lemma~\ref{lem:program} does not necessarily require $\pi^*$ follow a GLM with link $\phi$. In fact, it can be any function satisfying Condition~\ref{cd:overlap}. 

\begin{remark}
\label{rem:dc}
$\hat{\tau}_\est$ is the same as $\hat{\tau}_\dipw$ except now $\hat{\pi}$ is also replaced by the post-calibration PS estimator $\tilde{\pi} (x) = \phi (\tilde{\gamma}^{\top} x)$. Constraints \eqref{eq:const1} to \eqref{eq:constprop} in the program \eqref{eq:program} balances $\tilde{\pi}$ such that $T / \tilde{\pi}$ is on average close to $1$, gauged by covariates $\bX$ (and the augmented covariates $\mathbf{\tilde{X}}$) and the OR calibrator $\bm{\hat{\mu}}$ under certain data-dependent weights. There is no need to use the auxiliary dataset in the constrained program \eqref{eq:program} because the following moment equation holds as long as $\tilde{\pi}_i$ does not depend on $Y_i$:
\begin{equation}
\label{moment_or}
\frac{1}{n} \sum_{i = 1}^n \frac{T_i}{\tilde{\pi} (X_i)} (Y_i - r_i^\ast) \text{ has mean zero}.
\end{equation}
Obviously, the program \eqref{eq:program} is unrelated to $Y$ from the main dataset. The moment equations \eqref{moment_or} and \eqref{moment_ps} in Remark \ref{rem:dipw} are direct consequence of the double-robustness of the influence function $\IF (\theta)$ defined in \eqref{if}. 

It is worth noting that $\hat{\tau}_{\dc}$ calibrates OR and PS in order, and this ordering is vital for the two moment equations \eqref{moment_ps} and \eqref{moment_or} to hold. Were the ordering reversed, with $\hat{\bm{\mu}}$ calibrated using $\tilde{\pi}$ which in turn depends on $\mathbf{T}$ in the main sample, it would not have been as straightforward to show that moment equation \eqref{moment_ps} holds (even approximately).
\end{remark}

\section{Average treatment effects: theoretical analysis}
\label{sec:atethm}

In this section, we derive the asymptotic result for the proposed DCal estimator $\hat{\tau}_{\est}$. For simplicity, throughout this section and the next section, we take $n \equiv n_{\aux} \equiv n_\tr$ and assume throughout that we are under an asymptotic regime where $p \to \infty$ as $n \to \infty$ and that $\log p / \sqrt{n} = o(1)$. We also assume throughout the following conditions on the design matrix and outcomes as well as the estimated outcome regression model. Recall that $\hat{\pi}_{\aux}$ and $\hat{r}_{\aux}$ appears in the constraint \eqref{eq:dipwprogram} when calibrating the OR model.
\begin{condition}\label{cd:all}
We have that
\begin{itemize}
    \item[(i)] For some constant $m_Y > 0$, almost surely, $|Y(1)|, |r^*(X)| \le m_Y$;
    \item[(ii)] The first component of $X$ is $1$, representing an intercept term. Denoting by $V \in \R^{p - 1}$ the remaining components of $X$, we assume $\E V = 0$ and almost surely, $\|V\|_\infty \le m_V$. Moreover, $\bX$ is full rank and there exists a constant $\sigma_V$ such that for any unit vector $u \in \R^{p - 1}$ and all $\alpha \in \R$, $\E \{\exp(\alpha u^\top V)\} \le \exp(\alpha^2 \sigma_V^2 / 2)$.
\end{itemize}
\end{condition}

\begin{condition}\label{cd:est}
	There exists a constant $m_r > 0$ such that with probability converging to $1$, $\pr(|\psi(X^\top \hat{\beta})| \le m_r \mid \hat{\beta}) = 1$, where $X$ is a random vector that is independent from $\hat{\beta}$ and follows the same distribution as the $X_i$'s.
\end{condition}

Notice that as we have mentioned in the comments after Lemma~\ref{lem:program}, we say that ``$\bX$ is full rank'' if the rank of $\bX$ equals $\min\{n, p\}$. 
Notice also that the $\ell_\infty$ constraint on $V$ can be removed by imposing additional constraints in the optimization programs~\eqref{dipw_obj} and~\eqref{eq:program}. The purpose of the boundedness assumption imposed on $Y$ is for us to focus on the main thesis of the paper -- constructing a $\sqrt{n}$-consistent estimator of $\tau$ under \eqref{minimal sparsity}. We expect that similar conclusions should hold under light-tail assumptions such as sub-Gaussianity. 
In the rest of this section, we first prove that $\hat{\tau}_{\est}$ is $\sqrt{n}$-consistent when the OR follows a sparse GLM model in Section~\ref{sec:reg}, followed by a proof of its $\sqrt{n}$-consistency when the PS follows a sparse GLM model in Section~\ref{sec:prop}.

\subsection{Theoretical properties under the sparse OR model}\label{sec:reg}

Recall that our target of interest is the ATE over the main sample $\bar{\tau}^\ast$ in \eqref{ate target}. In this section, we prove the $\sqrt{n}$-consistency of $\hat{\tau}_{\est}$ for estimating $\bar{\tau}^\ast$ when only the OR model is sparse. To achieve this goal, we invoke the following conditions on the nuisance parameter estimates computed from the training and the auxiliary datasets. 

Moreover, we impose the following assumption on the estimated OR coefficient $\hat{\beta}$ computed from the training sample $\cD_{\tr}$:


\begin{condition}\label{cd:reg}
Condition~\ref{cond:nuis} (i) holds true and there exists a constant $m_\psi$ such that for any $t$, $\psi'(t), \psi''(t) \le m_\psi$. Moreover, $\beta^*$ has sparsity $s_r$ and with probability converging to $1$,
 \[
    \|\hat{\beta} - \beta^*\|_1 = O \left(s_r \sqrt{\frac{\log p}{n}}\right) \quad\&\quad \|\hat{\beta} - \beta^*\|_2 = O \left(\sqrt{s_r \frac{\log p}{n}}\right).
\]
\end{condition}
Condition \ref{cd:reg} can be met by using common $\ell_{1}$-regularized techniques such as the lasso \citep{tibshirani1996regression}, square-root lasso \citep{belloni2011square}, SLOPE \citep{su2016slope,BLT18} when the penalized loss function satisfies certain regularity conditions such as strong convexity.



We then have the following important observation (Lemma \ref{lem:mu}): when the initial OR estimate $\r$ is close to the true OR $r^\ast$, correspondingly $\|\hat{\mbb \mu}\|_2$ is small. This observation tells us that the OR-calibration is harmless when $\r$ is already close to $r^\ast$. This is appealing because in practice, we generally have no idea if $\r$ converges to $r^\ast$ at a sufficiently fast rate even under the \eqref{minimal sparsity} condition.



\begin{lemma}\label{lem:mu}
Suppose 
Conditions~\ref{cd:ignorability}--\ref{cd:overlap} and \ref{cd:all}--\ref{cd:reg}
hold with $s_r = o(\sqrt{n} / \log p)$ and let the constant $M_\gamma$ be given; then by choosing $\eta_r \asymp \sqrt{\frac{\log p}{n}}$, $M_r, M_\gamma \asymp 1$ sufficiently large, $\|\hat{\bmu}\|_2 = o_\pr(n^{1/4})$. 
\end{lemma}
The claim of Lemma \ref{lem:mu} is expected as $\hat{\bm\mu}$, by design, calibrates the estimation error $r_i^* - \r_i$. When $\r_i$ estimates $r_i^*$ with high accuracy, intuitively one should expect the calibrator $\hat{\bmu}$ to be negligible. Now equipped with the above result, using Taylor expansion, we can decompose $\hat{\tau}_\est - \bar{\tau}^\ast$ as
\begin{align*}
	\hat{\tau}_\est & - \bar{\tau}^\ast \\
	& \approx \frac{1}{n} \sum_{i=1}^n \left(\frac{T_i}{\tpi_i} - 1\right) \psi'(X_i^\top \hat{\beta}) X_i^\top (\hat{\beta} - \beta^*) + \frac{1}{n} \sum_{i=1}^n \left(\frac{T_i}{\tpi_i} - 1\right) \hmu_i + \frac{1}{n} \sum_{i=1}^n \frac{T_i \varepsilon_i(1)}{\tpi_i}, 
\end{align*}
where $\varepsilon_i(1) \coloneqq Y_i(1) - r^*(X_i)$ is the residual noise of the random variable $Y_i(1)$. As explained in Remark \ref{rem:dc}, by design of the program \eqref{eq:program}, the last term is mean zero and asymptotically normal. Therefore the first two terms constitute the bias. By applying \Holder{}'s inequality and using the constraint of $\tpi_i$ in~\eqref{eq:const1} and~\eqref{eq:const2}, we have that the first two terms are, with probability converging to one, of order
\[
s_r \frac{\log p}{n} + \frac{\sqrt{\log p}}{n} \|\hat{\bmu}\|_2.
\]
Now in light of Lemma~\ref{lem:mu}, we further have that the bias is of order 
\[
	s_r \frac{\log p}{n} + \sqrt{s_r} \frac{\log p}{n},
\]
which is asymptotically negligible whenever $s_r = o(\sqrt{n} / \log p)$. Putting together the above arguments, we establish Theorem~\ref{thm:reg} below. 


\begin{theorem}\label{thm:reg} Suppose 
Conditions~\ref{cd:ignorability}--\ref{cd:overlap} and \ref{cd:all}--\ref{cd:reg} hold with $s_r = o(\sqrt{n} / \log p)$ and let the constant $M_\gamma$ be given, by choosing $\eta_r, \eta_{\pi_1} \asymp \sqrt{\frac{\log p}{n}}, \eta_{\pi_2} \asymp \sqrt{\frac{\log p \vee n}{n}}, M_r, M_\pi \asymp 1$ sufficiently large, we have the following representation
	\[
	\sqrt{n} (\hat{\tau}_\est - \bar{\tau}^\ast) = \frac{1}{\sqrt{n}} \sum_{i = 1}^n \frac{T_i \varepsilon_i(1)}{\tpi_i} + o_\pr(1).
	\]
\end{theorem}

Let $\bar{\sigma}_r^2 \coloneqq n^{-1} \sum_{i=1}^n T_i \var(\varepsilon_i(1) \mid X_i) / \tpi_i^2$. Informally, Theorem~\ref{thm:reg} states that with sufficiently large $n$, with high probability, we can obtain by conditioning on $\{(\bX, \mathbf{T}), \cD_{\aux}, \cD_\tr\}$ that
\[
\sqrt{n} (\hat{\tau}_\est - \bar{\tau}^\ast) \mid \{(\bX, \mathbf{T}), \cD_{\aux}, \cD_\tr\} \approx \calN (0, \bar{\sigma}_r^2). 
\]
Once we further have the assumption that almost surely, $\var(\varepsilon_i(1) \mid X_i)$ is bounded above by some constant, we have that with probability converging to $1$, $\bar{\sigma}_r$ is $O (1)$. This means that $\hat{\tau}_\est$ is a $\sqrt{n}$-consistent estimator for $\bar{\tau}^\ast$. In fact, our convergence guarantee is quite similar to the one given by \citet{athey2018approximate}, except that they replaced $1 / \tpi_i$ by some weights learned by solving a convex program.

\begin{remark}
\label{rem:var reg}
As mentioned in the Introduction, one missing piece of Theorem \ref{thm:reg} is if the asymptotic linear representation can achieve the semiparametric efficiency bound (with respect to the main dataset $\cD$), which requires the representation to be of the following form:
\begin{align*}
\frac{1}{\sqrt{n}} \sum_{i = 1}^{n} \left( \frac{T_i}{\pi_i^\ast} \varepsilon_i (1) + r^\ast_i - \tau^\ast \right) + o_{\mathbb{P}} (1).
\end{align*}
We conjecture that it should be the case under the additional assumption that $\pi^\ast$ can be consistently estimated by $\tilde{\pi}$ when only the OR is assumed to be sparse.
\end{remark}



\subsection{Theoretical properties under the sparse PS model} \label{sec:prop}

In this section, we consider ``the other way around'', namely when the PS model satisfies a sparse GLM. Specifically, we impose the following assumption on the estimated PS coefficient $\hat{\gamma}$ computed from the training sample $\cD_{\tr}$:


\begin{condition}\label{cd:prop}
Condition~\ref{cond:nuis} (ii) holds true and there exists a constant $m_\phi$ such that for any $t$, $\phi'(t), \phi''(t) \le m_\phi$. Moreover, $\gamma^*$ has sparsity $s_\pi$ and with probability converging to $1$, the estimated $\hat{\gamma}$ satisfies that
    \[
	\|\hat{\gamma} - \gamma^*\|_1 = O \left(s_\pi \sqrt{\frac{\log p}{n}}\right) \quad\&\quad \|\hat{\gamma} - \gamma^*\|_2 = O \left(\sqrt{s_\pi \frac{\log p}{n}}\right).
	\]
\end{condition}
Once we further have that $s_\pi = o(\sqrt{n} / \log p)$ (which we will assume in the rest of this section), we have that with probability converging to $1$,
\[
\max_i |X_i^\top (\hat{\gamma} - \gamma^*)| = \max_i \|X_i\|_\infty \|\hat{\gamma} - \gamma^*\|_1 = o_\pr (1).
\]
This, together with Condition~\ref{cd:overlap}, implies that with probability converging $1$, by choosing $M_\gamma \asymp 1$ large enough, $|X_i^\top \hat{\gamma}|, |X_{\aux, i}^\top \hat{\gamma}| \le M_\gamma$, i.e., that with probability converging to $1$,  $\hat{\pi}_i \equiv \phi(X_i^\top \hat{\gamma})$ and $\hat{\pi}_{\aux, i} \equiv \phi(X_{\aux, i}^\top \hat{\gamma})$. This is helpful to derive the $\sqrt{n}$-consistency of $\hat{\tau}_\est$. 

We also impose the following ``lower bound'' constraint on $\phi'$:
%

\begin{condition}\label{cd:dphi}
	For all $u > 0$, there exists a constant $c_{\phi, u}$ depending on $u$ such that 
	\[
	\min_{|w| \le u} |\phi'(w)| \ge c_{\phi, u}.
	\]
\end{condition}

Condition~\ref{cd:dphi} is satisfied by e.g. logistic and probit regression models. With these additional constraints on the sparsity and link functions, the calibrated $\tilde{\gamma}$ satisfies that
\begin{lemma}\label{lem:gammaest}
	Suppose Conditions~\ref{cd:ignorability}--\ref{cd:overlap},~\ref{cd:all}--\ref{cd:est} and Conditions~\ref{cd:prop}--\ref{cd:dphi} hold; suppose moreover that $s_\pi = o(\sqrt{n} / \log p)$. By choosing $\eta_r, \eta_{\pi_1} \asymp \sqrt{\frac{\log p}{n}}, \eta_{\pi_2} \asymp \sqrt{\frac{\log p \vee n}{n}}, M_r, M_\pi, M_\gamma \asymp 1$ sufficiently large, we have that with probability converging to $1$, 
	\[
	\|\tilde{\gamma} - \hat{\gamma}\|_1 = O\left( s_\pi \sqrt{\frac{\log p}{n}}\right) \quad \& \quad \frac{1}{n} \sum_{i=1}^n (\tilde{X}_i^\top (\tilde\gamma - \hat{\gamma}))^2 = O\left(s_\pi \frac{\sqrt{\log p \log p \vee n}}{n}\right).
	\]
\end{lemma}


Armed with this lemma, and recall that $\hat{\tau}_\dipw$ is $\sqrt{n}$-consistent (see Appendix \ref{app:prop} for an elaboration of the analysis), 
it remains to analyze the difference between $\hat{\tau}_\est$ and $\hat{\tau}_\dipw$. We now apply the following decomposition
\begin{equation*}
    \begin{split}
	\hat{\tau}_\est  - \hat{\tau}_\dipw 
	= & \ \frac{1}{n} \sum_{i=1}^n \left(\frac{T_i (Y_i - \r_i - \hmu_i)}{\tpi_i} - \frac{T_i (Y_i - \r_i - \hmu_i)}{\hpi_i} \right) \\
	= & \ \frac{1}{n} \sum_{i=1}^n \left(\frac{T_i(Y_i - \r_i)}{\tpi_i \hpi_i}  - \frac{\hmu_i}{\hpi_i}\right) (\hpi_i - \tpi_i) + \frac{1}{n} \sum_{i=1}^n (T_i - \tpi_i) \frac{\hmu_i}{\hpi_i \tpi_i} (\hpi_i - \tpi_i).
    \end{split} 
\end{equation*}

Now by Taylor expansion, and using that with probability converging to $1$, $\hat{\pi}_i \equiv \phi(X_i^\top \hat{\gamma})$, we are able to bound the above display as follows:
\begin{equation}
\label{dc decomp}
    \begin{split}
	\hat{\tau}_\est - \hat{\tau}_\dipw
	\lesssim & \ \left|\frac{1}{n} \sum_{i=1}^n \left(\frac{T_i(Y_i - \r_i)}{\tpi_i \hpi_i} - \frac{\hmu_i}{\hpi_i}\right) \phi'(X_i^\top \hat{\gamma}) \tilde{X}_i^\top(\tilde{\gamma} - \hat{\gamma})\right| \\
 & + \left|\frac{1}{n} \sum_{i=1}^n \left(\frac{T_i}{\tpi_i} - 1\right) \frac{\phi'(X_i^\top \hat{\gamma}) \hmu_i}{\hpi_i}  \tilde{X}_i^\top (\tilde{\gamma} - \hat{\gamma})\right| + \frac{1}{n} \sum_{i=1}^n (X_i^\top(\tilde{\gamma} - \gamma^*))^2.
    \end{split} 
\end{equation}

Using an analogous analysis as the bias term in~\citet{wang2024debiased} and that $\|\hat{\gamma} - \tilde{\gamma}\|_1 \lesssim s_\pi \sqrt{\frac{\log p}{n}}$ from Lemma~\ref{lem:gammaest}, we are able to bound the first term to be of order $s_\pi \frac{\sqrt{\log p \log p \vee n}}{n}$. The third term can be controlled using Lemma~\ref{lem:gammaest} as well. Therefore the only missing piece is to control the second term of \eqref{dc decomp}. To this end, we utilize the following identity $\frac{T_i}{\tpi_i} - 1 = \frac{T_i}{\pi_i^*} - 1 + \frac{T_i}{\tpi_i} - \frac{T_i}{\pi_i^*}$ to bound the second term as below:
\begin{align*}
    & \ \left|\frac{1}{n} \sum_{i=1}^n \left(\frac{T_i}{\pi_i^*} - 1\right) \frac{\phi'(X_i^\top \hat{\gamma}) \hmu_i}{\hpi_i}  \tilde{X}_i^\top (\tilde{\gamma} - \hat{\gamma})\right| + \left|\frac{1}{n} \sum_{i=1}^n \left(\frac{T_i}{\tpi_i} - \frac{T_i}{\pi_i^*}\right) \frac{\phi'(X_i^\top \hat{\gamma}) \hmu_i}{\hpi_i}  \tilde{X}_i^\top (\tilde{\gamma} - \hat{\gamma})\right|\\
	\lesssim & \ \left\|\frac{1}{n} \sum_{i=1}^n \left(\frac{T_i}{\pi_i^*} - 1\right) \frac{\phi'(X_i^\top \hat{\gamma})}{\hpi_i} \hmu_i \tilde{X}_i \right\|_\infty \|\tilde{\gamma} - \hat{\gamma}\|_1 + \frac{1}{n} \sum_{i=1}^n (X_i^\top(\tilde{\gamma} - \gamma^*))^2
\end{align*}
where the last inequality follows from \Holder{}'s inequality and Cauchy-Schwarz inequality; for a detailed derivation, we refer readers to Appendix \ref{app:prop}. Then observing that $\{T_i / \pi_i^* - 1, i = 1, \cdots, n\}$ are mean-zero sub-Gaussian random variables and are conditionally uncorrelated with $\tilde{X}_i$'s, the first term of the above display can be shown to be of order $s_\pi \frac{\sqrt{\log p \log p \vee n}}{\sqrt{n}}$. Based on the above arguments, we obtain Theorem~\ref{thm:prop} below. 
%

\begin{theorem}\label{thm:prop}
	Suppose 
	Conditions~\ref{cd:ignorability}--\ref{cd:overlap},~\ref{cd:all}--\ref{cd:est} and Conditions~\ref{cd:prop}--\ref{cd:dphi} hold
	with $s_\pi = o(\sqrt{n} / \sqrt{\log p \log p \vee n})$, then by choosing we the tuning parameters in the same way as in Theorem~\ref{thm:reg}, we have the following representation
	\[
	\sqrt{n} (\htau_{\est} - \bar{\tau}^\ast) = \frac{1}{\sqrt{n}} \sum_{i = 1}^n \left(\frac{T_i \varepsilon_i(1)}{\pi_i^*} + \frac{(T_i - \pi_i^*) (r_i^* - \tilde{r}_i)}{\pi_i^*}\right) + o_\pr(1).
	\]
\end{theorem}
We now write 
\[
\sigma_\mu^2 \coloneqq \frac{1}{n} \sum_{i=1}^n \frac{(1 - \pi_i^*)(r_i^* - \tilde{r}_i)^2}{\pi_i^*}, \quad \bar{\sigma}_\pi^2 \coloneqq \frac{1}{n} \sum_{i=1}^n \frac{\var(\varepsilon_i(1)^2 \mid X_i)}{\pi_i^*}.
\]
Informally, Theorem~\ref{thm:prop} says that with sufficiently large $n$, with high probability, we can obtain by conditioning on $\{\bX, \cD_{\aux}, \cD_\tr\}$ that
\[
\sqrt{n} (\htau_{\est} - \bar{\tau}^\ast) \mid \{\bX, \cD_{\aux}, \cD_\tr\} \approx \calN (0, \bar{\sigma}_\pi^2 + \sigma_\mu^2).
\]

As also discussed in Theorem~\ref{thm:reg}, under some standard regularity conditions, one can expect $\bar{\sigma}_\pi$ to be $O(1)$. For $\sigma_\mu^2$, using the overlap condition as in Condition~\ref{cd:overlap} and recall the $\ell_\infty$ constraint~\eqref{eq:dipwinfinity}, one can expect $\sigma^2_\mu$ to be of order $O(1)$. This proves the $\sqrt{n}$-consistency of our estimator under a sparse PS model. This asymptotic distribution is similar to the one in~\citet{wang2024debiased}.

\begin{remark}
\label{rem:var prop}
Similar to the comment in Remark \ref{rem:var reg}, one missing piece of Theorem \ref{thm:prop} is if the asymptotic linear representation can achieve the semiparametric efficiency bound (with respect to the main dataset $\cD$). We again conjecture that it should be the case under the additional assumption that $r^\ast$ can be consistently estimated by $\tilde{r}$ when only the PS is assumed to be sparse.
\end{remark}

To conclude, Theorem \ref{thm:reg} and \ref{thm:prop}, taken together, address \ref{question} when $p > n$ as the sparsity constraint $s_\pi \ll \sqrt{n} / \sqrt{\log p \log p \vee n}$ reduces to $s_\pi \ll \sqrt{n} / \log p$. When $p \leq n$, the DCal estimator achieves $\sqrt{n}$-consistency under the \eqref{minimal sparsity} condition up to a minor $\sqrt{\log n / \log p}$ factor, which is a constant if we take $p = n^{\delta}$ for some fixed $\delta \in (0, 1)$. We conjecture that this extra log factor can be removed by calibrating the PS in a similar fashion to the OR calibration program \eqref{dipw_obj}, which we leave to future work.

We now briefly discuss the necessity of \eqref{minimal sparsity} for $\sqrt{n}$-consistency of the ATE. In terms of the matching minimax lower bounds, to the best of our knowledge, \citet{bradic2019minimax} was the first to show that \eqref{minimal sparsity} is necessary for $\sqrt{n}$-consistency for continuous $T$ when both the PS and OR models are linear. As mentioned in the Introduction, they also constructed two separate $\sqrt{n}$-consistent estimators when the PS or the OR is sufficiently sparse. At referees' request during the revision process, Appendix I of the final published version of \citet{wang2024debiased} includes a matching lower bound proof when $T$ is binary and the PS model follows a special GLM but the OR model is still linear, by adapting the proof strategy of \citet{cai2023statistical}.

\begin{remark}
The current analysis still requires that the observed data $O = (X, T, Y)$ be light-tailed. It is of great interest to investigate if the proposed approach can be extended into settings in which (part of) $O$ has heavier tails \citep{kuchibhotla2022moving} or has even diverging higher-order moments (e.g. transformed covariates by B-splines or Cohen-Daubechies-Vial wavelets \citep{belloni2015some, liu2017semiparametric, mukherjee2018optimal, liu2021adaptive})\footnote{Based on personal communications \citep{mukherjee2023sparse}, R. Mukherjee and collaborators have established nontrivial (and possibly tight) lower bounds for certain doubly-robust functionals in this context in an unpublished technical report.}.
\end{remark}

\begin{remark}
\label{rem:comp}
Last but not least, we briefly comment on the computational issue of the DCal approach at large. 
Depending on the link functions, the resulting optimization program for DCal can be non-convex. \citet{smucler2019unifying} discussed extensively the potential drawback of solving non-convex programs, such as the lack of theoretical convergence guarantees. Deriving statistical properties under computational constraint \citep{zadik2019computational} is a very important and potentially difficult problem that does not yet garner enough attention in the causal inference literature. But in order not to lead readers astray from the main message of the paper, we leave this problem to future work. Next in Section \ref{sec:sim}, for ATE, we propose a computationally feasible estimator that is motivated by the DCal methodology. Even though we do not have a complete theoretical proof, it nonetheless exhibits competitive performance against several popular state-of-the-art methods in a variety of simulation settings.
\end{remark}

\section{Applications and extensions of the Double-Calibration approach}
\label{sec:extensions}

Thus far, the new DCal approach has been delineated using the ATE under the sparse GLM nuisance models. To demonstrate the generality of the DCal approach, we now apply and extend it in several directions.

\subsection{Regression coefficient}
\label{sec:slope}

As mentioned in the Introduction, our DCal strategy can be extended to target parameters other than ATE. In this section, we demonstrate the broad applicability of our approach via another important parameter in statistics and econometrics. Specifically, we consider the following semiparametric partially linear model \citep{robinson1988root}:
\begin{equation}
\label{plm}
Y = T \tau^* + r^*(X) + \varepsilon, \quad\&\quad T = \pi^*(X) + e,
\end{equation}
for which the parameter of interest is $\tau^{\ast}$ and the nuisance parameters, with a slight abuse of notations, are again denoted as $\theta^\ast = (r^*, \pi^*)$. For simplicity of exposition we assume that $\varepsilon$ and $e$ are noises with homoscedastic variances $\sigma_\varepsilon^2, \sigma_e^2$ respectively; moreover, we assume that the noise distributions do not change with the sample size $n$. Since $\tau^\ast$ is the parameter of interest, we also regard it as a fixed constant. 

Given initial estimates $\hat{\tau}, \r$ and $\hpi$ computed from the training dataset $\cD_\tr$, just as in the ATE case, we write $\r_i \coloneqq \r(X_i), \hpi_i \coloneqq \hpi(X_i)$. We propose the following estimator for the regression coefficient $\tau^\ast$: 
\[
\hat{\tau}_\est \coloneqq \hat{\tau} + \frac{1}{n} \sum_{i=1}^n \frac{(T_i - \tpi_i)(Y_i - T_i \hat{\tau} - \r_i - \hat\mu_i)}{\tilde{\sigma}_e^2},
\]
where $\hat{\bm\mu}$ is the solution to the following $\ell_{\infty}$-calibration program:
\begin{align*}
\hat{\bm{\mu}} \coloneqq \calL_{\bm{\mu}, B_{\infty}^{n} (M_{r})} \left( \Vert \cdot \Vert_{2}^{2}; \Delta^{(\aux)}_{1, Y - T_{i} - \hat{\tau} - \hat{r}, \mu}; \phi' (X^{\top} \hat{\gamma}) X; \mathbbm{1}_{p} \eta_{r} \right)
\end{align*}
and $\tilde{\sigma}_e^2$ is an estimated variance using $\tpi$, i.e.,
\[
\tilde{\sigma}_e^2 \coloneqq \frac{1}{n} \sum_{i = 1}^n (T_i - \tpi_i)^2.
\]

\begin{remark}
Here we deliberately choose not to write $\tilde{r}$ in place of $\hat{r} + \hat{\mu}$ as the purpose of $\hat{\mu}$ is to correct the estimation error of $T \hat{\tau} + \hat{r}$ as an estimator of $\pi^\ast \tau^\ast + r^\ast$. This is slightly different from the setting of ATE in Sections \ref{sec:ate} and \ref{sec:atethm}.
\end{remark}

We now turn to the nuisance function $\tpi = \phi (\tilde{X}^{\top} \tilde{\gamma})$, which is constructed by modifying the estimator in~\eqref{ps operator} slightly as follows: 
\begin{align}\label{ps operator reg}
\tilde{\gamma} \coloneqq \calL_{\gamma, \tilde{B}_{\infty}^{\dag p} (M_{\pi})} \left( \Vert \cdot - \hat{\gamma} \Vert_{1}; \Delta_{1, T, \phi (X^{\top} \gamma)}; \mathrm{h}; \mathrm{b} \right)
\end{align}
where
\begin{align}
\tilde{B}_{\infty}^{\dag p} (M_{\pi}) & \coloneqq \left\{ \gamma: \Vert \bm{\pi} \Vert_{\infty} \leq M_{\pi} \right\} \tag{14a} \label{eq:const41}, \\
\mathrm{h} & \coloneqq \left( \psi' (X^{\top} \hat{\beta}) X^{\top}, \hat{\mu}, \tilde{X}^{\top}, \pi, \pi - T \right)^{\top} \tag{14b} \label{eq:const42}, \\
\mathrm{b} & \coloneqq \left( \frac{\eta_{\pi_1}}{\sqrt{n}} \|\psi' (X^{\top} \hat{\beta}) X \|_{n, 2}, \frac{\eta_{\pi_{1}} \Vert \hat{\bm{\mu}} \Vert_{2}}{\sqrt{n}}, \frac{\eta_{\pi_{2}}}{\sqrt{n}} \Vert \tilde{X} \Vert_{n, 2}, \frac{\eta_{\pi_{1}} \Vert \hat{\bm{\pi}} \Vert_{2}}{\sqrt{n}}, - M_{\pi}^{-1} \right)^{\top}.  \tag{14c} \label{eq:const43}
\end{align}
Specifically, we keep the objective function ~\eqref{ps operator}, change the residual to be calibrated from $\Delta_{2, T, \pi} = T / \pi - 1$ to $\Delta_{1, T, \pi} = T - \pi$, and add two extra inequality constraints to $\mathrm{h}$ and $\mathrm{b}$. In particular, the additional constraints are:
\addtocounter{equation}{+1}
\begin{equation}
\label{eq:const4}
\left|\frac{1}{n} \sum_{i = 1}^n (T_i - \pi_i) \pi_i \right| \le \eta_{\pi_1} \cdot \frac{\|{\mbb \pi}\|_2}{\sqrt{n}}, \frac{1}{n} \sum_{i=1}^n (T_i - \pi_i)^2 \ge M_\pi^{-1}.
\end{equation}

To understand the theoretical property of the DCal estimator of the regression coefficient $\tau^\ast$, we need to further modify the assumptions in Section~\ref{sec:atethm}. 
In particular, we adapt Conditions~\ref{cd:all}(i) and~\ref{cd:est} to the following:

\begin{condition} \label{cd:regcoef}
	We have that
	\begin{enumerate}
		\item[(i)] For some constant $m > 0$, almost surely the absolute value of random variables $(\varepsilon, e, r^*(X), \pi^*(X))$ are bounded below by $m$;
		
		\item[(ii)] There exists a constant $\hat{m} > 0$ such that with probability converging to $1$, $|\hat{\tau}| \le \hat{m}$, $\pr(|\psi(X^\top \hat{\beta})| \le \hat{m} \mid \hat{\beta}) = 1$ and that $\pr(|\phi(X^\top \hat{\gamma})| \le \hat{m} \mid \hat{\gamma}) = 1$, where $X$ follows the same interpretation as Condition~\ref{cd:est}.
	\end{enumerate}
\end{condition} 

%

With the above preparation, by mimicking the proofs of Theorems~\ref{thm:reg} and~\ref{thm:prop} (see Appendix \ref{app:coef}), we have the following results on the asymptotic statistical property of $\hat{\tau}_\est$ for estimating the $\tau^\ast$ in the semiparametric partially linear model \eqref{plm}. 

\begin{theorem}\label{thm:coef}
	Under Conditions~\ref{cd:all}(ii) and~\ref{cd:regcoef}, by choosing $\eta_r, \eta_{\pi_1} \asymp \sqrt{\frac{\log p}{n}}, \eta_{\pi_2} \asymp \sqrt{\frac{\log p \vee n}{n}}, M_r, M_\pi \asymp 1$ sufficiently large, we have the following:
	\begin{enumerate}
		\item[(i)] If Condition~\ref{cd:reg} holds, with the additional assumption that $|\hat{\tau} - \tau^*| = O(\sqrt{s_r \log p / n})$ and that $s_r = o(\sqrt{n} / \log p)$, we have
		\[
		\sqrt{n} (\hat{\tau}_\est - \tau^*) = \frac{1}{\sqrt{n}} \sum_{i = 1}^n \frac{(T_i - \tpi_i)\varepsilon_i}{\tilde{\sigma}_e^2} + o_\pr(1).
		\]
		\item[(ii)] If Conditions~\ref{cd:prop} and~\ref{cd:dphi} holds, then we have
		\[
		\sqrt{n} (\hat{\tau}_\est - \tau^*) = \frac{1}{\sqrt{n}} \sum_{i = 1}^n \frac{e_i \varepsilon_i}{\sigma_e^2} + \frac{1}{\sqrt{n}} \sum_{i = 1}^n \frac{e_i((\tau^* - \hat{\tau}) \pi_i^* + r_i^* - \r_i - \hat{\mu}_i)}{\sigma_e^2} + o_\pr(1).
		\]
	\end{enumerate}
\end{theorem}
We first discuss the intuitions of Theorem~\ref{thm:coef} (i). With a slight abuse of notation, we write $\bar{\sigma}_r^2 \coloneqq \frac{\var(\varepsilon)}{\hat{\sigma}_\varepsilon^2}$. Analogous to the discussions below Theorem~\ref{thm:reg}, informally, it says that
\[
\sqrt{n} (\hat{\tau}_\est - \tau^*) \mid \{(\mathbf{T}, \bX), \cD_{\aux}, \cD_\tr\} \overset{\mathsf{d}}{\approx} \calN(0, \bar{\sigma}_r^2),
\]
where the program for $\tilde{\gamma}$ ensures that with probability converging to $1$, $\bar{\sigma}_r^2$ is $O(1)$. For Theorem~\ref{thm:coef} (ii), using analogous discussions to those below Theorem~\ref{thm:prop}, we have that
\[
\sqrt{n} (\hat{\tau}_\est - \tau^*) \mid \{(\mathbf{T}, \bX), \cD_{\aux}, \cD_\tr\} \overset{\mathsf{d}}{\approx} \calN(0, \bar{\sigma}_\pi^2 + \sigma_\mu^2),
\]
where with a slight abuse of notation, we redefine 
\[
\sigma_\mu^2 \coloneqq \frac{1}{n} \sum_{i=1}^n ((\tau^* - \hat{\tau}) \pi_i^* + r_i^* - \r_i - \hat{\mu}_i)^2 \quad\&\quad \bar{\sigma}_\pi^2 \coloneqq \sigma_\varepsilon^2.
\]
Since all the summands in $\sigma_\mu^2$ are bounded with probability converging to $1$, $\sigma_\mu^2$ is with probability converging to $1$ of order $O(1)$, which in turn implies that $\hat{\tau}_\est$ is a $\sqrt{n}$-consistent estimator of $\tau^*$.

\begin{remark}
\label{rem:newey}
As alluded to in the Introduction, \citet{bradic2019minimax} also studied the problem of estimating the regression coefficient in a semiparametric partially linear model with the nuisance parameters approximated by sparse linear models. However, \citet{bradic2019minimax} constructed two separate $\sqrt{n}$-consistent estimators of $\tau^\ast$, one under the assumption $r^\ast$ is sparse, and the other under the assumption $\pi^\ast$ is sparse. In the penultimate paragraph of the main text of \citet{bradic2019minimax}, they listed as an open question to construct one estimator that is agnostic to the identity of the sparse nuisance parameters. Here we demonstrate that our new DCal estimator solves this question at least in the exact sparsity case.

Our result is also related to the vast literature on debiased lasso \citep{zhang2014confidence, van2014asymptotically, javanmard2014confidence, javanmard2018debiasing, cai2017confidence, zhu2018linear, cattaneo2018inference, cattaneo2019two, shah2023double}. We refer readers who are interested in a comparison between our result and the above references to Section 1 of \citet{bradic2019minimax}, which has comprehensively surveyed the conditions under which $\sqrt{n}$-consistent estimators of $\tau^\ast$ were developed in some of these referenced works.
\end{remark}

\subsubsection*{Further explanation on the DCal estimator for the regression coefficient through the lens of DR functionals}

We expect that it may be helpful to explain the DCal estimator for regression coefficient estimation through the lens of influence functions of DR functionals, first characterized in \citet{rotnitzky2021characterization} (also see Section 5 of \citet{chernozhukov2022locally}). The influence function of a DR functional $\tau$ with nuisance parameters $\theta = (a, b)$ has the following generic form:
\begin{align*}
\IF = H (\theta) - \tau, \text{ with } H (\theta) = S_1 a (Z) b (Z) + m_{a} (O, a) + m_{b} (O, b) + S_0,
\end{align*}
where the observable is $O = (Z, W)$, and $a \mapsto m_{a} (O, a)$ and $b \mapsto m_{b} (O, b)$ are two linear maps that satisfy some extra regularity and moment conditions. We refer readers to \citet{rotnitzky2021characterization} and \citet{liu2024assumption} for more details.

To see how this connects with the regression coefficient $\tau$ of $T$ in \eqref{plm}, we first decompose $O = (Z, Y)$ with $Z = (T, X)$. Then an influence function of $\tau$ is:
\begin{align*}
\IF = H (\theta) - \tau, \text{ with } H (\theta) = - a (Z) b (Z) + a (Z) Y + \frac{\partial}{\partial T} b (Z), \text{ and $\dfrac{\partial}{\partial T} b (Z) \equiv \tau$ by \eqref{plm},}
\end{align*}
where $a (z) \coloneqq (t - \pi (x)) \slash \mathbb{E} [(T - \pi (X))^2], b (z) \coloneqq r (x) + t \cdot \tau$. Hence as the ATE, the regression coefficient $\tau$ is also a special case of the DR functionals, as obviously $a (Z) Y$ and $\partial b (Z) / \partial T$ are linear maps of $a$ and $b$ respectively. Similarly, we denote $\theta^\ast = (a^\ast, b^\ast)$ as the true nuisance parameters. Similar to ATE, a natural estimator for $\tau$ is the DML estimator:
\begin{align*}
\hat{\tau}_{\DML} \coloneqq \hat{\tau} + \frac{1}{n} \sum_{i = 1}^{n} \hat{a} (Z_i) \left( Y_i - \hat{b} (Z_i) \right)
\end{align*}
where $\hat{b} (z) \equiv \hat{r} (x) + \hat{\tau} t \equiv \psi (x^\top \hat{\beta}) + \hat{\tau} t$ is simply some variant of the lasso estimate with the coefficients $(\hat{\tau}, \hat{\beta}^\top)^\top$ computed from the training dataset $\cD_\tr$. As for $\hat{a} (z)$, as is done in this paper, one can simply first estimate $\hat{\gamma}$ to obtain $\hat{\pi}$ and then estimate the denominator $\mathbb{E} [(T - \pi^\ast (X))^2]$ by its DML estimator $n^{-1} \sum_{i = 1}^{n} (T_i - \phi (X_i^\top \hat{\gamma}))^2$. To achieve $\sqrt{n}$-rate under $s_r \wedge s_\pi \ll \sqrt{n} / \log p$, one uses the Double-Calibration program to de-bias $\hat{a} (Z)$ and $\hat{b} (Z)$ by, respectively, $\tilde{a} (Z) = (T - \phi (X^\top \tilde{\gamma}))) / n^{-1} \sum_{i = 1}^{n} (T_i - \phi (X_i^\top \tilde{\gamma}))^2$ and $\tilde{b} (Z) = \hat{b} (Z) + \hat{\mu}$. By design of the Double-Calibration program, $\tilde{a}$ and $\tilde{b}$ again satisfy the following moment equations, similar to those in Remark \ref{rem:dipw} and \ref{rem:dc}:
\begin{align*}
\frac{1}{n} \sum_{i = 1}^{n} \left( \frac{\partial}{\partial T_i} \tilde{b}_i - a_i^\ast \tilde{b}_i \right) \text{ and } \frac{1}{n} \sum_{i = 1}^{n} \tilde{a}_i (Y_i - b_i^\ast) \text{ have mean zero}.
\end{align*}

%
%
%
%
%


\subsection{Local Average Treatment Effects (LATE) with Instrumental Variables (IV)}
\label{sec:iv}

Now we turn our attention to the setting allowing endogeneity by leveraging the availability of an Instrumental Variable (IV). Here we have access to $N$ i.i.d. observations $\{O_{i} = (X_{i}, Z_{i}, T_{i}, Y_{i}), i = 1, \cdots, N\}$, where $T \in \{0, 1\}$ is again the binary treatment and $Z \in \{0, 1\}$ is a binary IV. 
We are interested in estimating the so-called Local Average Treatment Effects (LATE) $\chi^* := \E[Y(1) \mid T(1) > T(0)]$~\citep{imbens1994identification}. Under some standard IV-identification conditions which we defer to the Appendix \ref{app:iv}, $\chi^*$ is identifiable in the form of ``2-Stage Least Squares'' (2SLS):
\begin{equation}
\label{late}
\chi^{\ast} = \frac{\E[Y(1) (T(1) - T(0))]}{\E[(T(1) - T(0))]} \coloneqq \frac{\bbE \left[ \frac{Z T Y}{\zeta^{\ast} (X)} - \frac{(1 - Z) T Y}{1 - \zeta^{\ast} (X)} \right]}{\bbE \left[ \frac{Z T}{\zeta^{\ast} (X)} - \frac{(1 - Z) T}{1 - \zeta^{\ast} (X)} \right]},
\end{equation}
where $\zeta^{\ast} (x) \coloneqq \bbE [Z | X = x] \equiv \varphi (x^{\top} \xi^{\ast})$ represents the Instrument Propensity Score (IPS).
It is worth noting that $\chi^{\ast}$ may identify ATE instead of LATE under certain alternative IV identification conditions \citep{wang2018bounded}.

To ease exposition, we consider the following parameter instead, which shares exactly the same structure as LATE:
\[
\tau^* \coloneqq \frac{\E[Y(1) T(1)]}{\E[T(1)]} \equiv \frac{\bbE \left[ \frac{Z T Y}{\zeta^{\ast} (X)} \right]}{\bbE \left[ \frac{Z T}{\zeta^{\ast} (X)} \right]} \eqqcolon \frac{\tau_n^*}{\tau_d^*}.
\]
We next introduce the following notations for nuisance parameters involved in estimating $\tau^{\ast}$:
\begin{align*}
\pi^{\ast} (x) & \coloneqq \bbE [T | X = x, Z = 1] \equiv \phi (x^{\top} \gamma^{\ast}), \\
r^{\ast} (x) & \coloneqq \bbE [Y | X = x, Z = 1, T = 1] \equiv \psi (x^{\top} \beta^{\ast}).
\end{align*}
Then applying again the standard IV-identification conditions, we have
\[
\tau^\ast_n \equiv \bbE\left[ \frac{Z (T Y - \pi^{\ast} (X) r^{\ast} (X))}{\zeta^{\ast}(X)} + \pi^{\ast} (X) r^{\ast} (X) \right] \quad\&\quad \tau^\ast_d \equiv \bbE\left[ \frac{Z (T - \pi^{\ast} (X))}{\zeta^{\ast}(X)} + \pi^{\ast} (X) \right].
\]
This gives rise to the following DML-2SLS estimator of $\tau^{\ast}$ via estimating $\tau_n^{\ast}$ and $\tau_d^{\ast}$ in \eqref{late} through their respective DML estimators \citep{tan2006regression}: $\hat{\tau}_{\DML} \coloneqq \hat{\tau}_{n, \DML} / \hat{\tau}_{d, \DML}$, where
\begin{equation*}
    \hat{\tau}_{n, \DML} \coloneqq \frac{1}{n} \sum_{i = 1}^{n} \frac{Z_{i}}{\hat{\zeta}_{i}} \left( T_{i} Y_{i} - \hat{\pi}_{i} \hat{r}_{i} \right) + \hat{\pi}_{i} \hat{r}_{i}, \quad\&\quad \hat{\tau}_{d, \DML} \coloneqq \frac{1}{n} \sum_{i = 1}^{n} \frac{Z_{i}}{\hat{\zeta}_{i}} \left( T_{i} - \hat{\pi}_{i} \right) + \hat{\pi}_{i}.
\end{equation*}
Following standard arguments in doubly robust estimation~\citep{chernozhukov2018double, liu2024assumption}, it is straightforward to see that $\hat{\tau}_{\DML}$ is $\sqrt{n}$-consistent if $s_{\zeta} \cdot (s_{\pi} + s_{r}) = o (n / \log^{2} p)$. Similar to ATE, we are interested in whether it is possible to construct $\sqrt{n}$-consistent estimators for LATE, under the laxed sparsity condition $s_{\zeta} \wedge (s_{\pi} \vee s_{r}) = o (\sqrt{n} / \log p)$. Indeed, following similar arguments to those in Section \ref{sec:ate}, we consider the following Double-Calibrated estimator for LATE: $\hat{\tau}_{\dc} \coloneqq \hat{\tau}_{n, \dc} / \hat{\tau}_{d, \dc}$, where
\begin{align*}
\hat{\tau}_{n, \dc} & \coloneqq \frac{1}{n} \sum_{i = 1}^{n} \frac{Z_{i}}{\tilde{\zeta}_{i}} \left( T_{i} Y_{i} - \hat{\pi}_{i} \hat{r}_{i} - \tilde{\mu}_{t y, i} \right) + \hat{\pi}_{i} \hat{r}_{i} + \tilde{\mu}_{t y, i}, \\
\hat{\tau}_{d, 1, \dc} & \coloneqq \frac{1}{n} \sum_{i = 1}^{n} \frac{Z_{i}}{\tilde{\zeta}_{i}} \left( T_{i} - \hat{\pi}_{i} - \tilde{\mu}_{t, i} \right) + \hat{\pi}_{i} + \tilde{\mu}_{t, i}.
\end{align*}
Here, for $i = 1, \cdots, n$, $\bm{\tilde{\mu}}_{t y} \coloneqq (\tilde{\mu}_{t y, 1}, \cdots, \tilde{\mu}_{t y, n})^{\top}$ and $\bm{\tilde{\mu}}_{t} \coloneqq (\tilde{\mu}_{t, 1}, \cdots, \tilde{\mu}_{t, n})^{\top}$ are, respectively, the solution to:
\begin{align*}
\tilde{\mbb\mu}_{t y} & \coloneqq \calL_{\bm{\mu}, B_{\infty}^{n} (M_r)} \left( \Vert \cdot \Vert_{2}^{2}; \Delta_{1, Z (T Y - \hat{\pi} \hat{r}) / \hat{\zeta}, \mu_{t y}}^{(\aux)}; \varphi' (X^{\top} \hat{\xi}) X / \hat{\zeta}; \mathbbm{1}_{p}^{\top} \eta_{t y} \right), \\
\tilde{\mbb\mu}_{t} & \coloneqq \calL_{\bm{\mu}, B_{\infty}^{n} (M_{\pi})} \left( \Vert \cdot \Vert_{2}^{2}; \Delta_{1, Z (T - \hat{\pi}) / \hat{\zeta}, \mu_{t}}^{(\aux)}; \varphi' (X^{\top} \hat{\xi}) X / \hat{\zeta}; \mathbbm{1}_{p}^{\top} \eta_{t} \right).
\end{align*}
Similarly, $\tilde{\zeta}_i \coloneqq \varphi (\tilde{\xi}^{\top} \tilde{X}_i)$, with $\tilde{\xi}$ as the solution to:
\begin{align*}
\tilde{\xi} \coloneqq \calL_{\xi, \tilde{B}_{\infty}^{p} (M_{\zeta})} \left( \Vert \cdot - \hat{\xi} \Vert_{1}; \Delta_{2, Z, \varphi (X^{\top} \xi)}; \mathrm{h}; \mathrm{b} \right),
\end{align*}
where
\begin{align*}
\tilde{B}_{\infty}^{p} (M_{\zeta}) & \coloneqq \left\{ \gamma: \Vert Z / \varphi (\tilde{X}^{\top} \xi) \Vert_{\infty} \leq M_{\zeta} \right\}, \\
\mathrm{h} & \coloneqq \left( \phi' (X^{\top} \hat{\gamma}) X^{\top}, \psi' (X^{\top} \hat{\beta}) \hat{\pi} X^{\top}, \phi' (X^{\top} \hat{\gamma}) \hat{r} X^{\top}, \tilde{\mu}_{t}, \tilde{\mu}_{t y}, \tilde{X}^{\top} \right)^{\top}, \\
\text{and } \mathrm{b} & \coloneqq \left( \frac{\eta_{\zeta_1}}{\sqrt{n}} \|\psi' (X^{\top} \hat{\beta}) X \|_{n, 2}, \frac{\eta_{\zeta_1} \Vert \tilde{\bm{\mu}}_{t} \Vert_{2}}{\sqrt{n}}, \frac{\eta_{\zeta_1} \Vert \tilde{\bm{\mu}}_{t y} \Vert_{2}}{\sqrt{n}}, \frac{\eta_{\zeta_2}}{\sqrt{n}} \Vert \tilde{X} \Vert_{n, 2} \right)^{\top}.
\end{align*}

Following exactly the same analysis as Theorems \ref{thm:reg} and \ref{thm:prop}, it can be proven that by choosing $\eta_{t}, \eta_{ty}, \eta_{\zeta_1} \asymp \sqrt{\frac{\log p}{n}}, \eta_{\zeta_2} \asymp \sqrt{\frac{\log p \vee n}{n}}$, and $M_{r}, M_{\pi}, M_{\zeta} \asymp 1$ sufficiently large, $\hat{\tau}_{\dc}$ is an $\sqrt{n}$-consistent estimator for $\tau^\ast$ if either $s_{\zeta} = o(\sqrt{n} / \sqrt{\log p \log p \vee n})$ or $s_{\pi} \vee s_{r} = o(\sqrt{n} / \log p)$.

\subsection{Extension to approximately sparse GLMs}
\label{sec:approximate sparsity}

In previous sections, we have assumed that at least one of the nuisance parameters, $\pi^{\ast}$ or $r^{\ast}$, exactly follows a sparse GLM, with sparsity below $\sqrt{n} / \log p$. \citet{belloni2014pivotal, chernozhukov2018double, bradic2019minimax} and \citet{smucler2019unifying} considered the so-called approximately sparse GLMs to model the nuisance parameters. Similar to Condition \ref{cond:nuis}, the condition below imposes that at least one of the two nuisance parameters, $r^\ast$ or $\pi^\ast$, follows the approximately sparse GLM (note that here we adopt the definition of approximately sparse GLMs from \citet{bradic2019minimax}).
\begin{condition}\label{cond:nuis as}\leavevmode
There exist two (nonlinear) monotonically increasing, twice-differentiable link functions $\phi, \psi$ with uniformly bounded first and second derivatives, such that either of the following holds.
\begin{enumerate}
\item[(i)] The OR model can be approximated by a GLM with (nonlinear) link $\psi$ with the following properties: $p$ is sufficiently large compared to $n$ such that there exists $\beta_{p}^{\ast}$ for which
\begin{equation}
\label{tail or}
\sqrt{n} \left\{ \mathbb{E} \left[ (r^{\ast} (X) - \psi (\beta_{p}^{\ast \top} X))^{2} \right] \right\}^{1 / 2} = o (1),
\end{equation}
and for any $s$, there exists $\xi_{r} > 0$ and $\beta_{s}^{\ast}$ such that $\Vert \beta_{s}^{\ast} \Vert_0 \leq s$ and $\Vert \beta_{s}^{\ast} - \beta_{p}^{\ast} \Vert_{2} \ll s^{- \xi_{r}}$;
\item[(ii)] The PS model can be approximated by a GLM with (nonlinear) link $\phi$ with the following properties: $p$ is sufficiently large compared to $n$ such that there exists $\gamma_{p}^{\ast}$ for which
\begin{equation}
\label{tail ps}
\sqrt{n} \left\{ \mathbb{E} \left[ (\pi^{\ast} (X) - \phi (\gamma_{p}^{\ast \top} X))^{2} \right] \right\}^{1 / 2} = o (1),
\end{equation}
and for any $s$, there exists $\xi_{\pi} > 0$ and $\gamma_{s}^{\ast}$ such that $\Vert \gamma_{s}^{\ast} \Vert_0 \leq s$ and $\Vert \gamma_{s}^{\ast} - \gamma_{p}^{\ast} \Vert_{2} \ll s^{- \xi_{\pi}}$.
\end{enumerate}
\end{condition}

We silence the dependence on $s$ by denoting $\beta^{\ast}$ and $\gamma^{\ast}$ when $s = p$. To simplify the notation, we let $\bar{r}^{\ast} (\cdot) \coloneqq \psi (\beta^{\ast \top} \cdot)$ and $\bar{\pi}^{\ast} (\cdot) \coloneqq \phi (\gamma^{\ast \top} \cdot)$. \citet{bradic2019minimax} proved the following sufficient and necessary condition for the existence of $\sqrt{n}$-consistent estimator of the ATE $\tau$ when $p$ is sufficiently large compared to $n$ (see Assumption 9, 10 and the statement of Theorem 8 of \citet{bradic2019minimax}, to ensure that the approximation bias in $L_{2} (\mathbb{P})$-norm by all $p$ covariates is sufficiently small): 
\begin{equation}\label{minimal as}
\xi_{r} \vee \xi_{\pi} > 1 / 2.
\end{equation} 
As a corollary of the results in Section \ref{sec:atethm}, the DCal estimator achieves $\sqrt{n}$-consistency under \eqref{minimal as}, together with several other conditions to be stated below. We first modify the regularity conditions on the OR and PS coefficient estimates computed from the training dataset $\cD_{\tr}$ as follows. Condition \ref{cd:reg} needs to be replaced by the following:
\begin{condition}\label{cd:reg as}
Condition~\ref{cond:nuis as} (i) holds true and there exists a constant $m_\psi$ such that for any $t$, $\psi'(t), \psi''(t) \le m_\psi$. Moreover, when $\xi_{r} > 1 / 2$, with probability converging to $1$, there exists $\hat{\beta}$ computed from $\cD_{\tr}$ such that
 \[
    \|\hat{\beta} - \beta^*\|_1 = O \left( \left( \sqrt{\frac{n}{\log p}} \right)^{- \frac{2 \xi_{r} - 1}{2 \xi_{r} + 1}} \right) \quad\&\quad \|\hat{\beta} - \beta^*\|_2 = O \left( \left( \sqrt{\frac{n}{\log p}} \right)^{- \frac{2 \xi_{r}}{2 \xi_{r} + 1}} \right).
\]
\end{condition}
Similarly, Condition \ref{cd:prop} needs to be modified as
\begin{condition}\label{cd:prop as}
Condition~\ref{cond:nuis as} (ii) holds true and there exists a constant $m_\phi$ such that for any $t$, $\phi'(t), \phi''(t) \le m_\phi$. Moreover, when $\xi_{\pi} > 1 / 2$, with probability converging to $1$, there exists $\hat{\gamma}$ computed from $\cD_{\tr}$ such that
 \[
    \|\hat{\gamma} - \gamma^*\|_1 = O \left( \left( \sqrt{\frac{n}{\log p}} \right)^{- \frac{2 \xi_{\pi} - 1}{2 \xi_{\pi} + 1}} \right) \quad\&\quad \|\hat{\gamma} - \gamma^*\|_2 = O \left( \left( \sqrt{\frac{n}{\log p}} \right)^{- \frac{2 \xi_{\pi}}{2 \xi_{\pi} + 1}} \right).
\]
\end{condition}

\begin{remark}
The justification for Condition \ref{cd:reg as} and Condition \ref{cd:prop as} can be found in Appendix \ref{app:nuis as}. In fact, Condition \ref{cd:reg as} and Condition \ref{cd:prop as} closely resemble Condition \ref{cd:reg} and Condition \ref{cd:prop}, respectively, which are standard results of lasso. In the latter case, the $\ell_{1}$-norm convergence rates are multiples of the $\ell_{2}$-norm convergence rates by a factor of the square root of the sparsity $\sqrt{s_r}$ or $\sqrt{s_\pi}$, while in the former case, the $\ell_{1}$-norm convergence rates are multiples of the $\ell_{2}$-norm convergence rates by a factor of $(\sqrt{n / \log p})^{1 / (2 \xi_{r} + 1)}$ or $(\sqrt{n / \log p})^{1 / (2 \xi_{\pi} + 1)}$, which can be interpreted as the square root of the ``effective sparsity'' in the approximately sparse GLMs (also see Section 1 of \citet{bradic2019minimax}).
\end{remark}

Finally, we state the following corollary of Theorem \ref{thm:reg} and \ref{thm:prop} in Section \ref{sec:atethm} that demonstrates the $\sqrt{n}$-consistency of $\hat{\tau}_{\est}$ under \eqref{minimal as}. The proof is deferred to Appendix \ref{app:as}.
\begin{corollary}\leavevmode
\label{cor:as}
We have the following:
\begin{enumerate}[label = (\roman*)]
\item The statement of Theorem \ref{thm:reg} still holds, with Condition \ref{cd:reg} replaced by Condition \ref{cd:reg as}, and $s_{r} = o (\sqrt{n / \log p})$ replaced by $\xi_{r} > 1 / 2$;
\item The statement of Theorem \ref{thm:prop} still holds, with Condition \ref{cd:prop} replaced by Condition \ref{cd:prop as}, and $s_{\pi} = o (\sqrt{n / \log p})$ replaced by $\xi_{\pi} > 1 / 2$.
\end{enumerate}
\end{corollary}

\section{Simulation studies}
\label{sec:sim}

As discussed in Remark \ref{rem:comp}, the actual DCal approach studied theoretically in previous sections might not be computationally feasible. In this section, to alleviate the computational burden, we consider a modified DCal estimator for ATE, simply by changing the PS calibration step originally described in \eqref{ps operator} to:
\begin{equation}
\label{modified dcal}
\bm{\tilde{\pi}} \coloneqq \calL_{\bm{\pi}, \tilde{B}_{\infty}^{p} (M_{\pi})} \left( \Vert \cdot \Vert_{2}^{2}; \Delta_{2, T, \pi}; \mathrm{h}; \mathrm{b} \right)
\end{equation}
with the same $\mathrm{h}$ and $\mathrm{b}$. This modified PS calibration step leads to easier implementation and faster algorithmic convergence {\it in practice}. It also mirrors closely the OR calibration step \eqref{sc operator}, in which one does not parameterize the calibrated OR estimator as GLMs. To lighten the notation, we use $\hat{\tau}_{\dc}$ to denote the modified DCal estimator \textit{only} in this section.

\begin{remark}[More comments on the modified DCal estimator]
Instead of calibrating the estimated PS regression coefficients $\hat{\gamma}$ as in the original DCal estimator, the modified DCal estimator directly calibrates the PS model. By speculating the proof of Theorem~\ref{thm:prop}, the original DCal estimator, by construction of the objective function, controls $\Vert \hat{\gamma} - \tilde{\gamma} \Vert_{1}$ when $\gamma^{\ast}$ is in fact sparse, which is crucial when controlling the difference between the DCal estimator $\hat{\tau}_{\dc}$ and the SCal estimator $\hat{\tau}_{\dipw}$ without calibrating the PS model (when unnecessary). It is still an open question whether minimizing $\Vert \bm{\tilde{\pi}} \Vert_{2}^{2}$ directly can also imply the conclusions in Theorem~\ref{thm:prop}. We emphasize that the main purpose of the simulation studies conducted here is to demonstrate that the essential idea of Double Calibration can be used to construct estimators with similar practical performance to competing methods detailed below.
\end{remark}

To study the practical merit of the modified DCal estimator, we compare it against the following competing methods: the IPW estimator $\hat{\tau}_{\IPW}$ \citep{horvitz1952generalization}, the estimator $\hat{\tau}_{g}$ based on the g-formula \citep{robins1986new}, DML/AIPW estimator $\hat{\tau}_{\DML}$ \citep{robins1994estimation, hahn1998role, chernozhukov2018double}, the Targeted Maximum Likelihood Estimation (TMLE) estimator $\hat{\tau}_{\TMLE}$ \citep{van2006targeted, van2011targeted}, the Approximate Residual Balancing (ARB) estimator $\hat{\tau}_{\ARB}$ \citep{athey2018approximate}, the Regularized Calibrated (RCal) estimator $\hat{\tau}_{\RCAL}$ \citep{tan2020regularized}, and the high-dimensional Covariate Balancing Propensity Score (hdCBPS) estimator $\hat{\tau}_{\cbps}$ \citep{ning2020robust}. More details on implementing these competing methods can be found in Appendix \ref{app:details}. We also set up a GitHub repository (\href{https://github.com/Cinbo-Wang/DCal}{https://github.com/Cinbo-Wang/DCal}) for readers who are interested in replicating our simulation results.

We consider two different settings to mirror our theoretical results: Section \ref{sec:sparse ps} below considers sparse PS and dense OR models, whilst Section \ref{sec:sparse or} considers the opposite setting. In both cases, the observed data is comprised of $N$ i.i.d. samples $\{X_{i}, T_{i}, Y_{i}\}_{i = 1}^{N}$, with $X \sim \calN_{p} (0, \Sigma)$, where $\Sigma_{i, j} \equiv 0.9^{|i - j|}$ for $i, j = 1, \cdots, p$. We consider the following combinations of $(n, p)$:
$$
(n, p) \in \{(200, 400), (400, 800), (800, 1000), (1600, 1000)\},
$$ 
with $s \in \{10, 20, 50\}$ for each possible $(n, p)$ pair. For each setting, 200 Monte Carlos are sampled to get the summary statistics such as (absolute) biases, $\sqrt{N}$-biases, standard errors, mean squared errors, coverage probabilities, and etc. and all the simulated data are generated anew in each run. Variance estimators of different ATE estimators are simply based on their respective influence functions.

\subsection{Case I: Sparse PS \& Dense OR}
\label{sec:sparse ps}

First, we consider the case where the PS is a sparse logistic-linear model and the OR is a dense nonlinear model:
\begin{align*}
& T | X = x \sim \text{Bernoulli} \left( \pi (x) = \psi (x^{\top} \gamma) \right), \\
& Y | X = x, T = t \sim \calN \left( r (x, t), 1 \right),
\end{align*}
where $\psi (z) \equiv 1 / (1 + \exp (-z))$ in this section. In particular, we choose $\gamma$ to satisfy the following: $\Vert \gamma \Vert_{2} = 1$, $\gamma_j \sim \mathrm{Uniform} ([1, 2])$ for $j = 1, \cdots, s$ and $\gamma_{s + 1} = \cdots = \gamma_{p} = 0$ so only the first $s$ coordinates of $\gamma$ are active. For the outcome model $r (x, t)$, we first transform the covariates $X$ nonlinearly as follows:
\begin{align*}
& \tilde{X}_{1} \equiv \texttt{bs} (X_{1}, 100)^{\top} \cdot \left( 1, 1 / 2, 1 / 3, \cdots, 1 / 100 \right), \tilde{X}_{2} \equiv 2 / (1 + \exp (- X_{2})), \\
& \tilde{X}_{3} \equiv \exp (X_{3} / 2), \tilde{X}_{4} \equiv X_{4} / (1 + \exp (X_{3})), \tilde{X}_{5} = X_{4} X_{5} / 10,
\end{align*}
and then scale $\tilde{X}_{2}, \cdots, \tilde{X}_{5}$ to have zero mean and unit variance. Here $\texttt{bs} (x, 100)$ denotes the B-spline transformation of $x$ with 100 degrees-of-freedom. We then define the outcome model $r (x, t)$ as
\begin{align*}
r (x, t) \equiv \left( |\tilde{x}_{1} + \tilde{x}_{2} - \tilde{x}_{3} / 2 + \tilde{x}_{4} / 3 - \tilde{x}_{5} / 4| + 0.05 \right)^{-1} + x^{\top} \beta - t,
\end{align*}
where $\Vert \beta \Vert_{2} \equiv 1$ and $\beta_{j} \propto j^{-1}$ for $j = 1, \cdots, p$.

The simulation results are displayed in Figure \ref{fig: simu_sparse_PS}, Table \ref{tab: comp_time_cont}, Table \ref{tab: sparse_PS_measure_CI} (see Appendix \ref{app:sparse ps}), and Table \ref{tab: sparse_PS_measure} (see Appendix \ref{app:sparse ps}). We summarize the numerical results below:
\begin{itemize}
\item Reading from root-N scaled absolute estimation errors from Figure \ref{fig: simu_sparse_PS}, as expected, the g-formula estimator $\hat{\tau}_{g}$ generally has the largest bias as the initial OR estimator without calibration should not even be a consistent estimator. Interestingly, although the PS model is sparse, $\hat{\tau}_{\dc}$ still has better performance than the IPW and AIPW estimators, possibly due to its calibration step.

 $\hat{\tau}_{\TMLE}$, $\hat{\tau}_{\ARB}$, and $\hat{\tau}_{\cbps}$ perform similarly to $\hat{\tau}_{\dc}$, but the performance of $\hat{\tau}_{\TMLE}$ deteriorates when the sparsity of the PS model is relatively large (the rightmost panel of Figure \ref{fig: simu_sparse_PS}). $\hat{\tau}_{\ARB}$ generally has slightly larger bias than $\hat{\tau}_{\dc}$, but the difference is not very pronounced. The performance of $\hat{\tau}_{\cbps}$ deteriorates as the sample size becomes smaller even when the sparsity of the PS model is low, but when the sample size is large, $\hat{\tau}_{\cbps}$ performs as good as $\hat{\tau}_{\dc}$, and sometimes better than $\hat{\tau}_{\dc}$ when the sparsity of the PS model is very low (the bottom left corner of Figure \ref{fig: simu_sparse_PS}).
 
In this setting, $\hat{\tau}_{\RCAL}$ does not perform as well as $\hat{\tau}_{\TMLE}$, $\hat{\tau}_{\ARB}$, $\hat{\tau}_{\cbps}$ and $\hat{\tau}_{\dc}$; 
we conjecture that this is due to the relatively low overlap in propensity score and in Appendix \ref{app:sparse ps} we indeed show that $\hat{\tau}_{\RCAL}$ has much more improved performance when the violation of the positivity assumption is alleviated.
Similar conclusion can be reached by reading from Tables \ref{tab: sparse_PS_measure_CI} and \ref{tab: sparse_PS_measure}, which respectively report more detailed numerical values of the Mean/Median Absolute Biases (MAB), standard errors, and Root-Mean-Squared-Errors (RMSE), and the coverage probabilities and lengths of 95\% Wald confidence intervals centered around different ATE estimators.
\item Reading from the left panel of Table \ref{tab: comp_time_cont}, it is quite obvious that $\hat{\tau}_{\dc}$ enjoys shorter computation time than $\hat{\tau}_{\cbps}$ and $\hat{\tau}_{\RCAL}$.
\end{itemize}

\begin{figure}[htpb]
    \centering
    \includegraphics[width = \textwidth, page=1]{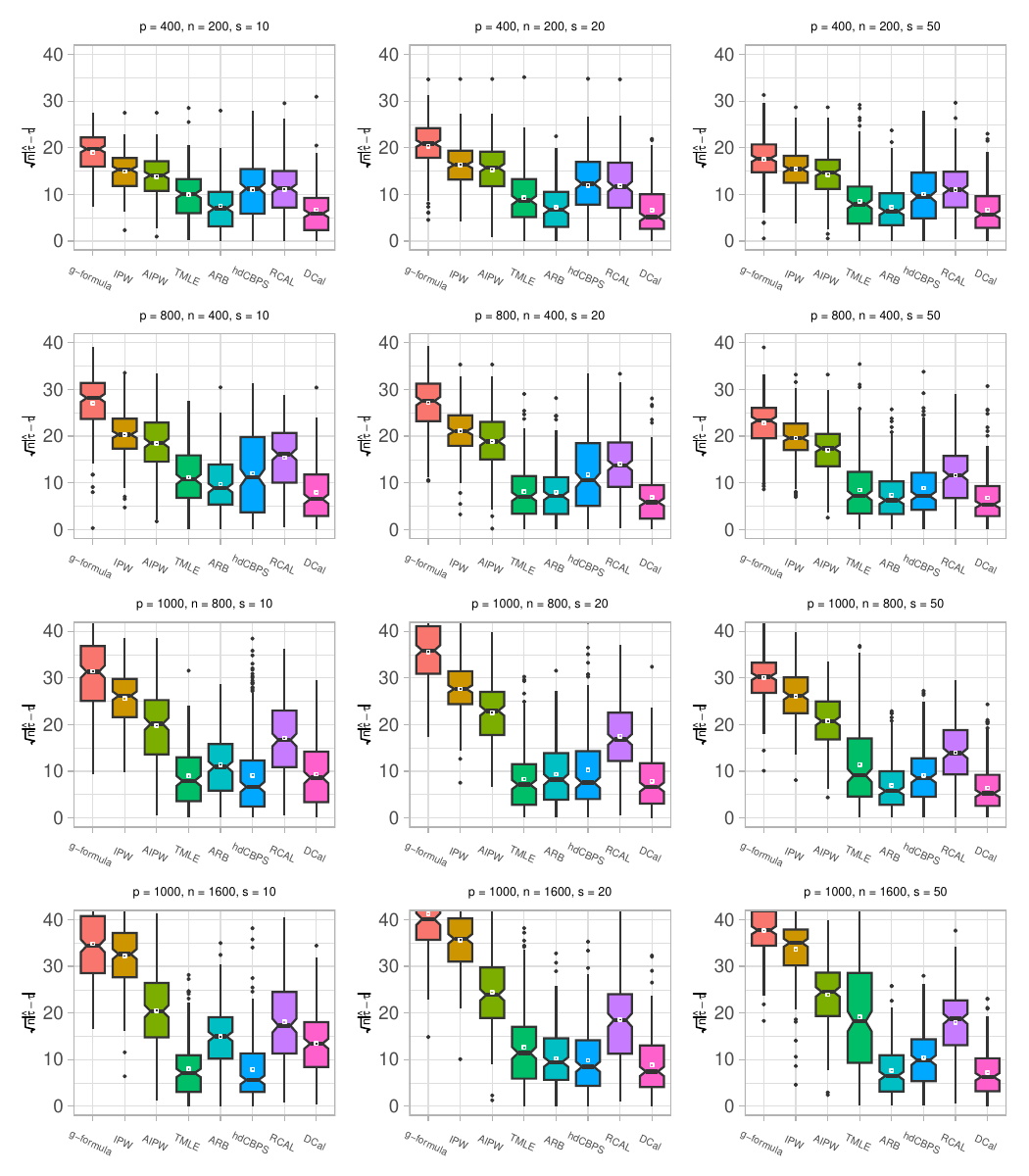}
    \caption{Boxplots of the $\sqrt{n}$-scaled estimation error $\sqrt{n}|\hat{\tau}-\tau|$ under different sample size $n$, dimension $p$ and sparsity level $s$ under the \textit{Sparse PS \& Dense OR} setting. The white dots correspond to mean.}
    \label{fig: simu_sparse_PS}
\end{figure}

\subsection{Case II: Dense PS \& Sparse OR}
\label{sec:sparse or}

Next, we consider the case where the PS is a dense nonlinear model bounded between 0 and 1 and the OR is a sparse linear model:
\begin{align*}
& T | X = x \sim \mathrm{Bernoulli} \left( \pi (x) \right), \\
& Y | X = x, T = t \sim \calN \left( r (x, t), 1 \right).
\end{align*}
Here we take the PS model to be
\begin{align*}
\pi (X) \equiv \min\{\max\{\psi (\check{X}_{1} - \frac{1}{2} \check{X}_{2} + \frac{1}{4} \check{X}_{3} - \frac{1}{8} \check{X}_{4} + \gamma^{\top} X), 0.05\}, 0.95\},
\end{align*}
where $\check{X}_{1} \coloneqq e^{0.5 X_{1}}$, $\check{X}_{2} \coloneqq 10 + X_{2} / (1 + e^{X_{1}})$, $\check{X}_{3} \coloneqq (0.05 X_{1} X_{3} + 0.6)^{2}$, $\check{X}_{4} \coloneqq (X_{2} + X_{4} + 10)^{2}$, and $\gamma_{j} \propto j^{-1}$ with $\Vert \gamma \Vert_{2} \equiv 2$. We take the OR model to be
\begin{align*}
r (X, t) \equiv \beta^{\top} X - 1 + 2 t
\end{align*}
where we first draw $\beta_{j} \sim \mathrm{Uniform} ([1, 2])$ for $j = 1, \cdots, s$ and set $\beta_{j} \equiv 0$ for $j = s + 1, \cdots, p$. We eventually normalize $\beta$ such that $\Vert \beta \Vert_{2} = 1$.

The simulation results are displayed in Figure \ref{fig: simu_sparse_OR}, Table \ref{tab: comp_time_cont}, Table \ref{tab: sparse_OR_measure_CI} (see Appendix \ref{app:sparse or}), and Table \ref{tab: sparse_OR_measure} (see Appendix \ref{app:sparse or}). We summarize the numerical results below:
\begin{itemize}
\item Reading from root-N scaled absolute estimation errors from Figure \ref{fig: simu_sparse_OR}, as expected, the IPW estimator $\hat{\tau}_{\IPW}$ generally has the worst bias as the initial PS estimator without calibration is not even a consistent estimator. Interestingly, although the OR model is sparse, $\hat{\tau}_{\dc}$ still has better performance than the g-formula and AIPW estimators, possibly due to its calibration step.

$\hat{\tau}_{\TMLE}$, $\hat{\tau}_{\ARB}$, and $\hat{\tau}_{\cbps}$ perform similarly to $\hat{\tau}_{\dc}$, but the performance of $\hat{\tau}_{\TMLE}$ again deteriorates when the sparsity of the OR model is relatively large (the rightmost panel of Figure \ref{fig: simu_sparse_OR}). In this case, $\hat{\tau}_{\ARB}$ and $\hat{\tau}_{\cbps}$ perform as well as $\hat{\tau}_{\dc}$. This is not very surprising at least for $\hat{\tau}_{\ARB}$, since it is proved that $\hat{\tau}_{\ARB}$ is $\sqrt{n}$-consistent in this regime. 
 
$\hat{\tau}_{\RCAL}$ does not perform as well as $\hat{\tau}_{\ARB}$, $\hat{\tau}_{\cbps}$ and $\hat{\tau}_{\dc}$: the performance of $\hat{\tau}_{\TMLE}$ recovers when the sample size increases, 
whereas for the performance of $\hat{\tau}_{\RCAL}$ we again conjecture that this is due to the relatively large $\ell_{2}$-norm of $\gamma$ (recall that $\Vert \gamma \Vert_{2} = 2$) and in Appendix \ref{app:sparse or} we indeed show that $\hat{\tau}_{\RCAL}$ has much more improved performance when $\Vert \gamma \Vert_{2}$ decreases. 
It is worth noting that by looking at the results across different sample sizes, $\hat{\tau}_{\RCAL}$ is still close to being $\sqrt{n}$-consistent empirically and the main difference between $\hat{\tau}_{\RCAL}$ and other estimators is in the constant.

Similar conclusion can be reached by reading from Tables \ref{tab: sparse_OR_measure_CI} and \ref{tab: sparse_OR_measure}, which respectively report the coverage probabilities and lengths of 95\% Wald confidence intervals centered around different ATE estimators, and more detailed numerical values of the Mean/Median Absolute Biases (MAB), standard errors, and Root-Mean-Squared-Errors (RMSE).
\item Reading from the right panel of Table \ref{tab: comp_time_cont}, it is quite obvious that $\hat{\tau}_{\dc}$ again enjoys much shorter computation time than $\hat{\tau}_{\cbps}$ and $\hat{\tau}_{\RCAL}$.
\end{itemize}

\begin{figure}[htpb]
    \centering
    \includegraphics[width = \textwidth, page=1]{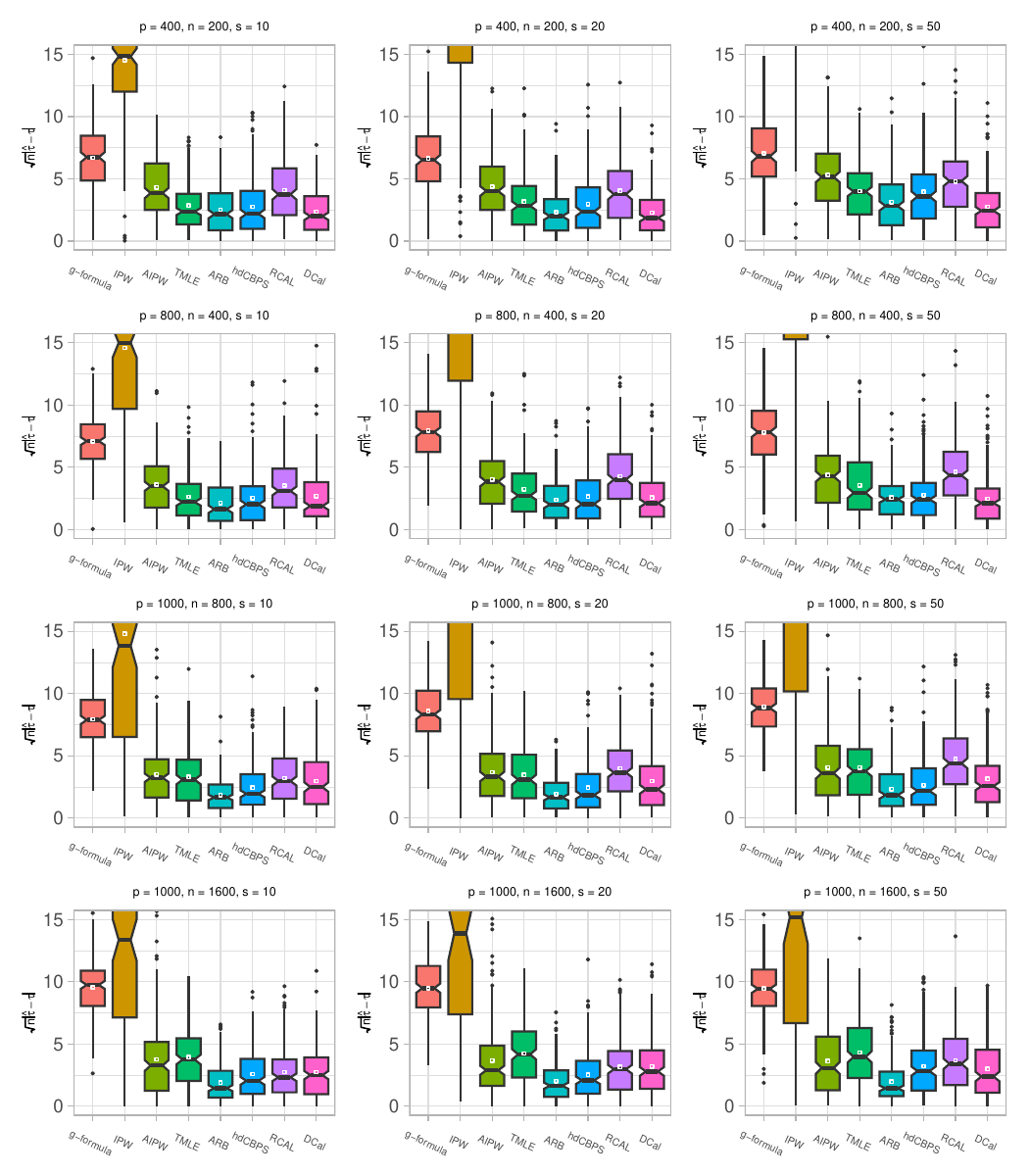}
    \caption{Boxplots of the $\sqrt{n}$-scaled estimation error $\sqrt{n}|\hat{\tau}-\tau|$ under different sample size $n$, dimension $p$ and sparsity level $s$ under the \textit{Dense PS \& Sparse OR} setting. The white dots correspond to means.}
    \label{fig: simu_sparse_OR}
\end{figure}

\begin{longtable}{@{}cccllllllll@{}}
\caption{Computation time (in minutes) of different methods averaged over 200 Monte Carlos.}
\label{tab: comp_time_cont}\\
\toprule
$n$ & $p$ & $s$ & \multicolumn{4}{c}{\textbf{Sparse PS \& Dense OR}} & \multicolumn{4}{c}{\textbf{Dense PS \& Sparse OR}} \\* \cmidrule(lr){4-7}
\cmidrule(lr){8-11}
 &  &  & \multicolumn{1}{c}{\textbf{hdCBPS}} & \multicolumn{1}{c}{\textbf{RCAL}} & \multicolumn{1}{c}{$\textbf{RCAL}^\star$\footnotemark} & \multicolumn{1}{c}{\textbf{DCal}} & \multicolumn{1}{c}{\textbf{hdCBPS}} & \multicolumn{1}{c}{\textbf{RCAL}} & \multicolumn{1}{c}{$\textbf{RCAL}^\star$} & \multicolumn{1}{c}{\textbf{DCal}} \\* \midrule
\endfirsthead
\multicolumn{11}{c}%
{{\bfseries Table \thetable\ continued from previous page}} \\
\endhead
\bottomrule
\endfoot
\endlastfoot
200 & 400 & 10 & 0.08 & 0.62 & 0.16 & 0.02 & 1.40 & 0.45 & 0.11 & 0.02 \\
200 & 400 & 20 & 0.10 & 0.59 & 0.15 & 0.02 & 2.06 & 0.46 & 0.12 & 0.02 \\
200 & 400 & 50 & 0.06 & 0.45 & 0.11 & 0.02 & 2.08 & 0.47 & 0.12 & 0.02 \\
400 & 800 & 10 & 0.10 & 4.33 & 0.97 & 0.06 & 3.10 & 2.71 & 0.55 & 0.07 \\
400 & 800 & 20 & 0.20 & 5.15 & 1.01 & 0.06 & 4.52 & 2.67 & 0.51 & 0.07 \\
400 & 800 & 50 & 0.05 & 4.81 & 0.88 & 0.06 & 4.55 & 2.67 & 0.51 & 0.08 \\
800 & 1000 & 10 & 0.64 & 60.77 & 6.84 & 0.16 & 7.47 & 24.41 & 2.72 & 0.19 \\
800 & 1000 & 20 & 0.30 & 43.52 & 7.18 & 0.16 & 10.19 & 24.30 & 2.36 & 0.21 \\
800 & 1000 & 50 & 0.19 & 48.06 & 6.01 & 0.16 & 10.09 & 24.24 & 2.26 & 0.18 \\
1600 & 1000 & 10 & 3.37 & 437.60 & 22.48 & 0.36 & 22.58 & 254.20 & 5.99 & 0.44 \\
1600 & 1000 & 20 & 1.87 & 513.17 & 24.26 & 0.34 & 27.20 & 248.58 & 5.03 & 0.41 \\
1600 & 1000 & 50 & 0.62 & 328.59 & 26.79 & 0.34 & 27.19 & 244.39 & 4.58 & 0.43 \\* \bottomrule
\end{longtable}
\footnotetext{$\text{RCAL}^\star$ refers to RCAL with a five-fold cross-validation and eleven possible values for the tuning parameters. Its performance is slightly inferior to the RCAL presented in the main text, although the difference is not statistically significant. For the sake of clarity, the statistical performance of $\text{RCAL}^\star$ is not reported for the settings in the main text (see \href{https://github.com/Cinbo-Wang/DCal}{the accompanied GitHub website} for more details).}

\section{Concluding remarks}
\label{sec:conclusions}

In this paper, we present a novel methodology called Double-Calibration, which produces $\sqrt{n}$-consistent and asymptotic normal estimators for ATE and regression coefficient in high-dimensional semiparametric partially linear models under the \eqref{minimal sparsity} condition. There are several problems that are worth considering in future works. First, as mentioned in the end of the Introduction, we only construct $\sqrt{n}$-consistent and asymptotic normal estimators for ATE. Semiparametric efficiency bound following Le Cam's approach under the \eqref{minimal sparsity} condition remains an open question, but ideas in \citet{mukherjee2020minimax} might be promising directions to explore. Second, here we only considered high-dimensional sparse GLMs. It is an open question to bridge the following five regimes in high-dimensional statistics: the sparse but $p \gg n$ regime (our case), the sparse but $p \leq n$ regime, the dense but $1 \ll p \ll n$ regime \citep{liu2017semiparametric, liu2023hoif, su2023estimated}, the dense but proportional asymptotic regime ($p / n \rightarrow c \in (0, \infty)$) \citep{yadlowsky2022explaining, jiang2022new, celentano2023challenges, chen2024method}, and beyond ($p \gg n$ but no sparsity) \citep{robins2008higher, robins2023minimax, liu2021adaptive}. \citet{verzelen2018adaptive} and \citet{kong2018estimating} have made some attempt at this problem for quadratic functionals. Third, our framework might also be extended to more complex causal inference settings, such as high-dimensional mediation analysis \citep{lin2023testing} or high-dimensional proximal causal learning \citep{deaner2021many}. Finally, we point out another interesting observation that warrants further study. When the nuisance parameters are modeled as sparse high-dimensional GLMs, at least for ATE, we have seen that $\sqrt{n}$-consistent estimation can be achieved by the proposed Double-Calibration strategy. However, in the context of classical nonparametric-type models (e.g. \Holder{} class) for the nuisance parameters, the only known $\sqrt{n}$-consistent estimators for ATE under minimal condition is based on higher-order influence functions \citep{robins2008higher, liu2017semiparametric, liu2023hoif}. It remains to be seen if similar idea to that developed in this paper can also be applied to the nonparametric setting or vice versa.




\section*{Acknowledgments}
The authors are grateful for the extremely insightful discussions and comments by \href{https://scholar.harvard.edu/rajarshi/home}{Rajarshi Mukherjee} and Jamie Robins. The simulations in this paper were run on the Siyuan-1 cluster supported by the Center for High Performance Computing at Shanghai Jiao Tong University. LL is affiliated with the Shanghai Artificial Intelligence Laboratory and The Smart Justice Lab at the Koguan Law School of Shanghai Jiao Tong University. LL gratefully acknowledges funding support by NSFC Grant No.12471274 and No.12090024, Shanghai Municipal Science and Technology Grant No.21ZR1431000, No.2021SHZDZX0102, and No.21JC1402900.

\bibliographystyle{plainnat}
\bibliography{Master}

\appendix

\section{Proofs of the results concerning Section~\ref{sec:atethm}}




\subsection{Proof of Lemma~\ref{lem:program}}

\begin{proof}[Proof of Lemma~\ref{lem:program}]
Let $W_i \coloneqq T_i / \pi_i^* - 1$. Then by conditioning on $\{\tilde{\bX}, \hat{\mbb \mu}, \hat{\gamma}, \hat{\beta}\}$, the $W_i$'s are independent sub-Gaussian random variables with variance proxy bounded below by a constant $(1 - c_\pi) / c_\pi$. From this, as a direct consequence of the Chernoff bound and the fact that almost surely, $T_i / \pi_i^* \le c_\pi^{-1}$ since we are under Condition~\ref{cd:overlap}, there exist constants $c_1, c_2, c'> 0$ and $M_\pi \ge c_\pi^{-1}$ sufficiently large such that by choosing $\eta_{\pi_1} \le c_1 \sqrt{\frac{\log p}{n}}, \eta_{\pi_2} \le c_2 \sqrt{\frac{\log p \vee n}{n}}$ for any realization of $\{\tilde{\bX}, \hat{\mbb \mu}, \hat{\gamma}, \hat{\beta}\}$, 
\[
\pr(\textrm{\eqref{eq:const1}--\eqref{eq:const3} hold with $\pi_i$ replaced by $\pi_i^*$} \mid \{\tilde{\bX}, \hat{\mbb \mu}, \hat{\gamma}, \hat{\beta}\}) \le c' n^{-m}.
\]

Second, from Condition~\ref{cd:overlap}, we have that all the $\pi_i^*$'s are in the domain space of the link function $\phi(\cdot)$, which means that the vector $\phi^{-1}({\mbb \pi}^*) := (\phi^{-1}(\pi^*_1), \ldots, \phi^{-1}(\pi^*_n))^\top$ is well defined. Moreover, from the construction of $\tilde{\mathbf X}$ and Condition~\ref{cd:all}(ii), it is immediate that the matrix $\tilde{\mathbf X}$ is almost surely full row rank, so that almost surely there is at least one $\gamma$ satisfying
\[
\tilde{\mathbf X} \gamma = \phi^{-1}({\mbb \pi}^*).
\]
Putting together, we prove the desired result.
\end{proof}

\subsection{Proof of Lemma~\ref{lem:mu}}
\label{app:mu}
Let
\[
\hat\mu_{\ora, i} \coloneqq \frac{\pi_i^* (r_i^* - \hat{r}_i)}{\hpi_i}, \hat{\bmu}_{\ora} \coloneqq (\hat\mu_{\ora, 1}, \cdots, \hat\mu_{\ora, n})^{\top}.
\]
We then have the following Lemma.

\begin{lemma}\label{lem:feasible}
Suppose 
Conditions~\ref{cd:ignorability}--\ref{cd:overlap} and \ref{cd:all}--\ref{cd:est}
hold and let the constant $M_\gamma$ be given; then by choosing $\eta_r \asymp \sqrt{\frac{\log p}{n}}, M_r \asymp 1$ large enough, with probability converging to $1$, $\hat{\mbb \mu}_{\ora}$ is a feasible solution to~\eqref{eq:dipwprogram}--\eqref{eq:dipwinfinity}.
\end{lemma}

\begin{proof}
Suppose $\cD_\tr$ is under the event defined by Condition~\ref{cd:est}, we have that by conditioning on $\cD_\tr$, the $j$-th index of the vector ${\mb X}_\aux^\top {\mb \Pi}_\aux \tilde{\mb R}_\aux$ is a sum of $n$ i.i.d. sub-exponential random variables, so that by Bernstein's inequality and a union bound, we have that given any $m \in \mathbb{N}$, there exist positive constants $c_1$ -- $c_4$ depending on $m$ such that when $\cD_\tr$ is under the event defined by Condition~\ref{cd:est},
\[
\pr\left(\left\|\frac{1}{n_{\aux}} {\bX}_{\aux}^\top {\mbb \Pi}_{\aux} \tilde{\mathbf R}_{\aux} - \E\left[\frac{1}{n_{\aux}} {\bX}_{\aux}^\top {\mbb \Pi}_{\aux} \tilde{\mathbf R}_{\aux} \mid \cD_\tr\right]\right\|_\infty \ge c_1 \sqrt{\frac{\log p}{n_{\aux}}} \mid \cD_\tr\right) \le c_2 p^{-m},
\]
and
\[
\pr\left(\left\|\frac{1}{n} {\bX}^\top {\mbb \Pi} \hat{\mbb \mu}_{\ora} - \E\left[\frac{1}{n} {\bX}^\top {\mbb \Pi} \hat{\mbb \mu}_{\ora} \mid \cD_\tr\right]\right\|_\infty \ge c_3 \sqrt{\frac{\log p}{n}} \mid \cD_\tr\right) \le c_4 p^{-m}.
\]
Since $\cD$ and $\cD_\aux$ are i.i.d. sampled from the same distribution, we can obtain that almost surely, 
\begin{equation*}
\E\left[\frac{1}{n} {\bX}^\top {\mbb \Pi} \hat{\mbb \mu}_{\ora} \mid \cD_\tr\right] = \E\left[\frac{1}{n_{\aux}} {\bX}_{\aux}^\top {\mbb \Pi}_{\aux} \tilde{\mathbf R}_{\aux} \mid \cD_\tr\right].
\end{equation*}
In light of the above three inequalities and a union bound, we prove that \eqref{eq:dipwprogram} can be satisfied with probability converging to $1$ by choosing $\eta_r \asymp \sqrt{\frac{\log p}{n}}$ sufficiently large.

The control of~\eqref{eq:dipwinfinity} is straightforward by simply setting $M_r := \frac{m_r + m_Y}{\phi(-M_\gamma)}$.
\end{proof}

\begin{proof}[Proof of Lemma~\ref{lem:mu}]
	Suppose that we are under the event where Lemma~\ref{lem:feasible} holds true, then it is easy to see that almost surely,
 \begin{equation}
\label{exact unbiased}
\|\hat{\mbb \mu}\|_2 \le \|\hat{\mbb \mu}_{\ora}\|_2 \lesssim \sqrt{\sum_{i=1}^n (X_i^\top (\hat{\beta} - \beta^*))^2},
 \end{equation}
 where the last inequality follows from (i) the definition of $\hat{\pi}$ and (ii) that $\psi'(\cdot)$ is uniformly bounded since we are under Condition~\ref{cd:reg}. 
 
 Now applying a Bernstein's inequality~\citep[Lemma 2.7.7]{vershynin2018high}, we have that there exists constants $c_1, c_2 > 0$ such that almost surely,
 \begin{equation}\label{eq:betax}
     \pr\left(\|\mb X (\hat{\beta} - \beta^*)\|_2 / \sqrt{n} \ge c_1 \|\hat{\beta} - \beta^*\|_2 \mid \hat{\beta}\right) \le e^{-c_2 n}.
 \end{equation}
 In light of both inequalities, the desired result follows from $s_r = o( \sqrt{n} / \log p)$.
\end{proof}

\subsection{Proof of Theorem~\ref{thm:reg}}
\label{app:reg}
\begin{proof}
Using Taylor expansion, we have that for some $\eta_i \in [0, 1]$ for $i = 1, \cdots, n$, the following decomposition holds 
\begin{equation}
\label{reg expansion}
\begin{split}
\hat{\tau}_\est & - \bar{\tau}^\ast \\
= & \ \underset{\mathrm{I}}{\underbrace{\frac{1}{n} \sum_{i=1}^n \left( 1 - \frac{T_i}{\tpi_i} \right) \psi'(X_i^\top \hat{\beta}) X_i^\top (\hat{\beta} - \beta^*) }} \\
& + \underset{\mathrm{II}}{\underbrace{\frac{1}{n} \sum_{i=1}^n \left( 1 - \frac{T_i}{\tpi_i} \right)  \psi''(\eta_i X_i^\top \hat{\beta} + (1 - \eta_i) X_i^\top \beta^*) \left(X_i^\top (\hat{\beta} - \beta^*)\right)^2 }} \\
& + \underset{\mathrm{III}}{\underbrace{\frac{1}{n} \sum_{i=1}^n \left( 1 - \frac{T_i}{\tpi_i} \right) \hmu_i}} + \frac{1}{n} \sum_{i=1}^n \frac{T_i \varepsilon_i(1)}{\tpi_i}.
\end{split}
\end{equation}
We first consider $\mathrm{I}$. Invoking Lemma~\ref{lem:program}, we have that the constraint~\eqref{eq:const1} holds with probability converging to $1$. In light of this and Condition~\ref{cd:reg}, we have that with probability converging to $1$,
	\begin{align*}
		|\mathrm{I}| & \leq \left\| \frac{1}{n} \sum_{i=1}^n \left(\frac{T_i}{\tpi_i} - 1\right) \psi'(X_i^\top \hat{\beta}) X_i\right\|_\infty \|\hat{\beta} - \beta^*\|_1 \\
		& \lesssim s_r \frac{\log p}{n} \cdot \max_j \frac{\|\psi'(\bX \hat{\beta}) \odot \bX_j\|_2}{\sqrt{n}} \lesssim  s_r \frac{\log p}{n} \cdot \max_j \frac{\|\bX_j\|_2}{\sqrt{n}},
	\end{align*}
 where for the last inequality we use the assumption that $\psi'$ is uniformly bounded since we are under Condition~\ref{cd:reg}. Under Condition~\ref{cd:all}(ii), we further have that almost surely, $\max_j \|\bX_j\|_2 / \sqrt{n} \lesssim 1$. Putting these together and recalling that $s_r = o(\sqrt{n} / \log p)$, we can conclude that $|\mathrm{I}| = o_\pr (n^{-1/2})$.
 
For $\mathrm{II}$, since Conditions~\ref{cd:est} and~\ref{cd:reg} hold, together with constraint~\eqref{eq:const3}, we have that, with probability converging to $1$,  
 \[
 \mathrm{II} \lesssim \frac{1}{n} \sum_{i=1}^n \left(X_i^\top (\hat{\beta} - \beta^*)\right)^2.
 \]
 In light of this upper bound and recall~\eqref{eq:betax}, we have that with probability converging to $1$, we have that $\mathrm{II} \lesssim \|\hat{\beta} - \beta^*\|_2^2$. Then it immediately follows that $\mathrm{II} = o_\pr(n^{-1/2})$ using again Condition~\ref{cd:reg} and that $s_r = o(\sqrt{n} / \log p)$.
 
	For $\mathrm{III}$, as a direct consequence of \eqref{eq:const2} and the conclusion of Lemma~\ref{lem:mu}, we have that
	\[
	|\mathrm{III}| = o_\pr\left(\frac{1}{\sqrt{n}} \cdot \sqrt{\frac{\log p}{\sqrt{n}}}\right) = o_\pr (n^{-1/2}).
	\]
	Given the above controls of $\mathrm{I}$, $\mathrm{II}$ and $\mathrm{III}$, the desired result follows.
\end{proof}

\subsection{Proof of Lemma~\ref{lem:gammaest}}
\label{app:gammaest}
\begin{proof}
Without loss of generality, throughout this section we assume that $p \ge n$, so that $\tilde{X}_i$ and $X_i$ are equivalent. The proof of the case $p < n$ can be obtained via applying an analogous argument. From Lemma~\ref{lem:program}, we have that with probability converging to $1$, $\gamma^*$ is a feasible solution to~\eqref{eq:program} when the tuning parameters are chosen according to Lemma~\ref{lem:program}, so that with probability converging to $1$,
	\[
	\|\tilde{\gamma} - \hat{\gamma}\|_1 \le \|\gamma^* - \hat{\gamma}\|_1 \lesssim s_\pi \sqrt{\frac{\log p}{n}},
	\]
	where for the last inequality we invoke Condition~\ref{cd:prop}. Moreover,
 \begin{equation}\label{eq:tildel1}
     \|\tilde{\gamma} - \gamma^*\|_1 \le \|\tilde{\gamma} - \hat{\gamma}\|_1 + \|\hat{\gamma} - \gamma^*\|_1 \lesssim s_\pi \sqrt{\frac{\log p}{n}}.
 \end{equation}
	
 We now focus on proving the second result. Working under the event given by Lemma~\ref{lem:program}, using that we are under constraint~\eqref{eq:constprop}, we have that
\[
\left\|\frac{1}{n} \sum_{i=1}^n (T_i / \tpi_i - 1) \tilde{X}_i^\top \right\|_\infty \lesssim \sqrt{\frac{\log p \vee n}{n}},
\]
and that the above inequality still holds with $\tpi_i$ replaced by $\pi_i^*$. Putting these together yields that under the event given by Lemma~\ref{lem:program},
\begin{equation}\label{eq:proginf}
\left\|\frac{1}{n} \sum_{i=1}^n (T_i / \tpi_i - T_i / \pi_i^*) \tilde{X}_i^\top \right\|_\infty \lesssim \sqrt{\frac{\log p \vee n}{n}}.
\end{equation}

In light of~\eqref{eq:tildel1} and~\eqref{eq:proginf}, and by applying \Holder's inequality, we have that with probability converging to $1$,
\[
\left|\frac{1}{n} \sum_{i=1}^n (T_i / \tpi_i - T_i / \pi_i^*) \tilde{X}_i^\top (\tilde\gamma - \gamma^*)\right| \le \left\|\frac{1}{n} \sum_{i=1}^n (T_i / \tpi_i - T_i / \pi_i^*) \tilde{X}_i^\top \right\|_\infty \|\tilde\gamma - \gamma^*\|_1 \lesssim s_\pi \frac{\sqrt{\log p \log p \vee n}}{n}.
\]
Using Taylor expansion we have that there exists $\iota_1, \cdots, \iota_n \in [0, 1]$ such that the left hand side in the above inequality is equal to 
\[
\frac{1}{n} \sum_{i=1}^n \frac{T_i}{\tpi_i \pi_i^*} \phi'(\iota_i X_i^\top \gamma^* + (1 - \iota) X_i^\top \tilde\gamma) (X_i^\top (\tilde\gamma - \gamma^*))^2.
\]

From Conditions~\ref{cd:all}(ii) and~\eqref{eq:tildel1}, we have that with probability converging to $1$, for all $i$, $|X_i^\top (\tilde\gamma - \gamma^*)| \le \|X_i\|_\infty \|\tilde\gamma - \gamma^*\|_1 \lesssim 1$. Using Condition~\ref{cd:overlap} and the constraint on the link function provided by Condition~\ref{cond:nuis}(ii), we have that almost surely $|X_i^\top \gamma^*| \lesssim 1$ as well. From this, and the lower bound in Condition~\ref{cd:dphi}, we have that there exists a constant $c > 0$ such that with probability converging to $1$, for all $i$, 
\[
\phi'(\iota_i X_i^\top \gamma^* + (1 - \iota) X_i^\top \tilde\gamma) \ge c,
\]
from which we have
\[
\frac{1}{n} \sum_{i=1}^n \frac{T_i}{\tpi_i \pi_i^*} (X_i^\top (\tilde\gamma - \gamma^*))^2 \lesssim \frac{1}{n} \sum_{i=1}^n \frac{T_i}{\tpi_i \pi_i^*} \phi'(\iota_i X_i^\top \gamma^* + (1 - \iota) X_i^\top \tilde\gamma) (X_i^\top (\tilde\gamma - \gamma^*))^2 \lesssim s_\pi \frac{\sqrt{\log p \log p \vee n}}{n}.
\]
This, in combination with ~\eqref{eq:const3}, yields
\[ 
\frac{1}{n} \sum_{i=1}^n \frac{T_i}{\pi_i^*} (X_i^\top (\tilde\gamma - \gamma^*))^2 \lesssim s_\pi \frac{\sqrt{\log p \log p \vee n}}{n}.
\]

Finally, recall that we have~\eqref{eq:betax} and Condition~\ref{cd:overlap}, we have that with probability converging to $1$,
\begin{equation}\label{eq:hatbern}
	\frac{1}{n} \sum_{i=1}^n \frac{T_i}{\pi_i^*} (X_i^\top (\hat{\gamma} - \gamma^*))^2 \lesssim \frac{1}{n} \sum_{i=1}^n (X_i^\top (\hat{\gamma} - \gamma^*))^2 \lesssim s_\pi \frac{\log p}{n}.
\end{equation}
Putting the above results together yields 
\[
\frac{1}{n} \sum_{i=1}^n \frac{T_i}{\pi_i^*} (X_i^\top (\tilde\gamma - \hat{\gamma}))^2 \lesssim s_\pi \frac{\sqrt{\log p \log p \vee n}}{n}.
\]

For the final missing piece, we first observe the following:
\begin{align*}
    \frac{1}{n} \sum_{i=1}^n (\tilde{X}_i^\top (\tilde{\gamma} - \gamma^*))^2 & = \frac{1}{n} \sum_{i=1}^n T_i / \pi_i^*(\tilde{X}_i^\top (\tilde{\gamma} - \gamma^*))^2 + (\tilde{\gamma} - \gamma^*)^\top \left( \frac{\tilde{\bX}^\top \tilde{\bX}}{n} - \frac{1}{n} \sum_{i=1}^n \frac{T_i}{\pi_i^*} \tilde{X}_i \tilde{X}_i^\top  \right) (\tilde{\gamma} - \gamma^*) \\
    & \lesssim s_\pi \frac{\sqrt{\log p \log p \vee n}}{n} + \left\|\frac{\tilde\bX^\top \tilde\bX}{n} - \frac{1}{n} \sum_{i=1}^n \frac{T_i}{\pi_i^*} \tilde{X}_i \tilde{X}_i^\top \right\|_\infty \|\tilde{\gamma} - \gamma^* \|_1^2,
\end{align*}
where the for the last inequality we apply \Holder's inequality.
By Hoeffding's inequality and a union bound, with probability converging to $1$, 
\[
\left\|\frac{\tilde\bX^\top \tilde\bX}{n} - \frac{1}{n} \sum_{i=1}^n \frac{T_i}{\pi_i^*} \tilde{X}_i \tilde{X}_i^\top \right\|_\infty \lesssim \sqrt{\frac{\log p \vee n}{n}}.
\]
From above, we conclude that with probability converging to $1$,
\[
\frac{1}{n} \sum_{i=1}^n (\tilde{X}_i^\top (\tilde{\gamma} - \gamma^*))^2 \lesssim s_\pi \frac{\sqrt{\log p \log p \vee n}}{n} + s_\pi^2 \frac{\log p}{n} \sqrt{\frac{\log p \vee n}{n}} \lesssim s_\pi \frac{\sqrt{\log p \log p \vee n}}{n},
\]
where the last inequality follows because $s_\pi = o(\sqrt{n} / \log p)$. In light of this and~\eqref{eq:tildel1}, the stated conclusion follows.
\end{proof}


\subsection{Proof of Theorem~\ref{thm:prop}}
\label{app:prop}

We begin with the following lemma on the negligible effect of trimming in $\hat{\pi}$ when $s_\pi = o(\sqrt{n} / \log p)$:
\begin{lemma}\label{lem:pitrim}
Under the conditions of Theorem~\ref{thm:prop}, we have that by choosing $M_\gamma \asymp 1$ sufficiently large, with probability converging to $1$, for all $i$,
\[
\hat{\pi}_i \equiv \phi(X_i^\top \hat{\gamma}) \quad\&\quad \hat{\pi}_{\aux, i} \equiv \phi(X_{\aux, i}^\top \hat{\gamma}).
\]
\end{lemma}

\begin{proof}
Without loss of generality, we only prove the case for $\hat{\pi}_i$.
From Condition~\ref{cd:overlap} and using the monotonicity of the link function $\phi$, we have that almost surely, there exists a $M_\gamma^*$ such that
\[
|X_i^\top \gamma^*| \le M_\gamma^*.
\]
Using \Holder's inequality, we have
\[
\max_i |X_i^\top \hat{\gamma} - X_i^\top \gamma^*| \le \max_i \|X_i\|_\infty \|\hat{\gamma} - \gamma^*\|_1 = o_\pr(1),
\]
where for the last inequality we use Condition~\ref{cd:prop} and that $s_\pi = o(\sqrt{n} / \log p)$. Thus the statement of this lemma is proved.
\end{proof}

Armed with the above lemma, we can see that in the large sample limit, $\hat{\tau}_\dipw$ is numerically equivalent to the estimator by redefining $\hat{\pi} \coloneqq \phi(X^\top \hat{\gamma})$, which allows us to apply the same analysis in \citet{wang2024debiased} to understand the asymptotic property of $\hat{\mbb \mu}$. Specifically, we have

\begin{lemma}\label{lem:infinities}
	Under the conditions of Theorem~\ref{thm:prop}, with probability converging to $1$, we have the following:
	\[
	\left\|\frac{1}{n} \sum_{i=1}^n \left(\frac{T_i}{\pi_i^*} - 1\right) \frac{\hmu_i}{\hpi_i} \phi'(X_i^\top \gamma^*) \tilde{X}_i \right\|_\infty \lesssim \sqrt{\frac{\log p \vee n}{n}}
	\]
	and
	\begin{equation}\label{eq:partialinfinity}
		\left\|\frac{1}{n} \sum_{i=1}^n \left(\frac{T_i (Y_i - \r_i)}{\hpi_i^2} - \frac{\hmu_i}{\hpi_i}\right) \phi'(X_i^\top \hat{\gamma}) \tilde{X}_i \right\|_\infty \lesssim \sqrt{\frac{\log p \vee n}{n}}.
	\end{equation}
\end{lemma}

\begin{proof}
	By conditioning on $(\tilde{\bX}, \hat{\mbb \mu})$, $\{\frac{T_i}{\pi_i^*} - 1, i = 1, \cdots, n\}$ of the main dataset $\cD$ are independent mean-zero and bounded random variables. Then as a direct consequence of Hoeffding's inequality, we obtain that for some constants $c_1, c_2 > 0$, almost surely
 \[
 \pr\left(\left\|\frac{1}{n} \sum_{i=1}^n \left(\frac{T_i}{\pi_i^*} - 1\right) \frac{\hmu_i}{\hpi_i} \phi'(X_i^\top \gamma^*) \tilde{X}_i \right\|_\infty \le c_1  \max_i \frac{\hat{\mu}_i}{\hat{\pi}_i} \sqrt{\frac{\log p \vee n}{n}} \;\bigg|\; \tilde{\bX}, \hat{\mbb \mu} \right) \le c_2 (p \vee n)^{-1}.
 \]
 Using above, the first inequality follows from Lemma~\ref{lem:mu}.
 
 For the second inequality, from Lemmas~\ref{lem:feasible} and~\ref{lem:pitrim}, we get, with probability converging to $1$, 
	\[
	\left\|\frac{1}{n_{\aux}} \sum_{i=1}^{n_{\aux}} \frac{T_{i, \aux} (Y_{i, \aux} - \r_{i, \aux})}{\hpi_{i, \aux}^2} \phi'(X_{i, \aux}^\top \hat{\gamma}) X_{i, \aux} - \frac{1}{n} \sum_{i=1}^n \frac{\hmu_i}{\hpi_i} \phi'(X_i^\top \hat{\gamma}) X_i \right\|_\infty \lesssim \sqrt{\frac{\log p}{n}}.
	\]
It also follows from the proof of Lemma~\ref{lem:feasible} that with probability converging to $1$, 
\[
	\left\|\frac{1}{n_{\aux}} \sum_{i=1}^{n_{\aux}} \frac{T_{i, \aux} (Y_{i, \aux} - \r_{i, \aux})}{\hpi_{i, \aux}^2} \phi'(X_{i, \aux}^\top \hat{\gamma}) X_{i, \aux} - \frac{1}{n} \sum_{i=1}^n \frac{T_i (Y_i - \r_i)}{\hpi_i^2} \phi'(X_i^\top \hat{\gamma}) X_i \right\|_\infty \lesssim \sqrt{\frac{\log p}{n}}.
\]

In light of the above inequalities, and using that the augmented entries of $\tilde{X}_i$ are i.i.d. random variables drawn uniformly from $[-1, 1]$ and are independent of $(\bX, \cD_{\aux}, \cD_\tr)$, the second inequality is again a direct consequence of Hoeffding's inequality.
\end{proof}

\begin{proof}[Proof of Theorem~\ref{thm:prop}] 
Without loss of generality, in this proof we assume throughout that $p \ge n$, so that $\bX$ and $\tilde{\bX}$ are equivalent. The proof of the case $p < n$ can be derived via an analogous argument. To prove the desired statement, it boils down to proving the following term to be $o_\pr (1)$.
\begin{equation}
\label{Q}
Q \coloneqq \frac{1}{\sqrt{n}} \sum_{i=1}^n \left(\frac{T_i (Y_i - \r_i - \hmu_i)}{\tpi_i} - \frac{T_i (Y_i - \r_i - \hmu_i)}{\pi_i^*}\right).
\end{equation}
We apply the following decomposition:
\begin{align}
	Q & = \frac{1}{\sqrt{n}} \sum_{i=1}^n \left(\frac{T_i (Y_i - \r_i)}{\pi_i^* \tpi_i} - \frac{\hmu_i}{\tpi_i}\right) (\pi_i^* - \tpi_i) - \frac{1}{\sqrt{n}} \sum_{i=1}^n \left(\frac{T_i}{\pi_i^*} - 1\right) \frac{\hmu_i}{\tpi_i} (\pi_i^* - \tpi_i) \notag \\
 & = \underset{\mathrm{I}}{\underbrace{\frac{1}{\sqrt{n}} \sum_{i=1}^n \left(\frac{T_i (Y_i - \r_i)}{\pi_i^* \tpi_i} - \frac{\hmu_i}{\tpi_i}\right) (\hpi_i - \tpi_i)}} + \underset{\mathrm{II}}{\underbrace{\frac{1}{\sqrt{n}} \sum_{i=1}^n \left(\frac{T_i (Y_i - \r_i)}{\pi_i^* \tpi_i} - \frac{\hmu_i}{\tpi_i}\right) (\pi_i^* - \hpi_i)}} \\
 & \quad - \underset{\mathrm{III}}{\underbrace{\frac{1}{\sqrt{n}} \sum_{i=1}^n \left(\frac{T_i}{\pi_i^*} - 1\right) \frac{\hmu_i}{\tpi_i} (\pi_i^* - \tpi_i)}}. \notag
 \label{eq:prop}
\end{align}
We first consider $\mathrm{I}$, which we further decompose as
\begin{align*}
    \mathrm{I} & = \underset{\mathrm{I}_1}{\underbrace{\frac{1}{\sqrt{n}} \sum_{i=1}^n \left(\frac{T_i (Y_i - \r_i)}{\hpi_i^2} - \frac{\hmu_i}{\hpi_i}\right) (\hpi_i - \tpi_i)}} + \underset{\mathrm{I}_2}{\underbrace{\frac{1}{\sqrt{n}} \sum_{i=1}^n \left(\frac{T_i (Y_i - \r_i)}{\pi_i^* \tpi_i} - \frac{T_i (Y_i - \r_i)}{\hpi_i^2}\right) (\hpi_i - \tpi_i)}} \\
    &\quad - \underset{\mathrm{I}_3}{\underbrace{\frac{1}{\sqrt{n}} \sum_{i=1}^n \left(\frac{\hmu_i}{\hpi_i} - \frac{\hmu_i}{\tpi_i}\right) (\hpi_i - \tpi_i)}}
\end{align*}
Using Taylor expansion and Lemma~\ref{lem:pitrim}, we have that for some $\iota_i \in [0, 1]$ for $i = 1,\ldots,n$, 
\begin{align*}
	\mathrm{I}_1 = & \ \frac{1}{\sqrt{n}} \sum_{i=1}^n \left(\frac{T_i (Y_i - \r_i)}{\hpi_i^2} - \frac{\hmu_i}{\hpi_i}\right) \phi'(X_i^\top \hat{\gamma}) X_i^\top(\hat{\gamma} - \tilde{\gamma}) \\
	& + \frac{1}{\sqrt{n}} \sum_{i=1}^n \left(\frac{T_i (Y_i - \r_i)}{\hpi_i^2} - \frac{\hmu_i}{\hpi_i}\right) \phi''(\iota_i X_i^\top \hat{\gamma} + (1 - \iota_i) X_i^\top \tilde{\gamma}) (X_i^\top(\hat{\gamma} - \tilde{\gamma}))^2.
\end{align*}
For the second term, using that with probability converging to $1$, $\frac{T_i (Y_i - \r_i)}{\hpi_i^2}, \frac{\hmu_i}{\hpi_i}$ are all bounded below by some constant and that $\phi''$ is uniformly bounded, as well as Lemma~\ref{lem:gammaest} and $s_\pi = o(\sqrt{n} / \sqrt{\log p \log p \vee n})$, we have the second term in the above decomposition is of order $o_\pr(1)$. For the first term, applying \Holder's inequality, it is a direct consequence of Lemmas~\ref{lem:gammaest} and~\ref{lem:infinities}.

For $\I_2$ and $\I_3$, from Lemma~\ref{lem:pitrim}, using mean value theorem and that $\phi'(\cdot)$ is uniformly bounded, we get with probability converging to $1$, for all $i$,
\[
|\hpi_i - \pi_i^*| \lesssim \left|X_i^\top(\hat{\gamma} - \gamma^*)\right| \quad\&\quad |\hpi_i - \tpi_i| \lesssim \left|X_i^\top(\hat{\gamma} - \tilde{\gamma})\right|.
\]
Moreover, using that $\hat{\gamma}$ and the main data set $\cD$ are independent, Condition~\ref{cd:prop} and a Bernstein's inequality (see e.g.~\eqref{eq:betax}), we can prove that with probability converging to $1$,
\[
\frac{1}{n} \sum_{i = 1}^n (X_i^\top(\hat{\gamma} - \gamma^*))^2 = O\left(s_\pi \frac{\log p}{n}\right).
\]
In light of the above two inequalities, the control of $\I_2$ and $\I_3$ follows from the same analysis as the second term in the decomposition of $\I_1$. Putting together we prove that $\I = o_\pr(1)$. Using an analogous argument we also prove that $\I\I = o_\pr(1)$.

It now remains to consider $\I\I\I$. Using Taylor expansion, we have that for some $\iota_i' \in [0, 1]$ for $i = 1,\ldots,n$, 
\begin{align*}
	\I\I\I = & \ \frac{1}{\sqrt{n}} \sum_{i=1}^n \left(\frac{T_i}{\pi_i^*} - 1\right) \frac{\hmu_i}{\tpi_i} \phi'(X_i^\top \gamma^*) X_i^\top (\gamma^* - \tilde{\gamma}) \\
	& + \frac{1}{\sqrt{n}} \sum_{i=1}^n \left(\frac{T_i}{\pi_i^*} - 1\right) \frac{\hmu_i}{\hpi_i} \phi''(\iota_i' X_i^\top \gamma^* + (1 - \iota_i') X_i^\top \tilde{\gamma}) (X_i^\top (\gamma^*- \tilde{\gamma}))^2.
\end{align*}
Applying an analogous analysis as above, we can show that the second terms in the above two displays are of order $o_\pr (1)$. Thus, to prove the desired result, it remains to show that the first terms in the above two displays are also of order $o_\pr (1)$. Applying \Holder's inequality, this is a direct consequence of Lemmas~\ref{lem:gammaest} and~\ref{lem:infinities}.
\end{proof}

\section{Proofs of the results concerning Section~\ref{sec:extensions}}

\subsection{Proof of Theorem~\ref{thm:coef}}
\label{app:coef}
\subsubsection{Proof of Theorem~\ref{thm:coef} (i)}

Using the same analysis as in Lemma~\ref{lem:program}, there exists some $\eta_{\pi_1} \asymp \sqrt{\frac{\log p}{n}}, \eta_{\pi_2} \asymp \sqrt{\frac{\log p \vee n}{n}}, M_\pi \asymp 1$ such that with probability converging to $1$, the $\tilde{\gamma}^*$ such that $\pi_i^* \equiv \phi(\tilde{X}_i^\top \tilde\gamma^*)$ for all $i$ is a feasible solution to~\eqref{ps operator reg}. We now apply the decomposition
\begin{align*}
	\hat{\tau}_\est - \tau^* = & \ \underset{\eqqcolon \, \mathrm{I}}{\underbrace{(\hat{\tau} - \tau^*) \left(1 - \frac{\sum_{i=1}^n T_i (T_i - \tpi_i)}{\sum_{i=1}^n (T_i - \tpi_i)^2}\right)}} + \underset{\eqqcolon \, \mathrm{II}}{\underbrace{\frac{\sum_{i=1}^n (T_i - \tpi(X_i)) (r(X_i) - \r(X_i))}{\sum_{i=1}^n (T_i - \tpi(X_i))^2}}} \\
	 & - \underset{\eqqcolon \, \mathrm{III}}{\underbrace{\frac{\sum_{i=1}^n (T_i - \tpi(X_i)) \hat{\mu}_i}{\sum_{i=1}^n (T_i - \tpi(X_i))^2}}} + \frac{\sum_{i=1}^n (T_i - \tpi(X_i)) \varepsilon_i}{\sum_{i=1}^n (T_i - \tpi(X_i))^2}.
\end{align*}

The first term can be bounded, with probability converging to $1$, by
\[
\mathrm{I} \le \left|\frac{\sum_{i=1}^n \tpi(X_i) (T_i - \tpi(X_i))}{\sum_{i=1}^n (T_i - \tpi(X_i))^2}\right| |\hat{\tau} - \tau^*| \lesssim \sqrt{s_r} \frac{\log p}{n},
\]
where for the last inequality we apply that $|\hat{\tau} - \tau^*| \lesssim \sqrt{s_r \frac{\log p}{n}}$ and the constraint~\eqref{eq:const4}. For the second term, following again the constraint~\eqref{eq:const4}, we get that with probability converging to $1$,
\[
\mathrm{II} \lesssim \frac{1}{n} \sum_{i=1}^n (T_i - \tpi(X_i)) (r(X_i) - \r(X_i)).
\]
From above and applying exactly the same proof as Theorem~\ref{thm:reg}, with probability converging to $1$,
\[
\mathrm{II} \lesssim \left\|\frac{1}{n} \sum_{i=1}^n (T_i - \tpi(X_i)) \psi'(X_i^\top \hat{\beta}) X_i\right\|_\infty \|\hat{\beta} - \beta^*\|_1 + \frac{1}{n} \sum_{i=1}^n (X_i^\top(\hat{\beta} - \beta^*))^2 \lesssim s_r \frac{\log p}{n}.
\]

For the third term, it directly follows from the constraints in the calibration program that with probability converging to $1$,
\[
\mathrm{III} \lesssim \frac{\sqrt{\log p}}{n} \|\hat{\mbb \mu}\|_2.
\]
Then the desired result is a direct consequence of the following lemma.

\begin{lemma}
	Consider the set up of Theorem~\ref{thm:coef} (i), with probability converging to $1$, we have that $\|\hat{\mbb \mu}\|_2 = O(n^{1/4})$.
\end{lemma}

\begin{proof}
	Following the same analysis as Lemma~\ref{lem:feasible}, we have that by choosing $\eta_r \asymp \sqrt{\frac{\log p}{n}}, M_r \asymp 1$ large enough,
	\[
	\hat\mu_{\ora, i} \coloneqq \pi_i^* (\tau^* - \hat{\tau}) + r_i^* - \r_i
	\]
	is, with probability converging to $1$, a feasible solution to the optimization program, so that the desired result follows from an analysis analogous to Lemma~\ref{lem:mu}.
\end{proof}

\subsubsection{Proof of Theorem~\ref{thm:coef} (ii)}

We start by presenting the following lemmas.

\begin{lemma}\label{lem:coefgammaest}
	Under the conditions of Theorem~\ref{thm:coef} (ii), the same conclusion in Lemma~\ref{lem:gammaest} still holds.
\end{lemma}

\begin{proof}
	This follows from the same analysis as in Lemma~\ref{lem:gammaest}.
\end{proof}

\begin{lemma}\label{lem:sigmacoef}
	Under the conditions of Theorem~\ref{thm:coef} (ii), we have that
	\[
	\left|\tilde{\sigma}_e^2 - \bar{\sigma}_e^2\right| = o_\pr(1 / \sqrt{n}),
	\]
	where $\bar{\sigma}_e^2 \coloneqq n^{-1} \sum_{i = 1}^n e_i^2$.
\end{lemma}

\begin{proof}
	We have the decomposition
	\[
	\tilde{\sigma}_e^2 - \frac{1}{n} \sum_{i=1}^n e_i^2 = \frac{1}{n} \sum_{i=1}^n (\pi_i^* - \tpi_i)^2 + \frac{2}{n} \sum_{i=1}^n e_i (\pi_i^* - \tpi_i).
	\]
	
	For the first term, using mean value theorem and that $\phi'(\cdot)$ is uniformly bounded, we have that
	\[
	\frac{1}{n} \sum_{i=1}^n (\pi_i^* - \tpi_i)^2 \lesssim \frac{1}{n} \sum_{i=1}^n (\tilde{X}_i^\top(\tilde{\gamma} - \gamma^*))^2.
	\]
	Then with Lemma~\ref{lem:coefgammaest}, it follows from a similar analysis to that in Lemma \ref{lem:gammaest} that the first term is $o_\pr(1 / \sqrt{n})$.
%

We now turn to the second term. Using exactly the same argument as the control of term $\mathrm{I}$ in the proof of Theorem~\ref{thm:prop}, it remains to prove that with probability converging to $1$,
\[
\left\|\frac{1}{n} \sum_{i = 1}^n e_i \phi'(X_i^\top \gamma^*) \tilde{X}_i \right\|_\infty \lesssim \sqrt{\frac{\log p \vee n}{n}}.
\]
Now that all the entries of $\phi'(X_i^\top \gamma^*) \tilde{X}_i$ are bounded and that $e_i$'s are i.i.d. bounded random variables which are independent from $\tilde{\bX}$, the desired result follows from an application of Hoeffding's inequality.
\end{proof}

\begin{proof}[Proof of Theorem~\ref{thm:coef} (ii)]
	Let $\hat{\tau}^*_\est$ be a modification of $\hat{\tau}_\est$ but with $\tilde{\pi}$ replaced by $\pi^*$. Then we have the decomposition
	\begin{equation}\label{eq:coefdecomp}
		\hat{\tau}_\est - \hat{\tau}^*_\est = \frac{1}{n} \sum_{i = 1}^n \frac{(Y_i - T_i \hat{\tau} - \r_i - \hat{\mu}_i) (\pi_i^* - \tpi_i)}{\hat{\sigma}_e^2} + \left(\frac{1}{\hat{\sigma}_e^2} - \frac{1}{\bar{\sigma}_e^2}\right) \cdot \frac{1}{n} \sum_{i = 1}^n (Y_i - T_i \hat{\tau} - \r_i - \hat{\mu}_i) (T_i - \pi_i^*).
	\end{equation}
	We first consider the second term, using that with probability converging to $1$,
	\[
	\max_{i} \left|(Y_i - T_i \hat{\tau} - \r_i - \hat{\mu}_i) (T_i - \pi_i^*)\right| \lesssim 1,
	\]
	we only need to prove that
	\[
	\frac{|\hat{\sigma}_e^2 - \bar{\sigma}_e^2|}{\hat{\sigma}_e^2 \bar{\sigma}_e^2} = o_\pr(1 / \sqrt{n}),
	\]
	which is a direct consequence of Lemma~\ref{lem:sigmacoef}.
	
	For the first term, recall we already have the constraint~\eqref{eq:const4}, it remains to show that
 \[
 \mathrm{I} \coloneqq \left|\frac{1}{\sqrt{n}} \sum_{i = 1}^n (Y_i - T_i \hat{\tau} - \r_i - \hat{\mu}_i) (\pi^*_i - \tilde{\pi}_i)\right| = o_\pr(1).
 \]
 With Lemma~\ref{lem:coefgammaest}, this directly follows from the exactly the same analysis as the terms $\mathrm{I}$ and $\mathrm{II}$ in the proof of Theorem~\ref{thm:prop}.

In light of our control of two terms in the decomposition~\eqref{eq:coefdecomp}, we have that 
\[
\sqrt{n}(\hat{\tau}_\est - \tau^\ast) = \sqrt{n}(\hat{\tau}_\est^* - \tau^\ast) + o_\pr(1).
\]

With the above preparation, the desired result is a direct consequence of the following lemma.
\end{proof}

\begin{lemma}
	Under the conditions of Theorem~\ref{thm:coef} (ii), we have the following asymptotic representation
	\[
	\sqrt{n}(\hat{\tau}^*_\est - \tau^*) = \frac{1}{\sqrt{n}} \sum_{i = 1}^n \frac{e_i \varepsilon_i}{\sigma_e^2} + \frac{1}{\sqrt{n}} \sum_{i = 1}^n \frac{e_i((\tau^* - \hat{\tau}) \pi_i^* + r_i^* - \r_i - \hat{\mu}_i)}{\sigma_e^2} + o_\pr(1)
	\]
\end{lemma}

\begin{proof}
	Reorganizing $Y_i - T_i \hat{\tau} - \r_i = e_i (\tau^* - \hat{\tau}) + \pi_i^* (\tau^* - \hat{\tau}) + r_i^* - \r_i + \varepsilon_i$, we rewrite $\hat{\tau}^*_\est$ as
	\[
	\hat{\tau}^*_\est = \tau^* + \frac{1}{n} \sum_{i = 1}^n \frac{e_i((\tau^* - \hat{\tau}) \pi_i^* + r_i^* - \r_i - \hat{\mu}_i + \varepsilon_i)}{\bar\sigma_e^2}.
	\]
	Using that the $e_i$'s are independent from $(\tau^* - \hat{\tau}) \pi_i^* + r_i^* - \r_i - \hat{\mu}_i$ and that with probability converging to $1$, all the $(\tau^* - \hat{\tau}) \pi_i^* + r_i^* - \r_i - \hat{\mu}_i + \varepsilon_i$'s are bounded, we have from standard results in central limit theorem that 
	\[
	\frac{1}{\sqrt{n}} \sum_{i = 1}^n e_i((\tau^* - \hat{\tau}) \pi_i^* + r_i^* - \r_i - \hat{\mu}_i) = O_\pr(1).
	\]
	Now using the law of large numbers, we have that $\bar\sigma_e^2 - \sigma_e^2 = o_\pr(1)$. Putting together the above analysis yields the desired result.
\end{proof}

\subsection{Conditions related to Section \ref{sec:iv}}
\label{app:iv}

In this section, we lay out the standard identification conditions for $\chi^{\ast}$ defined in \eqref{late} to identify LATE, i.e. $\chi^{\ast} \equiv \mathbb{E} [Y (1) - Y (0) | T (1) > T (0)]$. This is now a classical result in the causal inference literature, and hence we will not prove it or discuss it further.

\begin{condition}[IV identification conditions]\leavevmode
\label{cond:iv1}
\begin{enumerate}[label = (\roman*)]
\item IV unconfoundedness: $Y (z, t), T (z) \independent Z | X$ almost surely, for all $t, z \in \{0, 1\}$;
\item Exclusion restriction: $Y (1, t) = Y (0, t) = Y (t)$ almost surely, for all $t \in \{0, 1\}$;
\item Monotonicity: $T (1) \geq T (0)$ almost surely;
\item Relevance: $\pr (T (1) = 1) \neq \pr (T (0) = 1)$;
\item Consistency: $Y = T Y (1) + (1 - T) Y (0)$ and $T = Z T (1) + (1 - Z) T (0)$;
\item IV Positivity: There exists an absolute constant $c_{\zeta} \in (0, 0.5)$ such that $\zeta^{\ast} (X) \in (c_\zeta, 1 - c_\zeta)$ almost surely.
\end{enumerate}
\end{condition}

\subsection{Proof of Corollary \ref{cor:as}}
\label{app:as}

Similar to Section \ref{sec:ate} and \ref{sec:atethm}, the proof of Corollary \ref{cor:as} can be divided into the following steps as in the proof of Theorem \ref{thm:reg} and \ref{thm:prop}. 

As a first step, we need to show the following.
\begin{lemma}
\label{lem:mu as}
The same conclusion of Lemma \ref{lem:mu} holds when Condition \ref{cd:reg} is replaced by Condition \ref{cd:reg as} and $s_r = o (\sqrt{n} / \log p)$ is replaced by $\xi_r > 1 / 2$.
\end{lemma}

\begin{proof}
The proof follows a similar argument to the proof of Lemma \ref{lem:mu} in Appendix \ref{app:mu}, except that \eqref{exact unbiased} needs to be replaced by the following chain of inequalities:
\begin{align*}
\Vert \hat{\bmu} \Vert_{2} & \leq \Vert \hat{\bmu}_{\ora} \Vert_{2} \lesssim \sqrt{\sum_{i = 1}^{n} (\hat{r}_i - r_i^\ast)^{2}} \lesssim \sqrt{\sum_{i = 1}^{n} (\bar{r}_i^\ast - r_i^\ast)^{2}} + \sqrt{\sum_{i = 1}^{n} (X_i^\top (\hat{\beta} - \beta^\ast))^2}.
\end{align*}
Then the rest follows from applying the same argument as in Appendix \ref{app:mu}, together with the assumption \eqref{tail or} in Condition \ref{cond:nuis as}.
\end{proof}

We now prove Corollary \ref{cor:as} (i). 
\begin{proof}
The proof closely mirrors the proof for Theorem \ref{thm:reg} given in Appendix \ref{app:reg}, except that we need to modify the decomposition in \eqref{reg expansion} with an extra term $\mathrm{IV}$:
\begin{align*}
\hat{\tau}_\est & - \bar{\tau}^\ast \\
= & \ \underset{\mathrm{I}}{\underbrace{\frac{1}{n} \sum_{i=1}^n \left( 1 - \frac{T_i}{\tpi_i} \right) \psi'(X_i^\top \hat{\beta}) X_i^\top (\hat{\beta} - \beta^*) }} \\
& + \underset{\mathrm{II}}{\underbrace{\frac{1}{n} \sum_{i=1}^n \left( 1 - \frac{T_i}{\tpi_i} \right)  \psi''(\eta_i X_i^\top \hat{\beta} + (1 - \eta_i) X_i^\top \beta^*) \left(X_i^\top (\hat{\beta} - \beta^*)\right)^2 }} \\
& + \underset{\mathrm{III}}{\underbrace{\frac{1}{n} \sum_{i=1}^n \left( 1 - \frac{T_i}{\tpi_i} \right) \hmu_i}} + \underbrace{\frac{1}{n} \sum_{i = 1}^{n} \left( 1 - \frac{T_i}{\tpi_i} \right) \left( \bar{r}^{\ast}_{i} - r^{\ast}_{i} \right)}_{\mathrm{IV}} + \frac{1}{n} \sum_{i=1}^n \frac{T_i \varepsilon_i(1)}{\tpi_i}.
\end{align*}
Using the same argument as in Appendix \ref{app:reg}, with Condition \ref{cd:reg} replaced by Condition \ref{cd:reg as}, we have, since $\xi_{r} > 1 / 2$,
\begin{align*}
|\mathrm{I}| = O_{\mathbb{P}} \left( \left( \sqrt{\frac{n}{\log p}} \right)^{- \frac{4 \xi_{r}}{2 \xi_{r} + 1}} \right) = o_{\pr} (n^{-1 / 2}), |\mathrm{II}| \lesssim \Vert \hat{\beta} - \beta^{\ast} \Vert^{2} = O_{\mathbb{P}} \left( \left( \sqrt{\frac{n}{\log p}} \right)^{- \frac{4 \xi_{r}}{2 \xi_{r} + 1}} \right) = o_{\pr} (n^{-1 / 2}).
\end{align*}
Following Lemma \ref{lem:mu as}, we also have $|\mathrm{III}| = o_{\pr} (n^{-1 / 2})$. Finally, we can easily show that $|\mathrm{IV}| = o_{\pr} (n^{-1 / 2})$ by invoking \eqref{tail or} using Cauchy-Schwarz inequality.
\end{proof}

Finally, we are left to prove Corollary \ref{cor:as} (ii). To this end, we first need to show that $\gamma^{\ast}$ is a feasible solution to \eqref{eq:program}, or equivalently $\bar{\pi}^{\ast}_i$'s satisfy all the constraints in \eqref{eq:program}, with probability converging to 1.

\begin{proof}
First, we notice as in the proof of Lemma \ref{lem:program}, the true propensity score $\pi^{\ast}_{i}$'s satisfy all the constraints in \eqref{eq:program} with probability converging to 1. Hence, after applying triangle inequality and Cauchy-Schwarz inequality, we are only left to show the following holds with probability approaching 1:
\begin{align*}
\left\{ \frac{1}{n} \sum_{i = 1}^{n} \left( \frac{T_{i}}{\pi^{\ast}_{i}} - \frac{T_{i}}{\bar{\pi}^{\ast}_{i}} \right)^{2} \right\}^{1 / 2} \ll \sqrt{\frac{\log p}{n}},
\end{align*}
which is true due to \eqref{tail ps} imposed in Condition \ref{cd:prop as}.
\end{proof}

Next, we establish the following lemma analogous to Lemma \ref{lem:gammaest}.
\begin{lemma}
\label{lem:gammaest as}
Under the Conditions of Lemma \ref{lem:gammaest}, except with Condition \ref{cd:prop} replaced by Condition \ref{cd:prop as} and $s_\pi = o (\sqrt{n} / \log p)$ replaced by $\xi_\pi > 1 / 2$, by choosing $\eta_{r}, \eta_{\pi_{1}}, \eta_{\pi_{2}} \asymp \sqrt{\frac{\log p}{n}}$, $M_r, M_\pi, M_{\gamma} \asymp 1$, we have that with probability converging to 1,
\begin{align*}
\Vert \tilde{\gamma} - \hat{\gamma} \Vert_{1} = O \left( \left( \sqrt{\frac{n}{\log p}} \right)^{- \frac{2 \xi_{\pi} - 1}{2 \xi_{\pi} + 1}} \right) \,\,\, \& \,\,\, \frac{1}{n} \sum_{i = 1}^{n} (X_i^{\top} (\tilde{\gamma} - \hat{\gamma}))^{2} = O \left( \left( \sqrt{\frac{n}{\log p}} \right)^{- \frac{4 \xi_{\pi}}{2 \xi_{\pi} + 1}} \right).
\end{align*}
\end{lemma}

\begin{proof}
The proof closely follows the argument in Appendix \ref{app:gammaest}. The first statement follows exactly the same argument. In terms of the second statement, we first notice that with probability approaching 1,
\begin{align*}
\left\Vert \frac{1}{n} \sum_{i = 1}^{n} \left( \frac{T_{i}}{\tilde{\pi}_{i}} - \frac{T_{i}}{\bar{\pi}_{i}^{\ast}} \right) X_{i}^{\top} \right\Vert_{\infty} \lesssim \sqrt{\frac{\log p}{n}},
\end{align*}
which in turn implies that with probability approaching 1
\begin{align*}
& \ \left\vert \frac{1}{n} \sum_{i = 1}^{n} \left( \frac{T_{i}}{\tilde{\pi}_{i}} - \frac{T_{i}}{\bar{\pi}_{i}^{\ast}} \right) X_{i}^{\top} (\tilde{\gamma} - \gamma^{\ast}) \right\vert \\
\leq & \ \left\Vert \frac{1}{n} \sum_{i = 1}^{n} \left( \frac{T_{i}}{\tilde{\pi}_{i}} - \frac{T_{i}}{\bar{\pi}_{i}^{\ast}} \right) X_{i}^{\top} \right\Vert_{\infty} \Vert \tilde{\gamma} - \gamma^{\ast} \Vert_{1} \\
= & \ O \left( \left( \sqrt{\frac{n}{\log p}} \right)^{-1} \left( \sqrt{\frac{n}{\log p}} \right)^{- \frac{2 \xi_{\pi} - 1}{2 \xi_{\pi} + 1}} \right) \\
= & \ O \left( \left( \sqrt{\frac{n}{\log p}} \right)^{- \frac{4 \xi_{\pi}}{2 \xi_{\pi} + 1}} \right).
\end{align*}
Then following the same Taylor expansion argument as in the proof of Lemma \ref{lem:gammaest} in Appendix \ref{app:gammaest}, we obtain, with probability approaching 1,
\begin{align*}
\frac{1}{n} \sum_{i = 1}^{n} \frac{T_{i}}{\bar{\pi}^{\ast}_{i}} \left( X_{i}^{\top} (\tilde{\gamma} - \gamma^{\ast}) \right)^{2} = O \left( \left( \sqrt{\frac{n}{\log p}} \right)^{- \frac{4 \xi_{\pi}}{2 \xi_{\pi} + 1}} \right).
\end{align*}
Thus finally, again similar to the proof in Appendix \ref{app:gammaest},
\begin{align*}
\frac{1}{n} \sum_{i = 1}^{n} \left( X_{i}^{\top} (\tilde{\gamma} - \gamma^{\ast}) \right)^{2} & = \frac{1}{n} \sum_{i = 1}^{n} \frac{T_{i}}{\bar{\pi}^{\ast}_{i}} \left( X_{i}^{\top} (\tilde{\gamma} - \gamma^{\ast}) \right)^{2} + (\tilde{\gamma} - \gamma^{\ast})^{\top} \frac{1}{n} \sum_{i = 1}^{n} X_{i} X_{i}^{\top} \left( 1 - \frac{T_{i}}{\bar{\pi}^{\ast}_{i}} \right) (\tilde{\gamma} - \gamma^{\ast}) \\
& \lesssim \left( \sqrt{\frac{n}{\log p}} \right)^{- \frac{4 \xi_{\pi}}{2 \xi_{\pi} + 1}} + \left\Vert \frac{1}{n} \sum_{i = 1}^{n} X_{i} X_{i}^{\top} \left( 1 - \frac{T_{i}}{\bar{\pi}^{\ast}_{i}} \right) \right\Vert_{\infty} \Vert \tilde{\gamma} - \gamma^{\ast} \Vert_{1}^{2},
\end{align*}
with probability converging to 1. By the first statement of this lemma, $\Vert \tilde{\gamma} - \gamma^{\ast} \Vert_{1}^{2} \lesssim \left( \sqrt{\frac{n}{\log p}} \right)^{- \frac{4 \xi_{\pi} - 2}{2 \xi_{\pi} + 1}}$. By the same argument as in Appendix \ref{app:gammaest} and \eqref{tail ps} imposed in Condition \ref{cond:nuis as}, we have, with probability approaching 1,
\begin{align*}
& \ \left\Vert \frac{1}{n} \sum_{i = 1}^{n} X_{i} X_{i}^{\top} \left( 1 - \frac{T_{i}}{\bar{\pi}^{\ast}_{i}} \right) \right\Vert_{\infty} \\
\leq & \ \left\Vert \frac{1}{n} \sum_{i = 1}^{n} X_{i} X_{i}^{\top} \left( 1 - \frac{T_{i}}{\pi^{\ast}_{i}} \right) \right\Vert_{\infty} + \left\Vert \frac{1}{n} \sum_{i = 1}^{n} X_{i} X_{i}^{\top} \left( \frac{T_{i}}{\bar{\pi}^{\ast}_{i}} - \frac{T_{i}}{\pi^{\ast}_{i}} \right) \right\Vert_{\infty} \lesssim \sqrt{\frac{\log p}{n}}.
\end{align*}
Combining the above results, we have
\begin{align*}
\frac{1}{n} \sum_{i = 1}^{n} \left( X_{i}^{\top} (\tilde{\gamma} - \gamma^{\ast}) \right)^{2} & \lesssim \left( \sqrt{\frac{n}{\log p}} \right)^{- \frac{4 \xi_{\pi}}{2 \xi_{\pi} + 1}} + \left( \sqrt{\frac{n}{\log p}} \right)^{-1} \left( \sqrt{\frac{n}{\log p}} \right)^{- \frac{4 \xi_{\pi} - 2}{2 \xi_{\pi} + 1}} \\
& = \left( \sqrt{\frac{n}{\log p}} \right)^{- \frac{4 \xi_{\pi}}{2 \xi_{\pi} + 1}} + \left( \sqrt{\frac{n}{\log p}} \right)^{- \frac{4 \xi_{\pi} + 2 \xi_{\pi} - 1}{2 \xi_{\pi} + 1}} \\
& \lesssim \left( \sqrt{\frac{n}{\log p}} \right)^{- \frac{4 \xi_{\pi}}{2 \xi_{\pi} + 1}},
\end{align*}
with probability approaching 1, where the last inequality follows because $\xi_{\pi} > 1 / 2$.
\end{proof}

In this last step, we prove Corollary \ref{cor:as} (ii) with the above preparation. The proof resembles the proof of Theorem \ref{thm:prop} in Appendix \ref{app:prop}.

\begin{proof}
First, it is straightforward to see that Lemma \ref{lem:pitrim} and Lemma \ref{lem:infinities} with $\pi^{\ast}$ replaced by $\bar{\pi}^{\ast}$ still hold under approximately sparse GLMs. The proof then proceeds similarly by showing that $\bar{Q} = o_{\pr} (1)$, in which $\bar{Q}$ is defined in the same way as $Q$ in \eqref{Q} except replacing $\pi^\ast_i$'s by $\bar{\pi}^\ast_i$'s. The difference between $\bar{Q}$ and $Q$ can be shown to be $o_{\pr} (1)$ by invoking \eqref{tail ps} imposed under Condition \ref{cond:nuis as}. We apply the following decomposition of $\bar{Q}$, similar to that of $Q$ given in Appendix \ref{app:prop}:
\begin{align*}
\bar{Q} & = \frac{1}{\sqrt{n}} \sum_{i=1}^n \left(\frac{T_i (Y_i - \r_i)}{\bar{\pi}_i^* \tpi_i} - \frac{\hmu_i}{\tpi_i}\right) (\bar{\pi}_i^* - \tpi_i) - \frac{1}{\sqrt{n}} \sum_{i=1}^n \left(\frac{T_i}{\bar{\pi}_i^*} - 1\right) \frac{\hmu_i}{\tpi_i} (\bar{\pi}_i^* - \tpi_i) \\
 & = \underset{\mathrm{I}}{\underbrace{\frac{1}{\sqrt{n}} \sum_{i=1}^n \left(\frac{T_i (Y_i - \r_i)}{\bar{\pi}_i^* \tpi_i} - \frac{\hmu_i}{\tpi_i}\right) (\hpi_i - \tpi_i)}} + \underset{\mathrm{II}}{\underbrace{\frac{1}{\sqrt{n}} \sum_{i=1}^n \left(\frac{T_i (Y_i - \r_i)}{\bar{\pi}_i^* \tpi_i} - \frac{\hmu_i}{\tpi_i}\right) (\bar{\pi}_i^* - \hpi_i)}} \\
 & \quad - \underset{\mathrm{III}}{\underbrace{\frac{1}{\sqrt{n}} \sum_{i=1}^n \left(\frac{T_i}{\bar{\pi}_i^*} - 1\right) \frac{\hmu_i}{\tpi_i} (\bar{\pi}_i^* - \tpi_i)}}.
\end{align*}
In particular, term $\mathrm{III}$ can be shown, by the same Taylor expansion argument as in Appendix \ref{app:prop}, to be $o_{\pr} (1)$ by using Lemma \ref{lem:gammaest as} and $\xi_{\pi} > 1 / 2$. For terms $\mathrm{I}$ and $\mathrm{II}$, we follow the same decomposition strategy as in Appendix \ref{app:prop}, except with $s_{\pi} = o (\sqrt{n} / \log p)$ replaced by $\xi_{\pi} > 1 / 2$, Lemma \ref{lem:gammaest} replaced by Lemma \ref{lem:gammaest as}, and Lemma \ref{lem:infinities} with $\pi^{\ast}$ replaced by $\bar{\pi}^{\ast}$. The desired result then follows.
\end{proof}

\subsubsection{On Conditions \ref{cd:reg as} and \ref{cd:prop as}}
\label{app:nuis as}

The linear model case has been addressed in \citet{bradic2019minimax}. In this section, we show that Condition \ref{cd:reg as} or \ref{cd:prop as} holds for logistic linear model with the coefficients estimated by $\ell_{1}$-penalized maximum likelihood. We only need to demonstrate the existence of an estimator $\hat{\gamma}$ that satisfies Condition \ref{cd:prop as} and the same argument applies to Condition \ref{cd:reg as} by symmetry. Fix $s \asymp \sqrt{n / \log p}$. Under Condition \ref{cond:nuis as}, we have $\Vert \gamma^{\ast} - \gamma_{s}^{\ast} \Vert_{2} \ll \left( \sqrt{\frac{n}{\log p}} \right)^{- \xi_{\pi}}$. Let the corresponding active coordinate index set of $\gamma_{s}^{\ast}$ be $J^{\ast} \subset \{1, \cdots, p\}$ and $\vert J^{\ast} \vert = s$. We also let $\mathcal{L}_{n} (\gamma) \coloneqq \frac{1}{n} \sum_{i = 1}^{n} \left\{ \log \left( 1 + e^{\gamma^{\top} X_{i}} \right) - A_{i} \gamma^{\top} X_{i} \right\}$ denote the empirical mean of the negative log-likelihood function of the logistic-linear model.

Define, for some $C > 0$ sufficiently large,
\begin{equation}
\label{glm est}
\hat{\gamma} \coloneqq \arg \min_{\gamma \in \R^{p}} \frac{1}{n} \sum_{i = 1}^{n} \left\{ \log \left( 1 + e^{\gamma^{\top} X_{i}} \right) - A_{i} \gamma^{\top} X_{i} \right\} + C \lambda \Vert \gamma \Vert_{1}
\end{equation}
and
\begin{equation}
\label{glm pop}
\bar{\gamma} \coloneqq \arg \min_{\gamma \in \R^{p}} \E \left[ \log \left( 1 + e^{\gamma^{\top} X} \right) - \bar{\pi}^{\ast} (X) \gamma^{\top} X \right] + \lambda \sum_{j \in J^{\ast c}} \vert \gamma_{j} \vert.
\end{equation}
Let $\bar{J}$ be the indices of nonzero elements of $\bar{\gamma}$. In addition, we assume that there exist absolute constants $0 < c < C < 1$ such that
\begin{equation}
\label{bounded}
\text{$\phi (\bar{\gamma}^{\top} X)$ and $\phi (\gamma^{\ast \top} X)$ are both bounded between $c$ and $C$}.
\end{equation}
Assumption \ref{bounded} is guaranteed to hold (so not an assumption) for logistic linear model if we truncate $\bar{\gamma}^{\top} X$ when it is outside a certain range, as done in e.g. Section \ref{sec:dipw}. But we impose Assumption \ref{bounded} for convenience.

We first establish the following properties of $\bar{\gamma}$ as an approximation of $\gamma^{\ast}$.
\begin{lemma}
\label{lem:pop}
By choosing $\lambda = \sqrt{\frac{\log p}{n}}$, we have the following:
\begin{enumerate}
\item $\Vert \bar{\gamma} - \gamma^{\ast} \Vert_{2} \lesssim \left( \sqrt{\frac{n}{\log p}} \right)^{- \frac{2 \xi_{\pi}}{2 \xi_{\pi} + 1}}$;
\item $|\bar{J}| \lesssim \left( \sqrt{\frac{n}{\log p}} \right)^{\frac{2}{2 \xi_{\pi} + 1}}$;
\item $\left\Vert \E \left[ \left( \phi (\bar{\gamma}^{\top} X) - \phi (\gamma^{\ast \top} X) \right) X \right] \right\Vert_{\infty} \lesssim \lambda = \left( \sqrt{\frac{n}{\log p}} \right)^{-1}$;
\item If $\xi_{\pi} > 1 / 2$, $\Vert \bar{\gamma} - \gamma^{\ast} \Vert_{1} \lesssim \left( \sqrt{\frac{n}{\log p}} \right)^{- \frac{2 \xi_{\pi} - 1}{2 \xi_{\pi} + 1}}$.
\end{enumerate}
\end{lemma}

\begin{proof}
By the optimality of $\bar{\gamma}$ as a solution to \eqref{glm pop} and mean-value theorem, we have, for $\gamma_{1}$ between $\gamma^{\ast}, \bar{\gamma}$ and $\gamma_{2}$ between $\gamma^{\ast}, \gamma_{s}^{\ast}$,
\begin{align*}
& \ \E \left[ \log \left( 1 + e^{\bar{\gamma}^{\top} X} \right) - \phi (\gamma^{\ast \top} X) \bar{\gamma}^{\top} X \right] + \lambda \sum_{j \in J^{\ast c}} \vert \bar{\gamma}_{j} \vert \leq \E \left[ \log \left( 1 + e^{\gamma_{s}^{\ast \top} X} \right) - \phi (\gamma^{\ast \top} X) \gamma_{s}^{\ast \top} X \right] \\
\Rightarrow & \ (\bar{\gamma} - \gamma^{\ast})^{\top} \E \left[ [\phi (1 - \phi)] (\gamma_{1}^{\top} X) X X^{\top} \right] (\bar{\gamma} - \gamma^{\ast}) + \lambda \sum_{j \in J^{\ast c}} \vert \bar{\gamma}_{j} \vert \\
& \leq (\gamma_{s}^{\ast} - \gamma^{\ast})^{\top} \E \left[ [\phi (1 - \phi)] (\gamma_{2}^{\top} X) X X^{\top} \right] (\gamma_{s}^{\ast} - \gamma^{\ast}) \lesssim \Vert \gamma_{s}^{\ast} - \gamma^{\ast} \Vert_{2}^{2} \lesssim \left( \sqrt{\frac{n}{\log p}} \right)^{- \frac{4 \xi_{\pi}}{2 \xi_{\pi} + 1}}.
\end{align*}
Then by the lower bound assumed in \eqref{bounded}, we have
\begin{align*}
\Vert \bar{\gamma} - \gamma^{\ast} \Vert_{2} \lesssim \left( \sqrt{\frac{n}{\log p}} \right)^{- \frac{2 \xi_{\pi}}{2 \xi_{\pi} + 1}}.
\end{align*}

Next, by KKT conditions, we have for $j \in \bar{J} \setminus J^{\ast}$, $\left\vert \E \left[ (\phi (\bar{\gamma}^{\top} X) - \phi (\gamma^{\ast \top} X)) X^{\top} e_{j} \right] \right\vert = \lambda$, which implies
\begin{align*}
\sum_{j \in \bar{J} \setminus J^{\ast}} \left( \E \left[ (\phi (\bar{\gamma}^{\top} X) - \phi (\gamma^{\ast \top} X)) X^{\top} e_{j} \right] \right)^{2} = \lambda^{2} |\bar{J} \setminus J^{\ast}|.
\end{align*}
By the first statement of this lemma, we have
\begin{align*}
& \ \sum_{j \in \bar{J} \setminus J^{\ast}} \left( \E \left[ (\phi (\bar{\gamma}^{\top} X) - \phi (\gamma^{\ast \top} X)) X^{\top} e_{j} \right] \right)^{2} \\
\lesssim & \ \E \left[ (\phi (\bar{\gamma}^{\top} X) - \phi (\gamma^{\ast \top} X))^{2} \right] \lesssim \Vert \bar{\gamma} - \gamma^{\ast} \Vert_{2}^{2} \lesssim \left( \sqrt{\frac{n}{\log p}} \right)^{- \frac{4 \xi_{\pi}}{2 \xi_{\pi} + 1}}.
\end{align*}
Hence
\begin{align*}
|\bar{J} \setminus J^{\ast}| \lesssim \lambda^{-2} \left( \sqrt{\frac{n}{\log p}} \right)^{- \frac{4 \xi_{\pi}}{1 + 2 \xi_{\pi}}} = \left( \sqrt{\frac{n}{\log p}} \right)^{2 - \frac{4 \xi_{\pi}}{2 \xi_{\pi} + 1}} = \left( \sqrt{\frac{n}{\log p}} \right)^{\frac{2}{2 \xi_{\pi} + 1}}.
\end{align*}
Thus we have $|\bar{J}| = |\bar{J} \setminus J^{\ast}| + |J^{\ast}| \lesssim \left( \sqrt{\frac{n}{\log p}} \right)^{\frac{2}{2 \xi_{\pi} + 1}}$.

Finally, by directly following the proofs of Lemma B.1 and Lemma B.4 of \citet{bradic2019minimax}, we obtain the third and fourth claims.
\end{proof}

By Lemma \ref{lem:pop}.3, we obtain the following:
\begin{lemma}
\label{lem:emp}
With probability converging to 1, 
\begin{align*}
\left\Vert \nabla_{\gamma} \mathcal{L}_{n} (\bar{\gamma}) \right\Vert_{\infty} \equiv \left\Vert \frac{1}{n} \sum_{i = 1}^{n} \left( \phi (\bar{\gamma}^{\top} X_{i}) - A_{i} \right) X_{i} \right\Vert_{\infty} \leq \lambda = \left( \sqrt{\frac{n}{\log p}} \right)^{-1}.
\end{align*}
\end{lemma}

\begin{proof}
With probability converging to 1, we have
\begin{align*}
& \ \left\Vert \frac{1}{n} \sum_{i = 1}^{n} \left( \phi (\bar{\gamma}^{\top} X_{i}) - A_{i} \right) X_{i} \right\Vert_{\infty} \\
\leq & \ \left\Vert \frac{1}{n} \sum_{i = 1}^{n} \left( \phi (\bar{\gamma}^{\top} X_{i}) - A_{i} \right) X_{i} - \E \left[ \left( \phi (\bar{\gamma}^{\top} X) - A \right) X \right] \right\Vert_{\infty} + \left\Vert \E \left[ \left( \phi (\bar{\gamma}^{\top} X) - A \right) X \right] \right\Vert_{\infty} \\
\lesssim & \ \sqrt{\frac{\log p}{n}} + \left\Vert \E \left[ \left( \phi (\bar{\gamma}^{\top} X) - \phi (\gamma^{\ast \top} X) \right) X \right] \right\Vert_{\infty} + \left\Vert \E \left[ \left( \phi (\gamma^{\ast \top} X) - \pi^{\ast} (X) \right) X \right] \right\Vert_{\infty} \\
\overset{\text{Lemma \ref{lem:pop}.3}}{\lesssim} & \ \sqrt{\frac{\log p}{n}} + \left\Vert \E \left[ \left( \phi (\gamma^{\ast \top} X) - \pi^{\ast} (X) \right) X \right] \right\Vert_{\infty} \lesssim \sqrt{\frac{\log p}{n}},
\end{align*}
where the last inequality follows from \eqref{tail ps} imposed in Condition \ref{cond:nuis as}.
\end{proof}

Let $\delta \coloneqq \hat{\gamma} - \bar{\gamma}$ so $\hat{\gamma} = \bar{\gamma} + \delta$. We then have the following result.

\begin{lemma}
\label{lem:cone}
With probability approaching 1, we have $\Vert \delta_{\bar{J}^{c}} \Vert_{1} \lesssim \Vert \delta_{\bar{J}} \Vert_{1}$.
\end{lemma}

\begin{proof}
By the optimality of $\hat{\gamma}$ as a solution to \eqref{glm est} and the convexity of the log-likelihood function \citep{wainwright2019high}, we have
\begin{align*}
0 & \geq \underbrace{\frac{1}{n} \sum_{i = 1}^{n} \log \left( 1 + e^{(\delta + \bar{\gamma})^{\top} X_{i}} \right) - \log \left( 1 + e^{\bar{\gamma}^{\top} X_{i}} \right) - \delta^{\top} A_{i} X_{i} + C \lambda \left( \Vert \delta + \bar{\gamma} \Vert_{1} - \Vert \bar{\gamma} \Vert_{1} \right)}_{\eqqcolon \, \mathcal{E}_{n} (\delta)} \\
& \geq \frac{1}{n} \sum_{i = 1}^{n} (\phi (\bar{\gamma}^{\top} X_{i}) - A_{i}) X_{i}^{\top} \delta + C \lambda \left( \Vert \delta + \bar{\gamma} \Vert_{1} - \Vert \bar{\gamma} \Vert_{1} \right) \\
& \geq - \left| \frac{1}{n} \sum_{i = 1}^{n} (\phi (\bar{\gamma}^{\top} X_{i}) - A_{i}) X_{i}^{\top} \delta \right| + C \lambda \left( \Vert \delta + \bar{\gamma} \Vert_{1} - \Vert \bar{\gamma} \Vert_{1} \right) \\
& \geq - \left\Vert \frac{1}{n} \sum_{i = 1}^{n} (\phi (\bar{\gamma}^{\top} X_{i}) - A_{i}) X_{i} \right\Vert_{\infty} \Vert \delta \Vert_{1} + C \lambda \left( \Vert \delta + \bar{\gamma} \Vert_{1} - \Vert \bar{\gamma} \Vert_{1} \right) \\
& \gtrsim - \sqrt{\frac{\log p}{n}} \Vert \delta \Vert_{1} + C \sqrt{\frac{\log p}{n}} \left( \Vert \delta + \bar{\gamma} \Vert_{1} - \Vert \bar{\gamma} \Vert_{1} \right),
\end{align*}
with probability approach 1.

We further observe that
\begin{align*}
& \ \Vert \delta + \bar{\gamma} \Vert_{1} - \Vert \bar{\gamma} \Vert_{1} = \Vert \bar{\gamma}_{\bar{J}} + \delta_{\bar{J}^{c}} + \delta_{\bar{J}} \Vert_{1} - \Vert \bar{\gamma}_{\bar{J}} \Vert_{1} \\
\geq & \ \Vert \bar{\gamma}_{\bar{J}} + \delta_{\bar{J}^{c}} \Vert_{1} - \Vert \delta_{\bar{J}} \Vert_{1} - \Vert \bar{\gamma}_{\bar{J}} \Vert_{1} = \Vert \delta_{\bar{J}^{c}} \Vert_{1} - \Vert \delta_{\bar{J}} \Vert_{1}.
\end{align*}
So
\begin{align*}
0 \gtrsim - \sqrt{\frac{\log p}{n}} \left( \Vert \delta_{\bar{J}} \Vert_{1} + \Vert \delta_{\bar{J}^{c}} \Vert_{1} \right) + C \sqrt{\frac{\log p}{n}} \left( \Vert \delta_{\bar{J}^{c}} \Vert_{1} - \Vert \delta_{\bar{J}} \Vert_{1} \right) = \sqrt{\frac{\log p}{n}} \left( (C - 1) \Vert \delta_{\bar{J}^{c}} \Vert_{1} - (C + 1) \Vert \delta_{\bar{J}} \Vert_{1} \right),
\end{align*}
which implies that $\Vert \delta_{\bar{J}^{c}} \Vert_{1} \lesssim \Vert \delta_{\bar{J}} \Vert_{1}$ with probability approaching 1.
\end{proof}

We further use the fact that the following quantity satisfies the $(\kappa, \bar{\kappa})$-Restricted Strong Convexity (RSC) condition (Theorem 9.36 of \citet{wainwright2019high}):
\begin{equation}
\label{RSC}
\mathcal{L}_{n} (\bar{\delta} + \bar{\gamma}) - \mathcal{L}_{n} (\bar{\gamma}) - \nabla_{\gamma} \mathcal{L}_{n} (\bar{\gamma})^{\top} \bar{\delta} \geq \frac{\kappa}{2} \Vert \bar{\delta} \Vert_{2}^{2} - \bar{\kappa}^{2} \Vert \bar{\delta} \Vert_{1}^{2}
\end{equation}
for all $\bar{\delta}$ in a ball with radius $1$, with $\kappa > 0$ bounded from below and $\bar{\kappa} \asymp \sqrt{\frac{\log p}{n}}$ with probability approaching 1. We then eventually establish the following $\ell_{2}$- and $\ell_{1}$-convergence rates for $\hat{\gamma}$.

\begin{lemma}
\label{lem:rates}
With probability approaching 1, when $\xi_{\pi} > 1 / 2$, we have
\begin{align*}
\Vert \delta \Vert_{2} \lesssim \left( \sqrt{\frac{n}{\log p}} \right)^{- \frac{2 \xi_{\pi}}{2 \xi_{\pi} + 1}}, \Vert \delta \Vert_{1} \lesssim \left( \sqrt{\frac{n}{\log p}} \right)^{- \frac{2 \xi_{\pi} - 1}{2 \xi_{\pi} + 1}}.
\end{align*}
\end{lemma}

\begin{proof}
Take any vector $\bar{\delta}$ such that the conclusion in Lemma \ref{lem:cone} holds, $\Vert \bar{\delta} \Vert_{2} = \rho$ for some $\rho > 0$, and the $(\kappa, \bar{\kappa})$-RSC \eqref{RSC} condition holds (with probability approaching 1). We have, according to the proof of Lemma \ref{lem:cone},
\begin{align*}
\mathcal{E}_{n} (\bar{\delta}) & \geq \nabla_{\gamma} \mathcal{L}_{n} (\bar{\gamma})^{\top} \bar{\delta} + \frac{\kappa}{2} \Vert \bar{\delta} \Vert_{2}^{2} - \bar{\kappa}^{2} \Vert \bar{\delta} \Vert_{1}^{2} + C \lambda \left( \Vert \bar{\delta} + \bar{\gamma} \Vert_{1} - \Vert \bar{\gamma} \Vert_{1} \right) \\
& \geq \nabla_{\gamma} \mathcal{L}_{n} (\bar{\gamma})^{\top} \bar{\delta} + \frac{\kappa}{2} \Vert \bar{\delta} \Vert_{2}^{2} - \bar{\kappa}^{2} \Vert \bar{\delta} \Vert_{1}^{2} + C \lambda \left( \Vert \bar{\delta}_{\bar{J}^{c}} \Vert_{1} - \Vert \bar{\delta}_{\bar{J}} \Vert_{1} \right) \\
& \geq - \Vert \nabla_{\gamma} \mathcal{L}_{n} (\bar{\gamma}) \Vert_{\infty} \Vert \bar{\delta} \Vert_{1} + \frac{\kappa}{2} \Vert \bar{\delta} \Vert_{2}^{2} - \bar{\kappa}^{2} \Vert \bar{\delta} \Vert_{1}^{2} + C \lambda \left( \Vert \bar{\delta}_{\bar{J}^{c}} \Vert_{1} - \Vert \bar{\delta}_{\bar{J}} \Vert_{1} \right) \\
& \gtrsim - \sqrt{\frac{\log p}{n}} \Vert \bar{\delta} \Vert_{1} + \frac{\kappa}{2} \Vert \bar{\delta} \Vert_{2}^{2} - \bar{\kappa}^{2} \Vert \bar{\delta} \Vert_{1}^{2} + C \lambda \left( \Vert \bar{\delta}_{\bar{J}^{c}} \Vert_{1} - \Vert \bar{\delta}_{\bar{J}} \Vert_{1} \right) \\
& \gtrsim \frac{\kappa}{2} \Vert \bar{\delta} \Vert_{2}^{2} - \bar{\kappa}^{2} \Vert \bar{\delta} \Vert_{1}^{2} - \lambda \Vert \bar{\delta}_{\bar{J}} \Vert_{1} \\
& \gtrsim \frac{\kappa}{2} \Vert \bar{\delta} \Vert_{2}^{2} - \bar{\kappa}^{2} |\bar{J}| \Vert \bar{\delta}_{\bar{J}} \Vert_{2}^{2} - \lambda |\bar{J}|^{1 / 2} \Vert \bar{\delta}_{\bar{J}} \Vert_{2} \gtrsim \kappa^{2} \Vert \bar{\delta} \Vert_{2}^{2} - \lambda |\bar{J}|^{1 / 2} \Vert \bar{\delta}_{\bar{J}} \Vert_{2},
\end{align*}
where the second but last inequality also follows from Lemma \ref{lem:cone} and the last inequality follows from Lemma \ref{lem:pop}.2.

Then $\mathcal{E}_{n} (\bar{\delta}) > 0$ once $\Vert \bar{\delta} \Vert_{2} = c \lambda |\bar{J}|^{1 / 2}$ for some $c > 0$ sufficiently large, which is also the corresponding value that $\rho$ should take. Then by invoking Lemma 9.21 of \citet{wainwright2019high}, we must have $$\Vert \delta \Vert_{2} \lesssim \rho \asymp \lambda |\bar{J}|^{1 / 2} = \sqrt{\frac{\log p}{n}} \left( \sqrt{\frac{\log p}{n}} \right)^{- \frac{1}{2 \xi_{\pi} + 1}} \lesssim \left( \sqrt{\frac{n}{\log p}} \right)^{- \frac{2 \xi_{\pi}}{2 \xi_{\pi} + 1}}.$$
Then the rate of convergence of $\Vert \delta \Vert_{1}$ is a direct consequence of the above result and \eqref{l1-l2} below:
\begin{equation}
\label{l1-l2}
\Vert \delta \Vert_{1} \leq \Vert \delta_{\bar{J}} \Vert_{1} + \Vert \delta_{\bar{J}^{c}} \Vert_{1} \lesssim \Vert \delta_{\bar{J}} \Vert_{1} \leq \vert \bar{J} \vert^{1 / 2} \Vert \delta_{\bar{J}} \Vert_{2} \leq \vert \bar{J} \vert^{1 / 2} \Vert \delta \Vert_{2}.
\end{equation}
\end{proof}

\begin{remark}
The $\ell_{1}$- and $\ell_{2}$-convergence rates established in this section can also be extended to more general situations, as long as the underlying penalized loss function satisfies the RSC condition and the cone condition similar to that in Lemma \ref{lem:cone} is met.
\end{remark}

\section{Additional results on simulation studies}
\label{app:sim}

\subsection{More details on the simulation studies}
\label{app:details}

For the DCal estimator, we choose the number of sample splits as $K = 6$ and conduct cross-fitting. The initial estimators for the PS and OR are computed by calling the $\texttt{cv.glmnet}$ function in R with $\texttt{alpha = 0.9}$ and the tuning parameters are selected by using 10-fold cross-validation and the 1 standard error rule. To make a fair comparison, we use the same initial PS and OR estimates for the g-formula, IPW, AIPW, TMLE, ARB, hdCBPS, and DCal estimators. For hdCBPS, we use the function $\texttt{hdCBPS}$ in the $\texttt{CBPS}$ R package but fix the initial PS and OR estimates. As the optimization algorithm used in $\texttt{hdCBPS}$ is computationally demanding, we limit the number of iterations to 100. For RCAL, we directly call the function $\texttt{glm.regu.cv}$ from the $\texttt{RCAL}$ R package, with the number of folds for cross-validation equal to 10 and the total number of possible tuning parameter values equal to 100 for the calibrated estimation of PS and OR. In order to compare the computation time of DCal with those of hdCBPS and RCAL, we execute the programs on a server powered by the x86 architecture, specifically equipped with 2 Intel Xeon ICX Platinum 8358 processors (2.6GHz, 32 cores). For all the other tuning parameters or setups of the competing methods, we set them to the default options as in their respective R packages. Finally, we remark that for the simulation settings to be shown in Sections \ref{app:sparse ps supp} and \ref{app:sparse or supp}, we only show the simulation results for $\text{RCAL}^\star$, which reduces the number of folds for cross-validation to 5 and the total number of possible tuning parameter values to 11 and hence runs faster. The statistical performance for the faster version of RCAL is not significantly deteriorated (see \href{https://github.com/Cinbo-Wang/DCal}{the accompanied GitHub website} for more details).

\subsection{Case I: Sparse PS \& Dense OR}
\label{app:sparse ps}

\subsubsection{Supplementary tables for the simulation setting from the main text}
Tables \ref{tab: sparse_PS_measure_CI} and \ref{tab: sparse_PS_measure} below summarize the coverage probabilities and the lengths of 95\% confidence intervals, Mean/Median Absolute Biases, standard errors, and Root-Mean-Squared-Errors (RMSE) of different estimators for the Sparse PS \& Dense OR simulation setting from the Section \ref{sec:sparse ps} of the main text.

\begin{longtable}{@{}cccccccccc@{}}
\caption{Coverage probability of 95\% confidence intervals (Coverage), and length of 95\% confidence intervals (CI length) for estimating $\tau$ under the \textit{sparse PS \& dense OR} setting from the main text.}
\label{tab: sparse_PS_measure_CI}\\
\toprule
\textbf{Setting} & \textbf{measure} & \textbf{g-formula} & \textbf{IPW} & \textbf{AIPW} & \textbf{TMLE} & \textbf{ARB} & \textbf{hdCBPS} & \textbf{RCAL} & \textbf{DCal} \\* \midrule
\endfirsthead
\multicolumn{10}{c}%
{{\bfseries Table \thetable\ continued from previous page}} \\
\toprule
\endhead
\cmidrule(l){2-10}
\endfoot
\endlastfoot
\multirow{2}{*}{\begin{tabular}[c]{@{}c@{}}$s=10$,  $ p=400$,\\      $n=200$\end{tabular}} & CP & 0.0\% & 57.0\% & 15.0\% & 90.5\% & 92.0\% & 53.5\% & 32.0\% & 96.0\% \\
 & Length & 0.06 & 2.28 & 1.21 & 2.64 & 1.92 & 1.98 & 1.27 & 2.09 \\
\multirow{2}{*}{\begin{tabular}[c]{@{}c@{}}$s=10$,  $ p=800$,\\      $n=400$\end{tabular}} & CP & 0.5\% & 28.5\% & 8.5\% & 93.5\% & 82.0\% & 60.5\% & 15.5\% & 91.5\% \\
 & Length & 0.05 & 1.77 & 0.97 & 2.06 & 1.45 & 1.81 & 0.94 & 1.69 \\
\multirow{2}{*}{\begin{tabular}[c]{@{}c@{}}$s=10$,  $ p=1000$,\\      $n=800$\end{tabular}} & CP & 0.0\% & 18.0\% & 11.5\% & 97.5\% & 72.0\% & 88.0\% & 22.0\% & 89.0\% \\
 & Length & 0.09 & 1.38 & 0.75 & 1.62 & 1.05 & 1.71 & 0.77 & 1.28 \\
\multirow{2}{*}{\begin{tabular}[c]{@{}c@{}}$s=10$,  $ p=1000$,\\      $n=1600$\end{tabular}} & CP & 0.0\% & 15.5\% & 18.0\% & 99.5\% & 49.0\% & 98.5\% & 24.1\% & 75.5\% \\
 & Length & 0.10 & 1.12 & 0.59 & 1.19 & 0.73 & 1.33 & 0.58 & 0.93 \\
 \hline
\multirow{2}{*}{\begin{tabular}[c]{@{}c@{}}$s=20$,  $ p=400$,\\      $n=200$\end{tabular}} & CP & 0.0\% & 38.0\% & 11.5\% & 98.5\% & 95.0\% & 47.0\% & 32.5\% & 96.0\% \\
 & Length & 0.05 & 2.21 & 1.15 & 3.75 & 2.04 & 1.92 & 1.27 & 2.18 \\
\multirow{2}{*}{\begin{tabular}[c]{@{}c@{}}$s=20$,  $ p=800$,\\      $n=400$\end{tabular}} & CP & 0.0\% & 15.0\% & 6.5\% & 97.5\% & 92.0\% & 64.0\% & 22.0\% & 93.0\% \\
 & Length & 0.06 & 1.61 & 0.87 & 2.79 & 1.49 & 1.68 & 0.92 & 1.65 \\
\multirow{2}{*}{\begin{tabular}[c]{@{}c@{}}$s=20$,  $ p=1000$,\\      $n=800$\end{tabular}} & CP & 0.0\% & 10.0\% & 7.5\% & 100.0\% & 85.5\% & 82.5\% & 12.0\% & 95.5\% \\
 & Length & 0.08 & 1.29 & 0.72 & 2.28 & 1.11 & 1.61 & 0.68 & 1.34 \\
\multirow{2}{*}{\begin{tabular}[c]{@{}c@{}}$s=20$,  $ p=1000$,\\      $n=1600$\end{tabular}} & CP & 0.0\% & 7.5\% & 6.0\% & 99.5\% & 82.5\% & 92.5\% & 20.9\% & 93.5\% \\
 & Length & 0.09 & 1.03 & 0.56 & 1.79 & 0.77 & 1.32 & 0.57 & 0.97 \\  \hline
\multirow{2}{*}{\begin{tabular}[c]{@{}c@{}}$s=50$,  $ p=400$,\\      $n=200$\end{tabular}} & CP & 0.5\% & 43.0\% & 10.0\% & 99.0\% & 94.0\% & 62.0\% & 31.0\% & 96.5\% \\
 & Length & 0.09 & 2.15 & 1.08 & 4.96 & 2.04 & 1.77 & 1.26 & 2.11 \\
\multirow{2}{*}{\begin{tabular}[c]{@{}c@{}}$s=50$,  $ p=800$,\\      $n=400$\end{tabular}} & CP & 0.0\% & 17.0\% & 7.0\% & 99.5\% & 94.0\% & 79.5\% & 37.0\% & 95.5\% \\
 & Length & 0.10 & 1.60 & 0.82 & 3.82 & 1.50 & 1.48 & 0.93 & 1.61 \\
\multirow{2}{*}{\begin{tabular}[c]{@{}c@{}}$s=50$,  $ p=1000$,\\      $n=800$\end{tabular}} & CP & 0.0\% & 8.5\% & 4.0\% & 100.0\% & 93.5\% & 89.5\% & 23.0\% & 97.5\% \\
 & Length & 0.11 & 1.22 & 0.63 & 3.05 & 1.11 & 1.16 & 0.69 & 1.25 \\
\multirow{2}{*}{\begin{tabular}[c]{@{}c@{}}$s=50$,  $ p=1000$,\\      $n=1600$\end{tabular}} & CP & 0.0\% & 7.0\% & 5.5\% & 98.0\% & 94.5\% & 85.5\% & 13.5\% & 98.5\% \\
 & Length & 0.11 & 0.99 & 0.51 & 2.44 & 0.80 & 0.91 & 0.51 & 0.94 \\* \bottomrule
\end{longtable}

\begin{longtable}{cccccccccc}
\caption{Mean Absolute Bias ($\text{MAB}^\dag$), Median Absolute Bias ($\text{MAB}^\ddag$), standard error (std), root-mean-squared error (RMSE) for  estimating $\tau$ under the \textit{sparse PS \& dense OR} setting from the main text.}
\label{tab: sparse_PS_measure}\\
\toprule
\textbf{Setting} & \textbf{measure} & \textbf{g-formula} & \textbf{IPW} & \textbf{AIPW} & \textbf{TMLE} & \textbf{ARB} & \textbf{hdCBPS} & \textbf{RCAL} & \textbf{DCal} \\
\midrule
\endfirsthead
\multicolumn{10}{c}%
{{\bfseries Table \thetable\ continued from previous page}} \\
\midrule
\endhead
\multirow{4}{*}{\begin{tabular}[c]{@{}c@{}}$s=10$,   \\     $ p=400$,\\     $ n=200$\end{tabular}} & $\text{MAB}^\dag$ & 1.34 & 1.06 & 0.98 & 0.70 & 0.53 & 0.78 & 0.78 & 0.47 \\
 & $\text{MAB}^\ddag$ & 1.39 & 1.08 & 1.00 & 0.71 & 0.50 & 0.80 & 0.80 & 0.41 \\
 & std & 0.02 & 0.58 & 0.31 & 0.67 & 0.49 & 0.50 & 0.32 & 0.53 \\
 & RMSE & 1.38 & 1.09 & 1.03 & 0.79 & 0.64 & 0.88 & 0.87 & 0.60 \\
 \hline 
\multirow{4}{*}{\begin{tabular}[c]{@{}c@{}}$s=10$,   \\     $ p=800$,\\     $ n=400$\end{tabular}} & $\text{MAB}^\dag$ & 1.35 & 1.02 & 0.92 & 0.56 & 0.48 & 0.60 & 0.77 & 0.40 \\
 & $\text{MAB}^\ddag$ & 1.41 & 1.02 & 0.93 & 0.54 & 0.44 & 0.56 & 0.80 & 0.33 \\
 & std & 0.01 & 0.45 & 0.25 & 0.53 & 0.37 & 0.46 & 0.24 & 0.43 \\
 & RMSE & 1.39 & 1.05 & 0.97 & 0.63 & 0.57 & 0.74 & 0.84 & 0.50 \\
 \hline \multirow{4}{*}{\begin{tabular}[c]{@{}c@{}}$s=10$,   \\     $ p=1000$,\\     $ n=800$\end{tabular}} & $\text{MAB}^\dag$ & 1.11 & 0.91 & 0.70 & 0.32 & 0.40 & 0.32 & 0.61 & 0.33 \\
 & $\text{MAB}^\ddag$ & 1.11 & 0.93 & 0.71 & 0.28 & 0.39 & 0.23 & 0.59 & 0.30 \\
 & std & 0.02 & 0.35 & 0.19 & 0.41 & 0.27 & 0.44 & 0.20 & 0.33 \\
 & RMSE & 1.15 & 0.93 & 0.75 & 0.39 & 0.47 & 0.44 & 0.67 & 0.40 \\
 \hline \multirow{4}{*}{\begin{tabular}[c]{@{}c@{}}$s=10$,   \\     $ p=1000$,\\     $ n=1600$\end{tabular}} & $\text{MAB}^\dag$ & 0.87 & 0.81 & 0.51 & 0.20 & 0.37 & 0.20 & 0.45 & 0.34 \\
 & $\text{MAB}^\ddag$ & 0.86 & 0.82 & 0.51 & 0.18 & 0.37 & 0.14 & 0.43 & 0.34 \\
 & std & 0.03 & 0.29 & 0.15 & 0.30 & 0.19 & 0.34 & 0.15 & 0.24 \\
 & RMSE & 0.89 & 0.83 & 0.55 & 0.25 & 0.41 & 0.26 & 0.50 & 0.38 \\
 \hline \multirow{4}{*}{\begin{tabular}[c]{@{}c@{}}$s=20$,   \\     $ p=400$,\\     $ n=200$\end{tabular}} & $\text{MAB}^\dag$ & 1.44 & 1.15 & 1.08 & 0.66 & 0.51 & 0.84 & 0.84 & 0.47 \\
 & $\text{MAB}^\ddag$ & 1.48 & 1.15 & 1.11 & 0.62 & 0.48 & 0.87 & 0.82 & 0.36 \\
 & std & 0.01 & 0.56 & 0.29 & 0.96 & 0.52 & 0.49 & 0.32 & 0.56 \\
 & RMSE & 1.49 & 1.20 & 1.15 & 0.77 & 0.62 & 0.96 & 0.95 & 0.59 \\
 \hline \multirow{4}{*}{\begin{tabular}[c]{@{}c@{}}$s=20$,   \\     $ p=800$,\\     $ n=400$\end{tabular}} & $\text{MAB}^\dag$ & 1.36 & 1.06 & 0.94 & 0.40 & 0.40 & 0.59 & 0.70 & 0.34 \\
 & $\text{MAB}^\ddag$ & 1.38 & 1.05 & 0.94 & 0.35 & 0.36 & 0.53 & 0.69 & 0.29 \\
 & std & 0.02 & 0.41 & 0.22 & 0.71 & 0.38 & 0.43 & 0.23 & 0.42 \\
 & RMSE & 1.39 & 1.09 & 0.99 & 0.50 & 0.49 & 0.71 & 0.79 & 0.44 \\
 \hline \multirow{4}{*}{\begin{tabular}[c]{@{}c@{}}$s=20$,   \\     $ p=1000$,\\     $ n=800$\end{tabular}} & $\text{MAB}^\dag$ & 1.26 & 0.98 & 0.80 & 0.29 & 0.33 & 0.36 & 0.62 & 0.28 \\
 & $\text{MAB}^\ddag$ & 1.27 & 0.98 & 0.81 & 0.25 & 0.29 & 0.27 & 0.59 & 0.23 \\
 & std & 0.02 & 0.33 & 0.18 & 0.58 & 0.28 & 0.41 & 0.17 & 0.34 \\
 & RMSE & 1.28 & 1.00 & 0.84 & 0.38 & 0.40 & 0.47 & 0.68 & 0.35 \\
 \hline \multirow{4}{*}{\begin{tabular}[c]{@{}c@{}}$s=20$,   \\     $ p=1000$,\\     $ n=1600$\end{tabular}} & $\text{MAB}^\dag$ & 1.03 & 0.89 & 0.61 & 0.32 & 0.26 & 0.25 & 0.46 & 0.22 \\
 & $\text{MAB}^\ddag$ & 1.00 & 0.90 & 0.60 & 0.29 & 0.24 & 0.21 & 0.46 & 0.19 \\
 & std & 0.02 & 0.26 & 0.14 & 0.46 & 0.20 & 0.34 & 0.14 & 0.25 \\
 & RMSE & 1.05 & 0.91 & 0.65 & 0.38 & 0.30 & 0.31 & 0.52 & 0.27 \\
 \hline \multirow{4}{*}{\begin{tabular}[c]{@{}c@{}}$s=50$,   \\     $ p=400$,\\     $ n=200$\end{tabular}} & $\text{MAB}^\dag$ & 1.24 & 1.09 & 1.00 & 0.60 & 0.51 & 0.71 & 0.77 & 0.47 \\
 & $\text{MAB}^\ddag$ & 1.25 & 1.09 & 1.04 & 0.56 & 0.45 & 0.67 & 0.78 & 0.40 \\
 & std & 0.02 & 0.55 & 0.27 & 1.26 & 0.52 & 0.45 & 0.32 & 0.54 \\
 & RMSE & 1.29 & 1.13 & 1.07 & 0.74 & 0.62 & 0.84 & 0.87 & 0.59 \\
 \hline \multirow{4}{*}{\begin{tabular}[c]{@{}c@{}}$s=50$,   \\     $ p=800$,\\     $ n=400$\end{tabular}} & $\text{MAB}^\dag$ & 1.14 & 0.98 & 0.85 & 0.42 & 0.37 & 0.44 & 0.58 & 0.34 \\
 & $\text{MAB}^\ddag$ & 1.17 & 0.98 & 0.87 & 0.36 & 0.31 & 0.36 & 0.58 & 0.27 \\
 & std & 0.02 & 0.41 & 0.21 & 0.97 & 0.38 & 0.38 & 0.24 & 0.41 \\
 & RMSE & 1.17 & 1.01 & 0.89 & 0.54 & 0.45 & 0.55 & 0.66 & 0.43 \\
 \hline \multirow{4}{*}{\begin{tabular}[c]{@{}c@{}}$s=50$,   \\     $ p=1000$,\\     $ n=800$\end{tabular}} & $\text{MAB}^\dag$ & 1.06 & 0.92 & 0.73 & 0.40 & 0.25 & 0.32 & 0.49 & 0.23 \\
 & $\text{MAB}^\ddag$ & 1.07 & 0.93 & 0.73 & 0.32 & 0.20 & 0.30 & 0.49 & 0.18 \\
 & std & 0.03 & 0.31 & 0.16 & 0.78 & 0.28 & 0.30 & 0.18 & 0.32 \\
 & RMSE & 1.08 & 0.94 & 0.76 & 0.51 & 0.31 & 0.38 & 0.54 & 0.29 \\
 \hline \multirow{4}{*}{\begin{tabular}[c]{@{}c@{}}$s=50$,   \\     $ p=1000$,\\     $ n=1600$\end{tabular}} & $\text{MAB}^\dag$ & 0.94 & 0.84 & 0.60 & 0.48 & 0.19 & 0.26 & 0.45 & 0.18 \\
 & $\text{MAB}^\ddag$ & 0.94 & 0.88 & 0.61 & 0.45 & 0.16 & 0.25 & 0.47 & 0.16 \\
 & std & 0.03 & 0.25 & 0.13 & 0.62 & 0.20 & 0.23 & 0.13 & 0.24 \\
 & RMSE & 0.95 & 0.86 & 0.62 & 0.57 & 0.23 & 0.30 & 0.48 & 0.22 \\
 \bottomrule
\end{longtable}

\subsubsection{Supplementary simulation settings}
\label{app:sparse ps supp}

Here, we modify the simulation setting in Section \ref{sec:sparse ps} by only changing $\gamma_j \sim \mathrm{Uniform}([-2,-1]\cup[1,2])$, while keeping the other parameters the same as in the main text. This modification alleviates the near violation of the positivity assumption, as demonstrated in Figure \ref{fig: hist_propensity_sparse_PS}, showcasing the histogram of the true PS's for the case $n = 800$, $p = 1000$, and $s = 10$.
\begin{figure}[htpb]
    \centering    
    \includegraphics[page=1,scale=0.6]{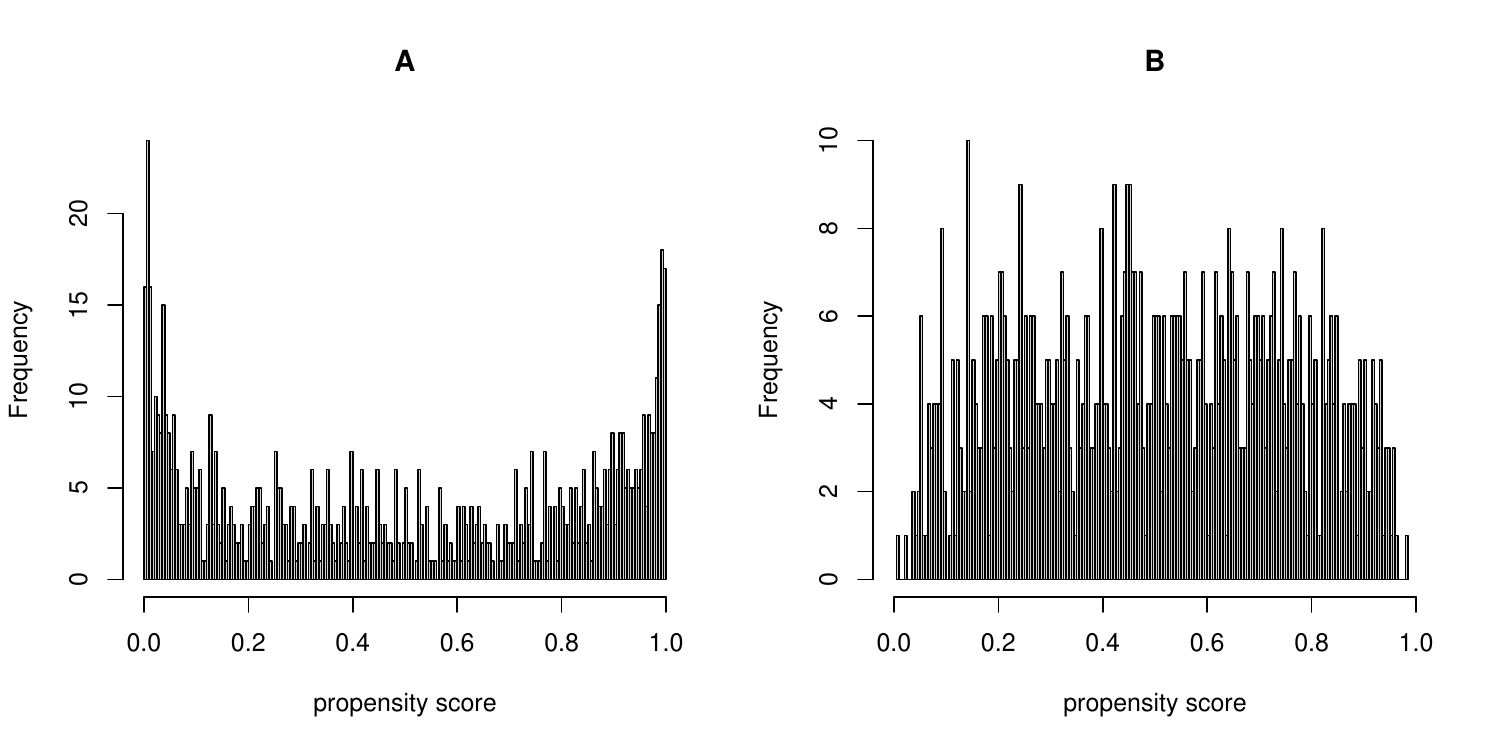}
    \caption{Histograms of the propensity score under the \textit{sparse PS \& dense OR} setting. The left figure (A) represents the setting described in the main text, while the right figure (B) represents the setting described in Appendix \ref{app:sparse ps}. Here, we chose a scenario with $n=800, p=1000 \text{ and } s=10$.}
    \label{fig: hist_propensity_sparse_PS}
\end{figure}

The corresponding simulation results are displayed in Figure \ref{fig: simu_sparse_PS_supp}, 
Table \ref{tab: sparse_PS_measure_CI_supp}, and Table \ref{tab: sparse_PS_measure_supp}. Root-N scaled absolute estimation errors are shown in Figure \ref{fig: simu_sparse_PS_supp}. Tables \ref{tab: sparse_PS_measure_CI_supp} and \ref{tab: sparse_PS_measure_supp} below summarize the coverage probabilities and the lengths of 95\% confidence intervals, Mean/Median Absolute Biases, standard errors, and Root-Mean-Errors (RMSE) of different estimators for this setting. Reading from these results, the performance of $\hat{\tau}_{\text{RCAL}}$, $\hat{\tau}_{\text{hdCBPS}}$, and $\hat{\tau}_{\text{DCal}}$ are quite similar. 

\begin{figure}[htpb]
    \centering
    \includegraphics[width = \textwidth, page=1]{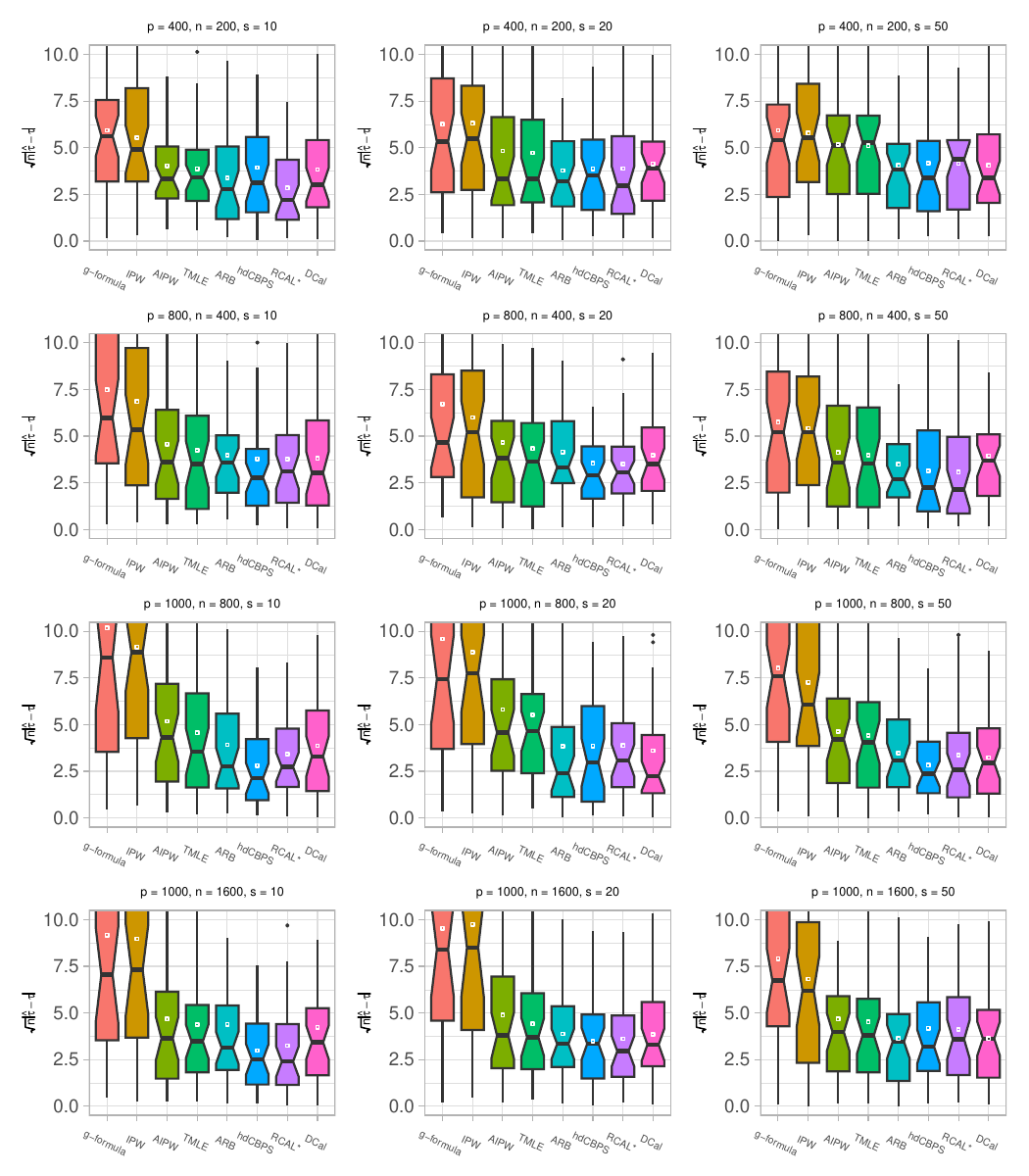}
    \caption{Boxplots of the $\sqrt{n}$-scaled estimation error $\sqrt{n}|\hat{\tau}-\tau|$ under different sample size $n$, dimension $p$ and sparsity level $s$ under the \textit{sparse PS \& dense OR} setting. The white dots correspond to mean.}
    \label{fig: simu_sparse_PS_supp}
\end{figure}

\begin{longtable}{@{}cccccccccc@{}}
\caption{Coverage probabilities of 95\% confidence intervals (Coverage), and length of 95\% confidence intervals (CI length) for estimating $\tau$ under the \textit{Sparse PS \& Dense OR} setting.}
\label{tab: sparse_PS_measure_CI_supp}\\
\toprule
\textbf{Setting} & \textbf{measure} & \textbf{g-formula} & \textbf{IPW} & \textbf{AIPW} & \textbf{TMLE} & \textbf{ARB} & \textbf{hdCBPS} & $\textbf{RCAL}^\star \footnotemark$ & \textbf{DCal} \\* \midrule
\endfirsthead
\multicolumn{10}{c}%
{{\bfseries Table \thetable\ continued from previous page}} \\
\toprule
\endhead
\cmidrule(l){2-10}
\endfoot
\endlastfoot
\multirow{2}{*}{\begin{tabular}[c]{@{}c@{}}$s=10$,  $ p=400$,\\      $n=200$\end{tabular}} & CP & 17.5\% & 100.0\% & 97.5\% & 100.0\% & 100.0\% & 100.0\% & 100.0\% & 100.0\% \\
 & Length & 0.23 & 3.06 & 1.38 & 1.42 & 1.58 & 1.57 & 1.29 & 1.71 \\
\multirow{2}{*}{\begin{tabular}[c]{@{}c@{}}$s=10$,  $ p=800$,\\      $n=400$\end{tabular}} & CP & 25.0\% & 97.5\% & 85.0\% & 95.0\% & 97.5\% & 92.3\% & 94.9\% & 97.5\% \\
 & Length & 0.22 & 2.21 & 1.01 & 1.03 & 1.14 & 1.17 & 0.96 & 1.22 \\
\multirow{2}{*}{\begin{tabular}[c]{@{}c@{}}$s=10$,  $ p=1000$,\\      $n=800$\end{tabular}} & CP & 20.0\% & 95.0\% & 85.0\% & 92.5\% & 95.0\% & 100.0\% & 95.0\% & 95.0\% \\
 & Length & 0.18 & 1.58 & 0.71 & 0.73 & 0.77 & 0.83 & 0.69 & 0.86 \\
\multirow{2}{*}{\begin{tabular}[c]{@{}c@{}}$s=10$,  $ p=1000$,\\      $n=1600$\end{tabular}} & CP & 27.5\% & 95.0\% & 85.0\% & 95.0\% & 97.5\% & 100.0\% & 97.5\% & 97.5\% \\
 & Length & 0.16 & 1.12 & 0.50 & 0.50 & 0.52 & 0.54 & 0.50 & 0.59 \\  \hline
\multirow{2}{*}{\begin{tabular}[c]{@{}c@{}}$s=20$,  $ p=400$,\\      $n=200$\end{tabular}} & CP & 12.5\% & 100.0\% & 85.0\% & 90.0\% & 100.0\% & 97.5\% & 97.5\% & 100.0\% \\
 & Length & 0.19 & 3.08 & 1.44 & 1.51 & 1.67 & 1.59 & 1.36 & 1.78 \\
\multirow{2}{*}{\begin{tabular}[c]{@{}c@{}}$s=20$,  $ p=800$,\\      $n=400$\end{tabular}} & CP & 20.0\% & 100.0\% & 92.5\% & 95.0\% & 95.0\% & 95.0\% & 95.0\% & 100.0\% \\
 & Length & 0.19 & 2.22 & 1.02 & 1.05 & 1.11 & 1.10 & 0.97 & 1.19 \\
\multirow{2}{*}{\begin{tabular}[c]{@{}c@{}}$s=20$,  $ p=1000$,\\      $n=800$\end{tabular}} & CP & 15.0\% & 95.0\% & 82.5\% & 87.5\% & 95.0\% & 97.5\% & 92.5\% & 97.5\% \\
 & Length & 0.17 & 1.57 & 0.72 & 0.73 & 0.79 & 0.78 & 0.71 & 0.88 \\
\multirow{2}{*}{\begin{tabular}[c]{@{}c@{}}$s=20$,  $ p=1000$,\\      $n=1600$\end{tabular}} & CP & 22.5\% & 95.0\% & 92.5\% & 95.0\% & 100.0\% & 100.0\% & 95.0\% & 100.0\% \\
 & Length & 0.16 & 1.13 & 0.50 & 0.50 & 0.52 & 0.53 & 0.50 & 0.59 \\ \hline
\multirow{2}{*}{\begin{tabular}[c]{@{}c@{}}$s=50$,  $ p=400$,\\      $n=200$\end{tabular}} & CP & 20.0\% & 100.0\% & 92.5\% & 95.0\% & 97.5\% & 95.0\% & 95.0\% & 97.5\% \\
 & Length & 0.23 & 3.07 & 1.43 & 1.47 & 1.64 & 1.48 & 1.34 & 1.74 \\
\multirow{2}{*}{\begin{tabular}[c]{@{}c@{}}$s=50$,  $ p=800$,\\      $n=400$\end{tabular}} & CP & 32.5\% & 100.0\% & 95.0\% & 97.5\% & 100.0\% & 97.5\% & 100.0\% & 100.0\% \\
 & Length & 0.22 & 2.19 & 1.01 & 1.02 & 1.12 & 1.04 & 0.98 & 1.23 \\
\multirow{2}{*}{\begin{tabular}[c]{@{}c@{}}$s=50$,  $ p=1000$,\\      $n=800$\end{tabular}} & CP & 20.0\% & 100.0\% & 95.0\% & 95.0\% & 100.0\% & 100.0\% & 100.0\% & 100.0\% \\
 & Length & 0.18 & 1.56 & 0.71 & 0.71 & 0.75 & 0.72 & 0.69 & 0.83 \\
\multirow{2}{*}{\begin{tabular}[c]{@{}c@{}}$s=50$,  $ p=1000$,\\      $n=1600$\end{tabular}} & CP & 12.5\% & 100.0\% & 90.0\% & 90.0\% & 97.5\% & 95.0\% & 92.5\% & 100.0\% \\
 & Length & 0.16 & 1.11 & 0.49 & 0.48 & 0.51 & 0.49 & 0.49 & 0.57 \\* \bottomrule
\end{longtable}
\footnotetext{$\text{RCAL}^\star$ refers to RCAL with a five-fold cross-validation and eleven possible values for the tuning parameters. Due to time constraints, we only consider this quicker setting.}

\begin{longtable}{cccccccccc}
\caption{Mean Absolute Bias ($\text{MAB}^\dag$), Median Absolute Bias ($\text{MAB}^\ddag$), standard error (std), root-mean-squared error (RMSE) for estimating $\tau$ under the \textit{Sparse PS \& Dense OR} setting.}
\label{tab: sparse_PS_measure_supp}\\
\toprule
\textbf{Setting} & \textbf{measure} & \textbf{g-formula} & \textbf{IPW} & \textbf{AIPW} & \textbf{TMLE} & \textbf{ARB} & \textbf{hdCBPS} & $\textbf{RCAL}^\star$ & \textbf{DCal} \\
\midrule
\endfirsthead
\multicolumn{10}{c}%
{{\bfseries Table \thetable\ continued from previous page}} \\
\midrule
\endhead
\multirow{4}{*}{\begin{tabular}[c]{@{}c@{}}$s=10$,   \\   $   p=400$,\\ $     n=200$\end{tabular}} & $\text{MAB}^\dag$ & 0.42 & 0.39 & 0.28 & 0.27 & 0.24 & 0.28 & 0.20 & 0.27 \\
 &$\text{MAB}^\ddag$ & 0.40 & 0.35 & 0.24 & 0.24 & 0.20 & 0.22 & 0.16 & 0.21 \\
 & std & 0.06 & 0.78 & 0.35 & 0.36 & 0.40 & 0.40 & 0.33 & 0.44 \\
 & RMSE & 0.50 & 0.46 & 0.34 & 0.32 & 0.30 & 0.33 & 0.25 & 0.33 \\
\hline \multirow{4}{*}{\begin{tabular}[c]{@{}c@{}}$s=10$,   \\   $   p=800$,\\ $     n=400$\end{tabular}} & $\text{MAB}^\dag$ & 0.37 & 0.34 & 0.23 & 0.21 & 0.20 & 0.19 & 0.19 & 0.19 \\
 &$\text{MAB}^\ddag$ & 0.30 & 0.27 & 0.18 & 0.17 & 0.18 & 0.14 & 0.16 & 0.15 \\
 & std & 0.06 & 0.57 & 0.26 & 0.26 & 0.29 & 0.30 & 0.24 & 0.31 \\
 & RMSE & 0.48 & 0.43 & 0.30 & 0.28 & 0.24 & 0.25 & 0.24 & 0.25 \\
\hline \multirow{4}{*}{\begin{tabular}[c]{@{}c@{}}$s=10$,   \\   $   p=1000$,\\ $     n=800$\end{tabular}} & $\text{MAB}^\dag$ & 0.36 & 0.32 & 0.18 & 0.16 & 0.14 & 0.10 & 0.12 & 0.14 \\
 &$\text{MAB}^\ddag$ & 0.30 & 0.31 & 0.15 & 0.13 & 0.10 & 0.07 & 0.10 & 0.12 \\
 & std & 0.05 & 0.40 & 0.18 & 0.19 & 0.20 & 0.21 & 0.18 & 0.22 \\
 & RMSE & 0.45 & 0.38 & 0.24 & 0.21 & 0.18 & 0.13 & 0.15 & 0.18 \\
\hline \multirow{4}{*}{\begin{tabular}[c]{@{}c@{}}$s=10$,   \\   $   p=1000$,\\ $     n=1600$\end{tabular}} & $\text{MAB}^\dag$ & 0.23 & 0.22 & 0.12 & 0.11 & 0.11 & 0.07 & 0.08 & 0.11 \\
 &$\text{MAB}^\ddag$ & 0.18 & 0.18 & 0.09 & 0.09 & 0.08 & 0.06 & 0.06 & 0.09 \\
 & std & 0.04 & 0.29 & 0.13 & 0.13 & 0.13 & 0.14 & 0.13 & 0.15 \\
 & RMSE & 0.30 & 0.28 & 0.16 & 0.14 & 0.14 & 0.10 & 0.12 & 0.14 \\
\hline \multirow{4}{*}{\begin{tabular}[c]{@{}c@{}}$s=20$,   \\   $   p=400$,\\ $     n=200$\end{tabular}} & $\text{MAB}^\dag$ & 0.44 & 0.45 & 0.34 & 0.33 & 0.27 & 0.27 & 0.27 & 0.29 \\
 &$\text{MAB}^\ddag$ & 0.38 & 0.39 & 0.24 & 0.24 & 0.23 & 0.25 & 0.21 & 0.28 \\
 & std & 0.05 & 0.79 & 0.37 & 0.39 & 0.43 & 0.41 & 0.35 & 0.45 \\
 & RMSE & 0.56 & 0.55 & 0.44 & 0.43 & 0.33 & 0.34 & 0.35 & 0.35 \\
\hline \multirow{4}{*}{\begin{tabular}[c]{@{}c@{}}$s=20$,   \\   $   p=800$,\\ $     n=400$\end{tabular}} & $\text{MAB}^\dag$ & 0.34 & 0.30 & 0.23 & 0.22 & 0.21 & 0.18 & 0.18 & 0.20 \\
 &$\text{MAB}^\ddag$ & 0.23 & 0.26 & 0.19 & 0.18 & 0.17 & 0.15 & 0.15 & 0.18 \\
 & std & 0.05 & 0.57 & 0.26 & 0.27 & 0.28 & 0.28 & 0.25 & 0.30 \\
 & RMSE & 0.45 & 0.39 & 0.31 & 0.29 & 0.24 & 0.24 & 0.22 & 0.24 \\
\hline \multirow{4}{*}{\begin{tabular}[c]{@{}c@{}}$s=20$,   \\   $   p=1000$,\\ $     n=800$\end{tabular}} & $\text{MAB}^\dag$ & 0.34 & 0.31 & 0.21 & 0.19 & 0.14 & 0.14 & 0.14 & 0.13 \\
 &$\text{MAB}^\ddag$ & 0.26 & 0.27 & 0.16 & 0.16 & 0.08 & 0.10 & 0.11 & 0.08 \\
 & std & 0.04 & 0.40 & 0.18 & 0.19 & 0.20 & 0.20 & 0.18 & 0.22 \\
 & RMSE & 0.44 & 0.38 & 0.27 & 0.25 & 0.19 & 0.19 & 0.18 & 0.18 \\
\hline \multirow{4}{*}{\begin{tabular}[c]{@{}c@{}}$s=20$,   \\   $   p=1000$,\\ $     n=1600$\end{tabular}} & $\text{MAB}^\dag$ & 0.24 & 0.24 & 0.12 & 0.11 & 0.10 & 0.09 & 0.09 & 0.10 \\
 &$\text{MAB}^\ddag$ & 0.21 & 0.21 & 0.09 & 0.09 & 0.08 & 0.08 & 0.07 & 0.08 \\
 & std & 0.04 & 0.29 & 0.13 & 0.13 & 0.13 & 0.14 & 0.13 & 0.15 \\
 & RMSE & 0.30 & 0.30 & 0.15 & 0.14 & 0.12 & 0.11 & 0.11 & 0.12 \\
\hline \multirow{4}{*}{\begin{tabular}[c]{@{}c@{}}$s=50$,   \\   $   p=400$,\\ $     n=200$\end{tabular}} & $\text{MAB}^\dag$ & 0.42 & 0.41 & 0.37 & 0.36 & 0.29 & 0.30 & 0.29 & 0.29 \\
 &$\text{MAB}^\ddag$ & 0.38 & 0.39 & 0.36 & 0.37 & 0.27 & 0.24 & 0.31 & 0.24 \\
 & std & 0.06 & 0.78 & 0.37 & 0.37 & 0.42 & 0.38 & 0.34 & 0.44 \\
 & RMSE & 0.51 & 0.48 & 0.44 & 0.43 & 0.35 & 0.37 & 0.35 & 0.35 \\
\hline \multirow{4}{*}{\begin{tabular}[c]{@{}c@{}}$s=50$,   \\   $   p=800$,\\ $     n=400$\end{tabular}} & $\text{MAB}^\dag$ & 0.29 & 0.27 & 0.21 & 0.20 & 0.17 & 0.16 & 0.15 & 0.20 \\
 &$\text{MAB}^\ddag$ & 0.26 & 0.26 & 0.18 & 0.18 & 0.14 & 0.11 & 0.11 & 0.18 \\
 & std & 0.06 & 0.56 & 0.26 & 0.26 & 0.29 & 0.27 & 0.25 & 0.31 \\
 & RMSE & 0.36 & 0.32 & 0.26 & 0.25 & 0.22 & 0.21 & 0.20 & 0.24 \\
\hline \multirow{4}{*}{\begin{tabular}[c]{@{}c@{}}$s=50$,   \\   $   p=1000$,\\ $     n=800$\end{tabular}} & $\text{MAB}^\dag$ & 0.28 & 0.26 & 0.16 & 0.16 & 0.12 & 0.10 & 0.12 & 0.11 \\
 &$\text{MAB}^\ddag$ & 0.27 & 0.21 & 0.15 & 0.14 & 0.11 & 0.08 & 0.09 & 0.10 \\
 & std & 0.04 & 0.40 & 0.18 & 0.18 & 0.19 & 0.18 & 0.18 & 0.21 \\
 & RMSE & 0.35 & 0.30 & 0.20 & 0.19 & 0.15 & 0.12 & 0.15 & 0.14 \\
\hline \multirow{4}{*}{\begin{tabular}[c]{@{}c@{}}$s=50$,   \\   $   p=1000$,\\ $     n=1600$\end{tabular}} & $\text{MAB}^\dag$ & 0.20 & 0.17 & 0.12 & 0.11 & 0.09 & 0.10 & 0.10 & 0.09 \\
 &$\text{MAB}^\ddag$ & 0.17 & 0.15 & 0.10 & 0.10 & 0.09 & 0.08 & 0.09 & 0.09 \\
 & std & 0.04 & 0.28 & 0.13 & 0.12 & 0.13 & 0.13 & 0.13 & 0.15 \\
 & RMSE & 0.24 & 0.22 & 0.15 & 0.14 & 0.11 & 0.13 & 0.13 & 0.11 \\
 \bottomrule
\end{longtable}

\subsection{Case II: Dense PS \& Sparse OR}
\label{app:sparse or}

\subsubsection{Supplementary tables for the simulation setting from the main text}

Tables \ref{tab: sparse_OR_measure_CI} and \ref{tab: sparse_OR_measure} below summarize the coverage probabilities and the lengths of 95\% confidence intervals, Mean/Median Absolute Biases, standard errors, and Root-Mean-Squared-Errors (RMSE) of different estimators for the Dense PS \& Sparse OR simulation setting from the Section \ref{sec:sparse or} of the main text.

\begin{longtable}{@{}cccccccccc@{}}
\caption{Coverage probability of 95\% confidence intervals (Coverage), and length of 95\% confidence intervals (CI length) for estimating $\tau$ under the \textit{Dense PS \& Sparse OR} setting from the main text.}
\label{tab: sparse_OR_measure_CI}\\
\toprule
\textbf{Setting} & \textbf{measure} & \textbf{g-formula} & \textbf{IPW} & \textbf{AIPW} & \textbf{TMLE} & \textbf{ARB} & \textbf{hdCBPS} & \textbf{RCAL} & \textbf{DCal} \\
\midrule
\endfirsthead
\multicolumn{10}{c}%
{{\bfseries Table \thetable\ continued from previous page}} \\
\toprule
\endhead
\multirow{2}{*}{\begin{tabular}[c]{@{}c@{}}$s=10$,  $ p=400$,\\      $n=200$\end{tabular}} & CP & 17.5\% & 37.5\% & 66.0\% & 87.0\% & 87.5\% & 94.0\% & 68.5\% & 99.0\% \\
 & Length & 0.56 & 2.21 & 0.81 & 0.99 & 0.71 & 1.25 & 0.77 & 1.03 \\
\multirow{2}{*}{\begin{tabular}[c]{@{}c@{}}$s=10$,  $ p=800$,\\      $n=400$\end{tabular}} & CP & 13.5\% & 56.0\% & 83.0\% & 91.0\% & 92.0\% & 98.5\% & 80.5\% & 98.0\% \\
 & Length & 0.42 & 2.69 & 0.73 & 0.66 & 0.50 & 0.94 & 0.56 & 0.90 \\
\multirow{2}{*}{\begin{tabular}[c]{@{}c@{}}$s=10$,  $ p=1000$,\\      $n=800$\end{tabular}} & CP & 7.0\% & 79.0\% & 92.5\% & 83.0\% & 97.5\% & 99.5\% & 84.0\% & 100.0\% \\
 & Length & 0.31 & 3.31 & 0.64 & 0.50 & 0.36 & 0.64 & 0.41 & 0.82 \\
\multirow{2}{*}{\begin{tabular}[c]{@{}c@{}}$s=10$,  $ p=1000$,\\      $n=1600$\end{tabular}} & CP & 2.0\% & 98.5\% & 98.0\% & 66.5\% & 96.5\% & 99.5\% & 90.7\% & 100.0\% \\
 & Length & 0.23 & 3.48 & 0.56 & 0.28 & 0.25 & 0.46 & 0.30 & 0.70 \\ \hline
\multirow{2}{*}{\begin{tabular}[c]{@{}c@{}}$s=20$,  $ p=400$,\\      $n=200$\end{tabular}} & CP & 31.5\% & 41.5\% & 76.0\% & 83.5\% & 90.0\% & 94.0\% & 78.5\% & 99.0\% \\
 & Length & 0.73 & 2.82 & 0.98 & 0.98 & 0.71 & 1.34 & 0.92 & 1.17 \\
\multirow{2}{*}{\begin{tabular}[c]{@{}c@{}}$s=20$,  $ p=800$,\\      $n=400$\end{tabular}} & CP & 13.5\% & 56.0\% & 84.5\% & 84.0\% & 92.0\% & 98.0\% & 78.0\% & 98.5\% \\
 & Length & 0.54 & 3.24 & 0.81 & 0.70 & 0.51 & 0.98 & 0.66 & 0.99 \\
\multirow{2}{*}{\begin{tabular}[c]{@{}c@{}}$s=20$,  $ p=1000$,\\      $n=800$\end{tabular}} & CP & 11.5\% & 74.5\% & 95.5\% & 79.5\% & 98.5\% & 98.0\% & 85.4\% & 100.0\% \\
 & Length & 0.40 & 3.86 & 0.71 & 0.49 & 0.37 & 0.68 & 0.48 & 0.90 \\
\multirow{2}{*}{\begin{tabular}[c]{@{}c@{}}$s=20$,  $ p=1000$,\\      $n=1600$\end{tabular}} & CP & 6.5\% & 96.0\% & 98.5\% & 59.5\% & 95.5\% & 100.0\% & 92.5\% & 100.0\% \\
 & Length & 0.29 & 4.11 & 0.59 & 0.28 & 0.25 & 0.49 & 0.35 & 0.75 \\  \hline
\multirow{2}{*}{\begin{tabular}[c]{@{}c@{}}$s=50$,  $ p=400$,\\      $n=200$\end{tabular}} & CP & 44.5\% & 43.5\% & 77.5\% & 70.0\% & 80.5\% & 92.0\% & 83.5\% & 97.0\% \\
 & Length & 0.89 & 3.11 & 1.09 & 0.92 & 0.71 & 1.22 & 1.05 & 1.24 \\ 
\multirow{2}{*}{\begin{tabular}[c]{@{}c@{}}$s=50$,  $ p=800$,\\      $n=400$\end{tabular}} & CP & 32.5\% & 54.0\% & 87.0\% & 76.5\% & 91.0\% & 98.5\% & 87.0\% & 99.5\% \\
 & Length & 0.65 & 3.37 & 0.89 & 0.64 & 0.51 & 0.97 & 0.76 & 1.05 \\
\multirow{2}{*}{\begin{tabular}[c]{@{}c@{}}$s=50$,  $ p=1000$,\\      $n=800$\end{tabular}} & CP & 20.0\% & 71.5\% & 93.0\% & 72.0\% & 91.0\% & 99.0\% & 86.0\% & 100.0\% \\
 & Length & 0.48 & 3.87 & 0.74 & 0.46 & 0.37 & 0.70 & 0.55 & 0.93 \\
\multirow{2}{*}{\begin{tabular}[c]{@{}c@{}}$s=50$,  $ p=1000$,\\      $n=1600$\end{tabular}} & CP & 14.5\% & 96.5\% & 99.0\% & 59.0\% & 92.5\% & 99.0\% & 92.0\% & 100.0\% \\
 & Length & 0.35 & 4.17 & 0.61 & 0.27 & 0.25 & 0.51 & 0.40 & 0.77 \\
 \bottomrule
\end{longtable}

\begin{longtable}{cccccccccc}
\caption{Mean Absolute Bias ($\text{MAB}^\dag$), Median Absolute Bias ($\text{MAB}^\ddag$), standard error (std), root-mean-squared error (RMSE) for estimating $\tau$ under the \textit{Dense PS \& Sparse OR} setting from the main text.}
\label{tab: sparse_OR_measure}\\
\toprule
\textbf{Setting} & \textbf{measure} & \textbf{g-formula} & \textbf{IPW} & \textbf{AIPW} & \textbf{TMLE} & \textbf{ARB} & \textbf{hdCBPS} & \textbf{RCAL} & \textbf{DCal} \\
\midrule
\endfirsthead
\multicolumn{10}{c}%
{{\bfseries Table \thetable\ continued from previous page}} \\
\midrule
\endhead
\multirow{4}{*}{\begin{tabular}[c]{@{}c@{}}$s=10$,   \\     $ p=400$,\\     $ n=200$\end{tabular}} & $\text{MAB}^\dag$ & 0.47 & 1.03 & 0.30 & 0.20 & 0.18 & 0.19 & 0.29 & 0.17 \\
 & $\text{MAB}^\ddag$ & 0.48 & 1.05 & 0.27 & 0.17 & 0.15 & 0.15 & 0.27 & 0.14 \\
 & std & 0.14 & 0.56 & 0.21 & 0.25 & 0.18 & 0.32 & 0.20 & 0.26 \\
 & RMSE & 0.51 & 1.08 & 0.35 & 0.26 & 0.22 & 0.25 & 0.34 & 0.21 \\
 \hline
\multirow{4}{*}{\begin{tabular}[c]{@{}c@{}}$s=10$,   \\     $ p=800$,\\     $ n=400$\end{tabular}} & $\text{MAB}^\dag$ & 0.35 & 0.73 & 0.18 & 0.13 & 0.11 & 0.13 & 0.18 & 0.13 \\
 & $\text{MAB}^\ddag$ & 0.35 & 0.75 & 0.17 & 0.11 & 0.08 & 0.10 & 0.15 & 0.09 \\
 & std & 0.11 & 0.69 & 0.19 & 0.17 & 0.13 & 0.24 & 0.14 & 0.23 \\
 & RMSE & 0.37 & 0.82 & 0.22 & 0.16 & 0.14 & 0.17 & 0.21 & 0.18 \\
\hline \multirow{4}{*}{\begin{tabular}[c]{@{}c@{}}$s=10$,   \\     $ p=1000$,\\     $ n=800$\end{tabular}} & $\text{MAB}^\dag$ & 0.28 & 0.52 & 0.12 & 0.12 & 0.07 & 0.09 & 0.11 & 0.10 \\
 & $\text{MAB}^\ddag$ & 0.28 & 0.49 & 0.11 & 0.11 & 0.06 & 0.07 & 0.10 & 0.09 \\
 & std & 0.08 & 0.84 & 0.16 & 0.13 & 0.09 & 0.16 & 0.11 & 0.21 \\
 & RMSE & 0.29 & 0.62 & 0.15 & 0.14 & 0.08 & 0.11 & 0.13 & 0.13 \\
\hline \multirow{4}{*}{\begin{tabular}[c]{@{}c@{}}$s=10$,   \\     $ p=1000$,\\     $ n=1600$\end{tabular}} & $\text{MAB}^\dag$ & 0.24 & 0.42 & 0.09 & 0.10 & 0.05 & 0.06 & 0.07 & 0.07 \\
 & $\text{MAB}^\ddag$ & 0.24 & 0.33 & 0.08 & 0.09 & 0.04 & 0.05 & 0.06 & 0.06 \\
 & std & 0.06 & 0.89 & 0.14 & 0.07 & 0.06 & 0.12 & 0.08 & 0.18 \\
 & RMSE & 0.25 & 0.53 & 0.12 & 0.12 & 0.06 & 0.08 & 0.09 & 0.09 \\
\hline \multirow{4}{*}{\begin{tabular}[c]{@{}c@{}}$s=20$,   \\     $ p=400$,\\     $ n=200$\end{tabular}} & $\text{MAB}^\dag$ & 0.47 & 1.24 & 0.31 & 0.22 & 0.16 & 0.21 & 0.29 & 0.16 \\
 & $\text{MAB}^\ddag$ & 0.46 & 1.27 & 0.28 & 0.20 & 0.14 & 0.17 & 0.27 & 0.13 \\
 & std & 0.19 & 0.72 & 0.25 & 0.25 & 0.18 & 0.34 & 0.23 & 0.30 \\
 & RMSE & 0.51 & 1.30 & 0.36 & 0.28 & 0.21 & 0.27 & 0.34 & 0.21 \\
\hline \multirow{4}{*}{\begin{tabular}[c]{@{}c@{}}$s=20$,   \\     $ p=800$,\\     $ n=400$\end{tabular}} & $\text{MAB}^\dag$ & 0.40 & 0.90 & 0.20 & 0.16 & 0.12 & 0.13 & 0.21 & 0.13 \\
 & $\text{MAB}^\ddag$ & 0.39 & 0.97 & 0.19 & 0.14 & 0.10 & 0.10 & 0.20 & 0.11 \\
 & std & 0.14 & 0.83 & 0.21 & 0.18 & 0.13 & 0.25 & 0.17 & 0.25 \\
 & RMSE & 0.42 & 1.00 & 0.24 & 0.20 & 0.15 & 0.17 & 0.25 & 0.16 \\
\hline \multirow{4}{*}{\begin{tabular}[c]{@{}c@{}}$s=20$,   \\     $ p=1000$,\\     $ n=800$\end{tabular}} & $\text{MAB}^\dag$ & 0.30 & 0.65 & 0.13 & 0.12 & 0.07 & 0.09 & 0.14 & 0.10 \\
 & $\text{MAB}^\ddag$ & 0.29 & 0.63 & 0.12 & 0.11 & 0.06 & 0.06 & 0.13 & 0.08 \\
 & std & 0.10 & 0.98 & 0.18 & 0.12 & 0.09 & 0.17 & 0.12 & 0.23 \\
 & RMSE & 0.31 & 0.77 & 0.16 & 0.15 & 0.08 & 0.11 & 0.16 & 0.14 \\
\hline \multirow{4}{*}{\begin{tabular}[c]{@{}c@{}}$s=20$,   \\     $ p=1000$,\\     $ n=1600$\end{tabular}} & $\text{MAB}^\dag$ & 0.24 & 0.46 & 0.09 & 0.11 & 0.05 & 0.06 & 0.08 & 0.08 \\
 & $\text{MAB}^\ddag$ & 0.24 & 0.35 & 0.07 & 0.10 & 0.04 & 0.05 & 0.07 & 0.07 \\
 & std & 0.07 & 1.05 & 0.15 & 0.07 & 0.06 & 0.13 & 0.09 & 0.19 \\
 & RMSE & 0.24 & 0.58 & 0.12 & 0.12 & 0.06 & 0.08 & 0.10 & 0.10 \\
\hline \multirow{4}{*}{\begin{tabular}[c]{@{}c@{}}$s=50$,   \\     $ p=400$,\\     $ n=200$\end{tabular}} & $\text{MAB}^\dag$ & 0.50 & 1.39 & 0.38 & 0.28 & 0.22 & 0.28 & 0.34 & 0.20 \\
 & $\text{MAB}^\ddag$ & 0.48 & 1.43 & 0.36 & 0.28 & 0.20 & 0.25 & 0.34 & 0.17 \\
 & std & 0.23 & 0.79 & 0.28 & 0.24 & 0.18 & 0.31 & 0.27 & 0.32 \\
 & RMSE & 0.54 & 1.45 & 0.42 & 0.33 & 0.27 & 0.34 & 0.39 & 0.25 \\
\hline \multirow{4}{*}{\begin{tabular}[c]{@{}c@{}}$s=50$,   \\     $ p=800$,\\     $ n=400$\end{tabular}} & $\text{MAB}^\dag$ & 0.39 & 1.04 & 0.22 & 0.18 & 0.13 & 0.14 & 0.23 & 0.12 \\
 & $\text{MAB}^\ddag$ & 0.39 & 1.07 & 0.21 & 0.15 & 0.12 & 0.12 & 0.22 & 0.11 \\
 & std & 0.17 & 0.86 & 0.23 & 0.16 & 0.13 & 0.25 & 0.19 & 0.27 \\
 & RMSE & 0.41 & 1.13 & 0.26 & 0.22 & 0.15 & 0.18 & 0.27 & 0.16 \\
\hline \multirow{4}{*}{\begin{tabular}[c]{@{}c@{}}$s=50$,   \\     $ p=1000$,\\     $ n=800$\end{tabular}} & $\text{MAB}^\dag$ & 0.32 & 0.75 & 0.14 & 0.14 & 0.08 & 0.09 & 0.17 & 0.11 \\
 & $\text{MAB}^\ddag$ & 0.31 & 0.72 & 0.13 & 0.13 & 0.06 & 0.08 & 0.15 & 0.09 \\
 & std & 0.12 & 0.99 & 0.19 & 0.12 & 0.09 & 0.18 & 0.14 & 0.24 \\
 & RMSE & 0.33 & 0.88 & 0.17 & 0.17 & 0.10 & 0.12 & 0.19 & 0.14 \\
\hline \multirow{4}{*}{\begin{tabular}[c]{@{}c@{}}$s=50$,   \\     $ p=1000$,\\     $ n=1600$\end{tabular}} & $\text{MAB}^\dag$ & 0.24 & 0.44 & 0.09 & 0.11 & 0.05 & 0.08 & 0.09 & 0.08 \\
 & $\text{MAB}^\ddag$ & 0.24 & 0.38 & 0.08 & 0.10 & 0.04 & 0.07 & 0.09 & 0.06 \\
 & std & 0.09 & 1.06 & 0.15 & 0.07 & 0.06 & 0.13 & 0.10 & 0.20 \\
 & RMSE & 0.24 & 0.55 & 0.12 & 0.13 & 0.06 & 0.10 & 0.11 & 0.09 \\
 \bottomrule
\end{longtable}

\subsubsection{Supplementary simulation settings}
\label{app:sparse or supp}

Here, we modify the simulation setting in Section \ref{sec:sparse or} by only changing $\gamma_{j} \propto j^{-0.5}\xi$  where $\xi \sim \mathrm{Uniform}(\{1,-1\}) $ with $\Vert \gamma \Vert_{2} \equiv 1$, while keeping other parameters the same as in the main text. This modification alleviates the near violation of the positivity assumption, as demonstrated in Figure \ref{fig: hist_propensity_sparse_OR}, showcasing the histogram of the true PS's for the case $n = 800$, $p = 1000$, and $s = 10$.

\begin{figure}[htpb]
    \centering    
    \includegraphics[page=1,scale=0.6]{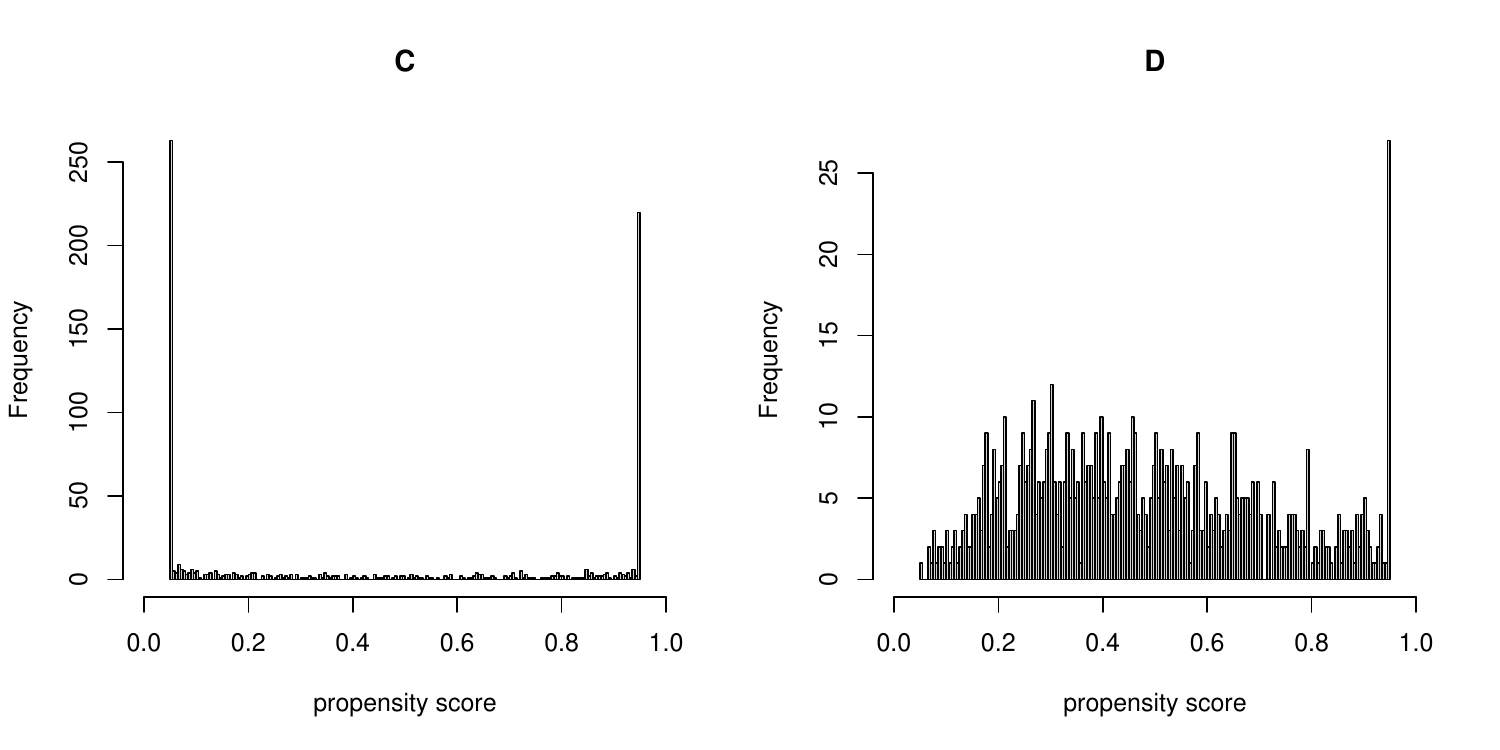}
    \caption{Histograms of the propensity score under the \textit{Dense PS \& Sparse OR} setting. The left figure (A) represents the setting described in the main text, while the right figure (B) represents the setting described in Appendix \ref{app:sparse or}. Here, we chose a scenario with $n=800, p=1000 \text{ and } s=10$.}
    \label{fig: hist_propensity_sparse_OR}
\end{figure}

The corresponding simulation results are displayed in Figure \ref{fig: simu_sparse_OR_supp}, Table \ref{tab: sparse_OR_measure_CI_supp}, and Table \ref{tab: sparse_OR_measure_supp}. Root-N scaled absolute estimation errors are given in Figure \ref{fig: simu_sparse_OR_supp}. Tables \ref{tab: sparse_OR_measure_CI_supp} and \ref{tab: sparse_OR_measure_supp} below summarize the coverage probabilities
and the lengths of 95\% confidence intervals, Mean/Median Absolute Biases, standard errors, and Root-
Mean-Errors (RMSE) of different estimators for this setting. The performance of $\hat{\tau}_{\text{RCAL}}$, $\hat{\tau}_{\text{hdCBPS}}$, and $\hat{\tau}_{\text{DCal}}$ are quite similar

\begin{figure}[htpb]
    \centering
    \includegraphics[width = \textwidth, page=1]{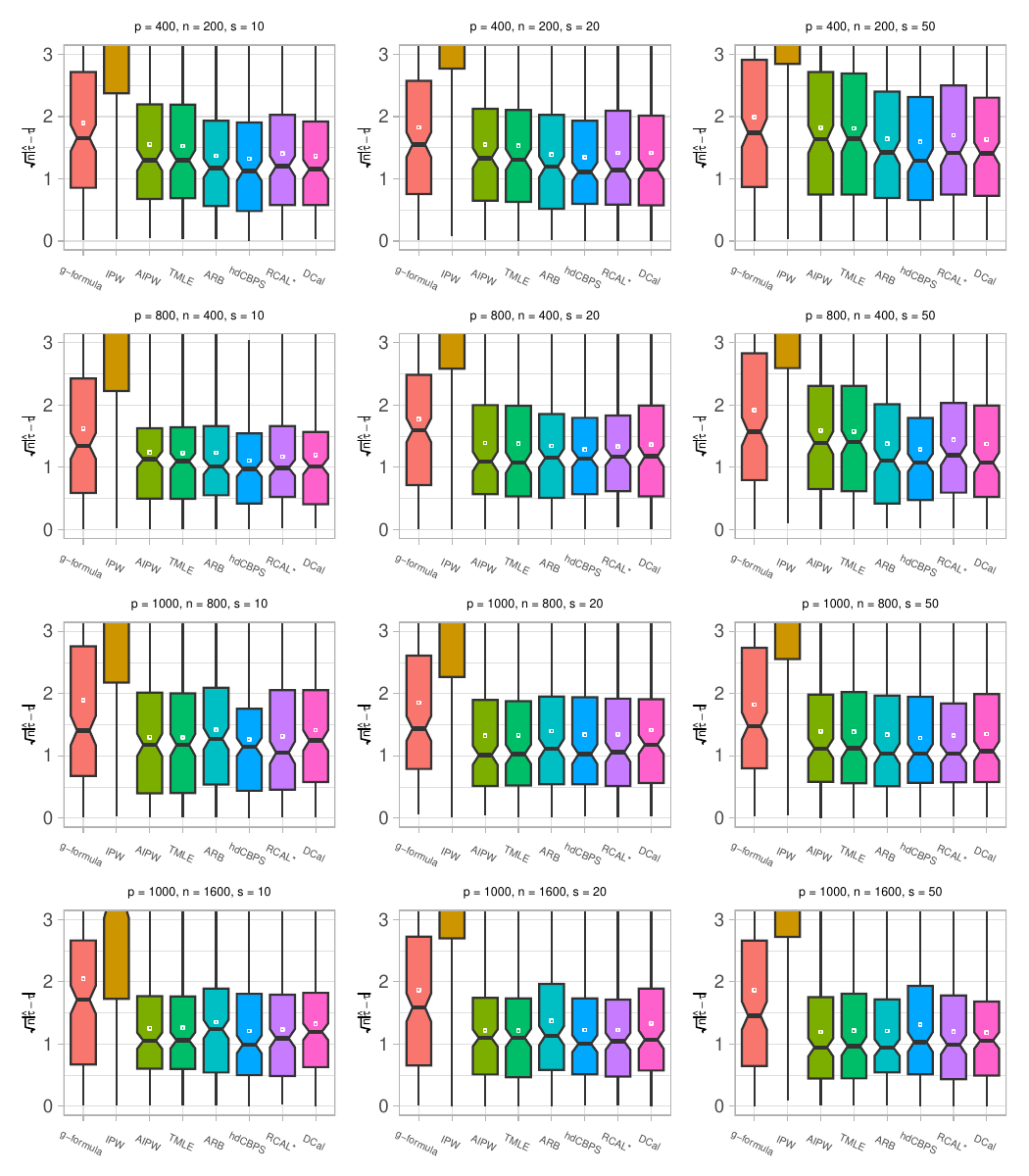}
    \caption{Boxplots of the  $\sqrt{n}$-scaled  estimation error $\sqrt{n}|\hat{\tau}-\tau|$ under different sample size $n$, dimension $p$ and sparsity level $s$ under the  \textit{Dense PS \& Sparse OR} setting. The white dots correspond to means.}
    \label{fig: simu_sparse_OR_supp}
\end{figure}

\begin{longtable}{@{}cccccccccc@{}}
\caption{Coverage probabilities of 95\% confidence intervals (Coverage), and length of 95\% confidence intervals (CI length) for estimating $\tau$ under the \textit{Dense PS \& Sparse OR} setting.}
\label{tab: sparse_OR_measure_CI_supp}\\
\toprule
\textbf{Setting} & \textbf{measure} & \textbf{g-formula} & \textbf{IPW} & \textbf{AIPW} & \textbf{TMLE} & \textbf{ARB} & \textbf{hdCBPS} & $\textbf{RCAL}^\star$ & \textbf{DCal} \\
\midrule
\endfirsthead
\multicolumn{10}{c}%
{{\bfseries Table \thetable\ continued from previous page}} \\
\toprule
\endhead
\multirow{2}{*}{\begin{tabular}[c]{@{}c@{}}$s=10$,  $ p=400$,\\      $n=200$\end{tabular}} & CP & 92.5\% & 99.5\% & 99.5\% & 92.0\% & 97.5\% & 100.0\% & 99.0\% & 100.0\% \\
 & Length & 0.61 & 2.49 & 0.84 & 0.45 & 0.52 & 0.79 & 0.82 & 0.89 \\
\multirow{2}{*}{\begin{tabular}[c]{@{}c@{}}$s=10$,  $ p=800$,\\      $n=400$\end{tabular}} & CP & 96.0\% & 100.0\% & 100.0\% & 95.5\% & 98.0\% & 100.0\% & 100.0\% & 100.0\% \\
 & Length & 0.45 & 1.79 & 0.60 & 0.31 & 0.35 & 0.56 & 0.59 & 0.63 \\
\multirow{2}{*}{\begin{tabular}[c]{@{}c@{}}$s=10$,  $ p=1000$,\\      $n=800$\end{tabular}} & CP & 92.5\% & 100.0\% & 100.0\% & 96.0\% & 94.5\% & 100.0\% & 100.0\% & 100.0\% \\
 & Length & 0.33 & 1.32 & 0.43 & 0.22 & 0.24 & 0.40 & 0.42 & 0.45 \\
\multirow{2}{*}{\begin{tabular}[c]{@{}c@{}}$s=10$,  $ p=1000$,\\      $n=1600$\end{tabular}} & CP & 90.0\% & 100.0\% & 100.0\% & 95.5\% & 95.5\% & 100.0\% & 100.0\% & 100.0\% \\
 & Length & 0.24 & 0.97 & 0.30 & 0.15 & 0.17 & 0.29 & 0.30 & 0.31 \\ \hline
\multirow{2}{*}{\begin{tabular}[c]{@{}c@{}}$s=20$,  $ p=400$,\\      $n=200$\end{tabular}} & CP & 99.0\% & 100.0\% & 99.5\% & 87.0\% & 93.5\% & 100.0\% & 99.5\% & 100.0\% \\
 & Length & 0.79 & 2.93 & 0.99 & 0.43 & 0.52 & 0.93 & 0.97 & 1.03 \\
\multirow{2}{*}{\begin{tabular}[c]{@{}c@{}}$s=20$,  $ p=800$,\\      $n=400$\end{tabular}} & CP & 98.5\% & 99.5\% & 100.0\% & 91.0\% & 96.5\% & 100.0\% & 100.0\% & 100.0\% \\
 & Length & 0.57 & 2.09 & 0.70 & 0.30 & 0.35 & 0.66 & 0.70 & 0.73 \\
\multirow{2}{*}{\begin{tabular}[c]{@{}c@{}}$s=20$,  $ p=1000$,\\      $n=800$\end{tabular}} & CP & 97.5\% & 99.5\% & 100.0\% & 92.0\% & 93.0\% & 100.0\% & 100.0\% & 100.0\% \\
 & Length & 0.42 & 1.51 & 0.50 & 0.22 & 0.24 & 0.47 & 0.50 & 0.52 \\
\multirow{2}{*}{\begin{tabular}[c]{@{}c@{}}$s=20$,  $ p=1000$,\\      $n=1600$\end{tabular}} & CP & 98.5\% & 100.0\% & 100.0\% & 94.5\% & 94.5\% & 100.0\% & 100.0\% & 100.0\% \\
 & Length & 0.30 & 1.12 & 0.36 & 0.15 & 0.17 & 0.34 & 0.35 & 0.36 \\ \hline
\multirow{2}{*}{\begin{tabular}[c]{@{}c@{}}$s=50$,  $ p=400$,\\      $n=200$\end{tabular}} & CP & 99.0\% & 99.5\% & 100.0\% & 77.5\% & 92.5\% & 99.5\% & 100.0\% & 100.0\% \\
 & Length & 0.95 & 3.30 & 1.11 & 0.39 & 0.51 & 1.04 & 1.10 & 1.15 \\
\multirow{2}{*}{\begin{tabular}[c]{@{}c@{}}$s=50$,  $ p=800$,\\      $n=400$\end{tabular}} & CP & 99.0\% & 99.5\% & 100.0\% & 86.5\% & 93.5\% & 100.0\% & 100.0\% & 100.0\% \\
 & Length & 0.68 & 2.33 & 0.80 & 0.29 & 0.35 & 0.75 & 0.79 & 0.83 \\
\multirow{2}{*}{\begin{tabular}[c]{@{}c@{}}$s=50$,  $ p=1000$,\\      $n=800$\end{tabular}} & CP & 100.0\% & 100.0\% & 100.0\% & 93.0\% & 95.5\% & 100.0\% & 100.0\% & 100.0\% \\
 & Length & 0.49 & 1.67 & 0.57 & 0.21 & 0.25 & 0.54 & 0.57 & 0.59 \\
\multirow{2}{*}{\begin{tabular}[c]{@{}c@{}}$s=50$,  $ p=1000$,\\      $n=1600$\end{tabular}} & CP & 98.0\% & 99.5\% & 100.0\% & 95.0\% & 96.5\% & 100.0\% & 100.0\% & 100.0\% \\
 & Length & 0.35 & 1.23 & 0.41 & 0.15 & 0.17 & 0.39 & 0.40 & 0.41 \\ \bottomrule
\end{longtable}

\begin{longtable}{cccccccccc}
\caption{Mean Absolute Bias ($\text{MAB}^\dag$), Median Absolute Bias ($\text{MAB}^\ddag$), standard error (std), root-mean-squared error (RMSE) for estimating $\tau$ under the \textit{Dense PS \& Sparse OR} setting.}
\label{tab: sparse_OR_measure_supp}\\
\toprule
\textbf{Setting} & \textbf{measure} & \textbf{g-formula} & \textbf{IPW} & \textbf{AIPW} & \textbf{TMLE} & \textbf{ARB} & \textbf{hdCBPS} & $\textbf{RCAL}^\star$ & \textbf{DCal} \\
\midrule
\endfirsthead
\multicolumn{10}{c}%
{{\bfseries Table \thetable\ continued from previous page}} \\
\midrule
\endhead
\multirow{4}{*}{\begin{tabular}[c]{@{}c@{}}$s=10$,   \\   $   p=400$,\\ $     n=200$\end{tabular}} & $\text{MAB}^\dag$ & 0.12 & 0.34 & 0.09 & 0.09 & 0.08 & 0.08 & 0.09 & 0.08 \\
 &$\text{MAB}^\ddag$ & 0.12 & 0.34 & 0.09 & 0.09 & 0.08 & 0.08 & 0.09 & 0.08 \\
 & std & 0.16 & 0.64 & 0.22 & 0.12 & 0.13 & 0.20 & 0.21 & 0.23 \\
 & RMSE & 0.17 & 0.45 & 0.14 & 0.13 & 0.12 & 0.12 & 0.13 & 0.12 \\
\hline \multirow{4}{*}{\begin{tabular}[c]{@{}c@{}}$s=10$,   \\   $   p=800$,\\ $     n=400$\end{tabular}} & $\text{MAB}^\dag$ & 0.07 & 0.23 & 0.06 & 0.05 & 0.05 & 0.05 & 0.05 & 0.05 \\
 &$\text{MAB}^\ddag$ & 0.07 & 0.23 & 0.06 & 0.05 & 0.05 & 0.05 & 0.05 & 0.05 \\
 & std & 0.12 & 0.46 & 0.15 & 0.08 & 0.09 & 0.14 & 0.15 & 0.16 \\
 & RMSE & 0.10 & 0.28 & 0.08 & 0.08 & 0.08 & 0.07 & 0.07 & 0.08 \\
\hline \multirow{4}{*}{\begin{tabular}[c]{@{}c@{}}$s=10$,   \\   $   p=1000$,\\ $     n=800$\end{tabular}} & $\text{MAB}^\dag$ & 0.05 & 0.14 & 0.04 & 0.04 & 0.04 & 0.04 & 0.04 & 0.04 \\
 &$\text{MAB}^\ddag$ & 0.05 & 0.14 & 0.04 & 0.04 & 0.04 & 0.04 & 0.04 & 0.04 \\
 & std & 0.08 & 0.34 & 0.11 & 0.06 & 0.06 & 0.10 & 0.11 & 0.11 \\
 & RMSE & 0.09 & 0.18 & 0.06 & 0.06 & 0.06 & 0.06 & 0.06 & 0.06 \\
\hline \multirow{4}{*}{\begin{tabular}[c]{@{}c@{}}$s=10$,   \\   $   p=1000$,\\ $     n=1600$\end{tabular}} & $\text{MAB}^\dag$ & 0.04 & 0.08 & 0.03 & 0.03 & 0.03 & 0.02 & 0.03 & 0.03 \\
 &$\text{MAB}^\ddag$ & 0.04 & 0.08 & 0.03 & 0.03 & 0.03 & 0.02 & 0.03 & 0.03 \\
 & std & 0.06 & 0.25 & 0.08 & 0.04 & 0.04 & 0.07 & 0.08 & 0.08 \\
 & RMSE & 0.07 & 0.11 & 0.04 & 0.04 & 0.04 & 0.04 & 0.04 & 0.04 \\
\hline \multirow{4}{*}{\begin{tabular}[c]{@{}c@{}}$s=20$,   \\   $   p=400$,\\ $     n=200$\end{tabular}} & $\text{MAB}^\dag$ & 0.11 & 0.36 & 0.09 & 0.09 & 0.08 & 0.08 & 0.08 & 0.08 \\
 &$\text{MAB}^\ddag$ & 0.11 & 0.36 & 0.09 & 0.09 & 0.08 & 0.08 & 0.08 & 0.08 \\
 & std & 0.20 & 0.75 & 0.25 & 0.11 & 0.13 & 0.24 & 0.25 & 0.26 \\
 & RMSE & 0.16 & 0.52 & 0.14 & 0.14 & 0.13 & 0.12 & 0.13 & 0.13 \\
\hline \multirow{4}{*}{\begin{tabular}[c]{@{}c@{}}$s=20$,   \\   $   p=800$,\\ $     n=400$\end{tabular}} & $\text{MAB}^\dag$ & 0.08 & 0.29 & 0.05 & 0.05 & 0.06 & 0.06 & 0.06 & 0.06 \\
 &$\text{MAB}^\ddag$ & 0.08 & 0.29 & 0.05 & 0.05 & 0.06 & 0.06 & 0.06 & 0.06 \\
 & std & 0.15 & 0.53 & 0.18 & 0.08 & 0.09 & 0.17 & 0.18 & 0.19 \\
 & RMSE & 0.11 & 0.37 & 0.09 & 0.09 & 0.08 & 0.08 & 0.08 & 0.09 \\
\hline \multirow{4}{*}{\begin{tabular}[c]{@{}c@{}}$s=20$,   \\   $   p=1000$,\\ $     n=800$\end{tabular}} & $\text{MAB}^\dag$ & 0.05 & 0.18 & 0.04 & 0.04 & 0.04 & 0.04 & 0.04 & 0.04 \\
 &$\text{MAB}^\ddag$ & 0.05 & 0.18 & 0.04 & 0.04 & 0.04 & 0.04 & 0.04 & 0.04 \\
 & std & 0.11 & 0.38 & 0.13 & 0.06 & 0.06 & 0.12 & 0.13 & 0.13 \\
 & RMSE & 0.08 & 0.24 & 0.06 & 0.06 & 0.06 & 0.06 & 0.06 & 0.06 \\
\hline \multirow{4}{*}{\begin{tabular}[c]{@{}c@{}}$s=20$,   \\   $   p=1000$,\\ $     n=1600$\end{tabular}} & $\text{MAB}^\dag$ & 0.04 & 0.11 & 0.03 & 0.03 & 0.03 & 0.03 & 0.03 & 0.03 \\
 &$\text{MAB}^\ddag$ & 0.04 & 0.11 & 0.03 & 0.03 & 0.03 & 0.03 & 0.03 & 0.03 \\
 & std & 0.08 & 0.29 & 0.09 & 0.04 & 0.04 & 0.09 & 0.09 & 0.09 \\
 & RMSE & 0.06 & 0.15 & 0.04 & 0.04 & 0.04 & 0.04 & 0.04 & 0.04 \\
\hline \multirow{4}{*}{\begin{tabular}[c]{@{}c@{}}$s=50$,   \\   $   p=400$,\\ $     n=200$\end{tabular}} & $\text{MAB}^\dag$ & 0.12 & 0.42 & 0.12 & 0.12 & 0.10 & 0.09 & 0.10 & 0.10 \\
 &$\text{MAB}^\ddag$ & 0.12 & 0.42 & 0.12 & 0.12 & 0.10 & 0.09 & 0.10 & 0.10 \\
 & std & 0.24 & 0.84 & 0.28 & 0.10 & 0.13 & 0.27 & 0.28 & 0.29 \\
 & RMSE & 0.18 & 0.56 & 0.16 & 0.16 & 0.14 & 0.14 & 0.15 & 0.14 \\
\hline \multirow{4}{*}{\begin{tabular}[c]{@{}c@{}}$s=50$,   \\   $   p=800$,\\ $     n=400$\end{tabular}} & $\text{MAB}^\dag$ & 0.08 & 0.28 & 0.07 & 0.07 & 0.06 & 0.05 & 0.06 & 0.05 \\
 &$\text{MAB}^\ddag$ & 0.08 & 0.28 & 0.07 & 0.07 & 0.06 & 0.05 & 0.06 & 0.05 \\
 & std & 0.17 & 0.60 & 0.20 & 0.07 & 0.09 & 0.19 & 0.20 & 0.21 \\
 & RMSE & 0.12 & 0.41 & 0.10 & 0.10 & 0.09 & 0.08 & 0.09 & 0.09 \\
\hline \multirow{4}{*}{\begin{tabular}[c]{@{}c@{}}$s=50$,   \\   $   p=1000$,\\ $     n=800$\end{tabular}} & $\text{MAB}^\dag$ & 0.05 & 0.18 & 0.04 & 0.04 & 0.04 & 0.04 & 0.04 & 0.04 \\
 &$\text{MAB}^\ddag$ & 0.05 & 0.18 & 0.04 & 0.04 & 0.04 & 0.04 & 0.04 & 0.04 \\
 & std & 0.13 & 0.43 & 0.15 & 0.05 & 0.06 & 0.14 & 0.14 & 0.15 \\
 & RMSE & 0.08 & 0.27 & 0.06 & 0.06 & 0.06 & 0.06 & 0.06 & 0.06 \\
\hline \multirow{4}{*}{\begin{tabular}[c]{@{}c@{}}$s=50$,   \\   $   p=1000$,\\ $     n=1600$\end{tabular}} & $\text{MAB}^\dag$ & 0.04 & 0.13 & 0.02 & 0.02 & 0.02 & 0.03 & 0.02 & 0.03 \\
 &$\text{MAB}^\ddag$ & 0.04 & 0.13 & 0.02 & 0.02 & 0.02 & 0.03 & 0.02 & 0.03 \\
 & std & 0.09 & 0.31 & 0.10 & 0.04 & 0.04 & 0.10 & 0.10 & 0.11 \\
 & RMSE & 0.06 & 0.19 & 0.04 & 0.04 & 0.04 & 0.04 & 0.04 & 0.04 \\
 \bottomrule
\end{longtable}

\end{document}